\documentclass[10pt]{article}
\usepackage{amsfonts}
\usepackage[leqno]{amsmath}
\usepackage{graphicx}
\usepackage{latexsym}
\usepackage{amsmath,amsfonts,amssymb,amsthm,mathrsfs,euscript,makeidx,color}
\usepackage{enumerate}
\allowdisplaybreaks[1]
\usepackage{multirow}

\def\sqr#1#2{{\vcenter{\vbox{\hrule height.#2pt
				\hbox{\vrule width.#2pt height#1pt \kern#1pt \vrule width.#2pt}
				\hrule height.#2pt}}}}
\def\signed #1{{\unskip\nobreak\hfil\penalty50
		\hskip2em\hbox{}\nobreak\hfil#1
		\parfillskip=0pt \finalhyphendemerits=0 \par}}
\def\endpf{\signed {$\sqr69$}}

\def\3n{\negthinspace \negthinspace \negthinspace }
\def\2n{\negthinspace \negthinspace }
\def\1n{\negthinspace }
\def\bel{\begin{equation}\label}
	\def\eel{\end{equation}}

\def\dbE{\mathbb{E}}
\def\dbF{\mathbb{F}}

\def\dbH{\mathbb{H}}

\def\dbP{\mathbb{P}}

\def\dbR{\mathbb{R}}
\def\dbS{\mathbb{S}}
\def\dbT{\mathbb{T}}

\def\dbX{\mathbb{X}}
\def\dbY{\mathbb{Y}}
\def\dbZ{\mathbb{Z}}

\def\sD{\mathscr{D}}

\def\sR{\mathscr{R}}
\def\sS{\mathscr{S}}

\def\sU{\mathscr{U}}

\def\fU{\mathfrak{U}}

\def\={\buildrel \triangle \over =}

\def\ds{\displaystyle}

\def\ns{\noalign{\ss}}
\def\a{\alpha}
\def\b{\beta }
\def\g{\gamma}
\def\d{\delta}
\def\e{\varepsilon}
\def\z{\zeta}

\def\l{\lambda}
\def\m{\mu}
\def\n{\nu}
\def\si{\sigma}
\def\t{\tau}
\def\f{\varphi}
\def\th{\theta}
\def\o{\omega}

\def\i{\infty}
\def\G{\Gamma}

\def\Th{\Theta}
\def\L{\Lambda}
\def\Si{\Sigma}
\def\F{\Phi}
\def\O{\Omega}

\def\BTh{{\bf\Theta}}
\def\cD{{\cal D}}

\def\cF{{\cal F}}

\def\cJ{{\cal J}}
\def\cK{{\cal K}}
\def\cL{{\cal L}}

\def\cQ{{\cal Q}}
\def\cR{{\cal R}}
\def\cS{{\cal S}}

\def\Bq{{\bf q}}
\def\Br{{\bf r}}

\def\BBxi{\boldsymbol\xi}

\def\G{\Gamma}

\def\Th{\Theta}
\def\L{\Lambda}
\def\Si{\Sigma}
\def\F{\varPhi}
\def\O{\Omega}

\def\BTh{\boldsymbol\Theta}

\def\no{\noindent}

\def\ss{\smallskip}
\def\ms{\medskip}
\def\bs{\bigskip}
\def\q{\quad}
\def\qq{\qquad}
\def\hb{\hbox}

\def\lan{{\langle}}
\def\ran{{\rangle}}

\def\esssup{\mathop{\rm esssup}}

\def\h{\widehat}
\def\wt{\widetilde}

\def\cd{\cdot}
\def\cds{\cdots}

\def\ae{\hbox{\rm a.e.{ }}}
\def\as{\hbox{\rm a.s.}}

\def\les{\leqslant}
\def\ges{\geqslant}

\def\({\Big (}
\def\){\Big )}
\def\[{\Big[}
\def\]{\Big]}
\def\lan{\langle}
\def\ran{\rangle}

\def\bde{\begin{definition}\label}
	\def\ede{\end{definition}}

	\def\bel{\begin{equation}\label}
		\def\ee{\end{equation}}
	\def\bt{\begin{theorem}\label}
		\def\et{\end{theorem}}
	\def\bc{\begin{corollary}\label}
		\def\ec{\end{corollary}}
	\def\bl{\begin{lemma}\label}
		\def\el{\end{lemma}}
	\def\bp{\begin{proposition}\label}
		\def\ep{\end{proposition}}
	\def\bas{\begin{assumption}}
		\def\eas{\end{assumption}}
	\def\br{\begin{remark}\label}
		\def\er{\end{remark}}
	\def\ba{\begin{array}}
		\def\ea{\end{array}}
	\def\ed{\end{document}}

\def\rf{\eqref}
\newcommand{\ad}{&\!\!\!\displaystyle}

\def\square#1{\vbox{\hrule\hbox{\vrule height#1%
			\kern#1\vrule}\hrule}}
\def\rectangle#1#2{\vbox{\hrule\hbox{\vrule height#1%
			\kern#2\vrule}\hrule}}

\font\tenbb=msbm10 \font\sevenbb=msbm7 \font\fivebb=msbm5

\newfam\bbfam
\scriptscriptfont\bbfam=\fivebb \textfont\bbfam=\tenbb
\scriptfont\bbfam=\sevenbb

\newtheorem{theorem}{Theorem}[section]
\newtheorem{corollary}[theorem]{Corollary}

\newtheorem{lemma}[theorem]{Lemma}
\newtheorem{proposition}[theorem]{Proposition}

\theoremstyle{definition}
\newtheorem{definition}[theorem]{Definition}

\newtheorem{remark}[theorem]{Remark}

\makeatletter

\@addtoreset{equation}{section}
\makeatother

\usepackage{xcolor}
\makeatletter

\input pdf-trans
\newbox\qbox
\def\usecolor#1{\csname\string\color@#1\endcsname\space}
\newcommand\bordercolor[1]{\colsplit{1}{#1}}
\newcommand\fillcolor[1]{\colsplit{0}{#1}}
\newcommand\outline[1]{\leavevmode%
	\def\maltext{#1}%
	\setbox\qbox=\hbox{\maltext}%
	\boxgs{Q q 2 Tr \thickness\space w \fillcol\space \bordercol\space}{}%
	\copy\qbox%
}
\makeatother
\newcommand\colsplit[2]{\colorlet{tmpcolor}{#2}\edef\tmp{\usecolor{tmpcolor}}%
	\def\tmpB{}\expandafter\colsplithelp\tmp\relax%
	\ifnum0=#1\relax\edef\fillcol{\tmpB}\else\edef\bordercol{\tmpC}\fi}
\def\colsplithelp#1#2 #3\relax{%
	\edef\tmpB{\tmpB#1#2 }%
	\ifnum `#1>`9\relax\def\tmpC{#3}\else\colsplithelp#3\relax\fi
}
\bordercolor{black}
\fillcolor{white}
\def\thickness{.3}

\def\hTh{\outline{$\Theta$}}

\begin{document}

	\title{\bf Linear-Quadratic Optimal Control Problem for Mean-Field Stochastic Differential Equations with a Type of Random Coefficients\thanks{This work is supported in part by NSF Grant DMS-2305475.}}
	
	\author{Hongwei Mei\footnote{ Department of Mathematics and Statistics, Texas Tech University, Lubbock, TX 79409, USA. Email: {\tt hongwei.mei@ttu.edu}}~~~~~Qingmeng Wei\footnote{School of Mathematics and Statistics, Northeast Normal University, Changchun 130024, China. Email: {\tt weiqm100@nenu.edu.cn}. This author is supported by  NSF of Jilin Province for Outstanding Young Talents (No.
			20230101365JC) and  NSF of P.R. China (No. 11971099)} ~~~\text{ and }~~~
		Jiongmin Yong\footnote{Department of Mathematics, University of Central Florida, Orlando, FL 32816, USA. Email: {\tt jiongmin.yong@ucf.edu.}} }
	
	\maketitle

\centerline{({\it This paper is dedicated to Professor G.~George Yin on the occasion of his 70th birthday})}

\bs
	
	\no\bf Abstract: \rm Motivated by linear-quadratic optimal control problems (LQ problems, for short) for mean-field stochastic differential equations (SDEs, for short) with the coefficients containing regime switching governed by a Markov chain, we consider an LQ problem for an SDE with the coefficients being adapted to a filtration independent of the Brownian motion driving the control system. Classical approach of completing the square is applied to the current problem and obvious shortcomings are indicated. Open-loop and closed-loop solvability are introduced and characterized.

	\ms
	
	\no\bf Keywords: \rm  Linear-quadratic problem, forward-backward stochastic differential equations, differential Riccati equation.
	
	\ms
	
	\no\bf AMS Mathematics Subject Classification. \rm 93E20, 49N10, 60F17.

	\section{Introduction}

	Let $(\O,\cF,\dbF,\dbP)$ be a complete filtered probability space on which a one-dimensional standard Brownian motion $W(\cd)$ is defined. We begin with the following controlled linear mean-field stochastic differential equation (MF-SDE, for short):
	\bel{SDE0}\left\{\2n\ba{ll}
	\ds dX(t)\1n=\1n\(A(t,\a(t))X(t)\1n+\1n\bar A(t,\a(t))\dbE[X(t)]\1n+\1n B(t,\a(t))u(t)\1n+\1n\bar B(t,\a(t))\dbE[u(t)]\1n+\1n b(t)\)dt\\
	\ns\ds\qq\q+\1n\(\1n C(t,\a(t))X(t)\1n+\1n\bar C(t,\a(t))\dbE[X(t)]\1n+\1n D(t,\a(t))u(t)\1n+\1n\bar D(t,\a(t))\dbE[u(t)]\1n+\1n\si(t)\1n\)dW(t),~t\1n\ges\1n s,\\
	\ns\ds X(s)=\xi,\ea\right.\ee
	In the above, $u(\cd)$ is a control process valued in $\dbR^m$, and $\a(\cd)$ is a Markov chain, independent of $W(\cd)$, valued in some finite subset of $\dbR$, determining the regime switching for the system, $X(\cd)$ is the corresponding state process. A typical form of Markov chain $\a(\cd)$ is the solution to the following SDE:
	\bel{SDE12}d\a(t)=\int_\dbR\m(t,\a(t-),\th)N(dt,d\th),\q t\ges s;\qq\a(s)=\a_0.\ee
	Here, $N(dt,d\th)$ is a Poisson random measure on $\dbR$ with intensity measure
	$\dbE[N(dt,d\th)]$, and $\m(\cd)$ is some given map. Let $\dbF^N\equiv\{\cF_t^N\}_{t\ges0}$ be the filtration of the Poisson process associated with $N(dt,d\th)$. Then the coefficients $A(\cd\,,\a(\cd))$, etc. are all assumed to be $\dbF^N$-adapted. Under proper conditions, for any control $u(\cd)$ (taken from certain class), and initial triple $(s,\xi,\a_0)$, state equation \rf{SDE0} admits a unique solution $X(\cd)\equiv X(\cd\,;s,x,\a_0,u(\cd))$. To measure the performance of the control, one could introduce a quadratic cost functional. Then a corresponding linear-quadratic (LQ, for short) optimal control problem for such a system can be formulated in a finite time horizon.
	
	\ms
	
	There is a vast number of papers dealing with various LQ problems since the seminal works of Bellman--Glicksberg--Gross \cite{Bellman-Glicksberg-Gross-1958}, Kalman \cite{Kalman-1960} and Letov \cite{Letov-1960} appeared around 1960. Let us briefly mention a very small portion of the relevant works on LQ problems. Standard LQ theory for ordinary differential equations can be found in Lee--Markus \cite{Lee-Markus-1967}, Anderson--Moore \cite{Anderson-Moore-1971}, Willems \cite{Willems-1971}, Wonham \cite{Wonham-1979}, and so on. Study of stochastic LQ problems began with the works of Kushner \cite{Kushner-1962} and Wonham \cite{Wonham-1968} in the 1960s. See also McLane \cite{McLane-1971}, Davis \cite{Davis-1977}, Bensoussan \cite{Bensoussan-1981}, and so on, for classical stochastic LQ theory. In 1998, Chen--Li--Zhou \cite{Chen-Li-Zhou-1998} found that for stochastic LQ problems, the weighting matrices in the cost functional could be indefinite to some extent. See Yong--Zhou \cite{Yong-Zhou-1999} and Sun--Yong \cite{Sun-Yong-2020a} for some comprehensive presentations along this line. Stochastic LQ problems with mean-field was studied by Yong \cite{Yong-2013} in 2013. See Huang--Li--Yong \cite{Huang-Li-Yong-2015}, Li--Sun--Yong \cite{Li-Sun-Yong-2016}, Sun \cite{Sun-2017}, Wei--Yong--Yu \cite{Wei-Yong-Yu-2019}, Sun--Yong \cite{Sun-Yong-2020b}, Li--Shi--Yong \cite{Li-Shi-Yong-2021} for some follow-up works. On the other hand, egordic LQ control problem was studied by Mei--Wei--Yong \cite{Mei-Wei-Yong-2021} and LQ problem with regime switching in finite time horizon was studied by Zhang--Li \cite{Zhang-Li-2018}, Wen--Li--Xiong--Zhang \cite{Wen-Li-2022 }.
	
	\ms
	
	We now introduce the general framework of this paper, which is strictly more general than the above, and it could cover some other interesting situations. Let $(\O,\cF,\dbF,\dbP)$ and $W(\cd)$ be as above.  Also, let a square integrable c\`adl\`ag martingale $M(\cd)$ be defined. We assume that $W(\cd)$ and $M(\cd)$ are independent and their natural filtrations augmented by all the $\dbP$-null sets in $\cF$ are denoted by $\dbF^W\equiv\{\cF^W_t\}_{t\ges0}$ and $\dbF^M\equiv\{\cF^M_t\}_{t\ges0}$, respectively. Next, let $\dbF=\dbF^W\vee\dbF^M\equiv\{\cF^W_t\vee\cF^M_t\}_{t\ges0}$ (with $\cF^W_t\vee\cF^M_t=\si\big(\cF^W_t\cup\cF^M_t\big)$) and denote $\dbE_t^M[\,\cd\,]:=\dbE[\,\cd\,|\cF_t^M]$. Further, we write $\dbE^M[\,\cd\,]:=\dbE[\,\cd\,|\cF_\infty^M]$ with $\cF^M_\infty=\si\big(\bigcup_{t\ges0}\cF^M_t\big)$. Now, we consider the following controlled SDE:
	\bel{SDE1}\left\{\2n\ba{ll}
	\ds dX(t)=\[A(t)X(t)+\bar A(t)\dbE^M[X(t)]+B(t)u(t)+\bar B(t)\dbE^M[u(t)]+b(t)\]dt,\\
	\ns\ds\qq\qq+\[C(t)X(t)+\bar C(t)\dbE^M[X(t)]+D(t)u(t)+\bar D(t)\dbE^M[u(t)]
	+\si(t)\]dW(t),\qq t\ges s,\\
	\ns\ds X(s)=\xi,\ea\right.\ee
	where $X(\cd)$ is the state process valued in $\dbR^n$ (with initial pair $(s,\xi)$), and $u(\cd)$ is a control process valued in $\dbR^m$, both are $\dbF$-adapted. The coefficients $A(\cd),\bar A(\cd),C(\cd),\bar C(\cd)$ and $B(\cd),\bar B(\cd)$, $D(\cd),\bar D(\cd)$ are either $\dbR^{n\times n}$ or $\dbR^{n\times m}$-valued, $\dbF^M$-adapted stochastic processes; the nonhomogeneous terms $b(\cd),\si(\cd)$ are $\dbR^n$-valued, $\dbF$-adapted stochastic processes. We see that \rf{SDE1} is a linear SDE with special type mean-field terms and special type random coefficients. By Lemma \ref{infty=t} (in the appendix), we will see that
	$$\dbE^M[X(t)]=\dbE^M_t[X(t)],\qq\dbE^M[u(t)]=\dbE_t^M[u(t)],\qq t\ges s.$$
	Thus, $\dbE^M$ in \rf{SDE1} can also be replaced by $\dbE^M_t$.  Note that the mean-field terms $\dbE^M[X(t)]$ and $\dbE^M[u(t)]$ are $\cF^M_t$-measurable for each $t$. In the case that $M$ is a constant process, \rf{SDE1} reduces to the classical linear mean-field SDE with deterministic coefficients (see \cite{Yong-2013,Sun-2017}). On the other hand, if $M(t)=N(t)-\lambda t$ where $N(\cd)$ is a Poisson process with intensity $\lambda$, our framework recovers the usual regime switching case. It could also be the case that $M(\cd)$ is another Brownian motion that is independent of $W(\cd)$ and/or a pair of a Brownian motion and a Poisson process which are independent of $W(\cd)$, and so on. From now on, by  $\lan M\ran$, we denote the quadratic variation process of $M$. Moreover, we assume that the natural filtration $\dbF^M$ of $M$ (augmented by all the $\dbP$-null sets) is complete, quasi-left continuous, right-continuous.
	
	\ms
	
	In the case that the {\it nonhomogeneous terms} $b(\cd)$ and $\si(\cd)$ are zero, the control system is said to be {\it homogeneous}, and such a system is denoted by $[A(\cd),\bar A(\cd),C(\cd),\bar C(\cd);B(\cd),\bar B(\cd),D(\cd),\bar D(\cd)]$. The state process of the homogeneous system is denoted by $X^0(\cd)\equiv X^0(\cd\,;s,\xi,u(\cd))$. Thus,
	\bel{SDE1-h}\left\{\2n\ba{ll}
	\ds dX^0(t)=\(A(t)X^0(t)+\bar A(t)\dbE^M[X^0(t)]+B(t)u(t)+\bar B(t)\dbE^M[u(t)]\)dt\\
	\ns\ds\qq\qq\qq+\(C(t)X^0(t)+\bar C(t)\dbE^M[X^0(t)]+D(t)u(t)+\bar D(t)\dbE^M[u(t)]\)dW(t),\qq t\ges s,\\
	\ns\ds X^0(s)=\xi.\ea\right.\ee
	Next, for given Euclidean space $\dbH$ (which could be $\dbR^n$, $\dbR^{n\times m}$, etc.), we let
	$$\ba{ll}
	\ds L^2_{\cF_s}(\O;\dbH)=\Big\{\xi:\O\to\dbH\bigm|\xi\hb{ is $\cF_s$-measurable, }\dbE|\xi|^2<\infty\Big\},\\
	\ds L^2_\dbF(s,T;\dbH)\1n\equiv\1n\Big\{\f\1n:\1n[s,T]\1n\times\1n\O\1n\to\1n
	\dbH\bigm|\f(\cd)\hb{ is $\dbF$-progressively measurable, }\dbE\2n\int_s^T\3n|\f(t)|^2dt\1n<\1n\i\Big\},\\
	\ds L^2_{\dbF_-}(s,T;\dbH)\1n\equiv\1n\Big\{\f\in L^2_\dbF(s,T;\dbH)\bigm|\f(\cd)\hb{ is $\dbF$-predictable}\Big\},\ea$$
	and
	$$\ba{ll}
	\ds\sD=\Big\{(s,\xi)\bigm|s\in[0,\i),~\xi\in L_{\cF_s}^2(\O;\dbR^n)\Big\},\\
	\ns\ds\sU[s,T]=L^2_\dbF(s,T;\dbR^m),\q\hb{if }0<T<\infty;\qq\sU[s,\i)=L^2_\dbF(s,\i;\dbR^m).\ea$$
	Under proper conditions, for any given {\it initial pair} $(s,\xi)\in\sD$ with $0\les s<T<\i$, and any control $u(\cd)\in\sU[s,T]$, the state equation \rf{SDE1} admits a unique strong solution $X(\cd)\equiv X(\cd\,;s,\xi,u(\cd))$ on $[s,T]$. To measure the performance of the control $u(\cd)$ over $[s,T]$, we may introduce the following quadratic cost functional:
	\bel{cost[0,T]}\ba{ll}
	\ns\ds J(s,\xi;u(\cd))=\dbE\int_s^Tf\big(t,X(t), \dbE^M[X(t)],u(t),\dbE^M[u(t)]\big)dt+\dbE\big[ F\big(X(T),\dbE^M[X(T)]\big)\big],\ea\ee
	where
	\bel{f-F}\left\{\2n\ba{ll}
	\ds f(t ,x,\bar x,u,\bar u)=\frac12\[ \lan Q(t)x,x\ran+2\lan S(t)x,u\ran+\lan R(t)u,u\ran +\lan\bar Q(t)\bar x,\bar x\ran\1n+2\lan\bar S(t) \bar x ,\bar u\ran\1n+\1n\lan\bar R(t)\bar u,\bar u\ran\\ [1mm]
	\ns\ds\qq\qq\qq\qq+2\lan q(t),x\ran+2\lan\bar q(t),\bar x\ran+2\lan r(t),u\ran+2\lan\bar r(t),\bar u\ran\], \\
	\ns\ds F(x,\bar x)=\frac12\[\lan Gx,x\ran+\lan\bar G\bar x,\bar x\ran+2\lan g,x\ran+2\lan\bar g,\bar x\ran\] ,\ea\right.\ee
	with $Q(\cd),\bar Q(\cd),S(\cd),\bar S(\cd),R(\cd),\bar R(\cd)$ being some $\dbF^M$-adapted matrix-valued processes of suitable sizes, $q(\cd)$, $r(\cd)$ being $\dbF$-adapted (vector-valued) process, $\bar q(\cd)$, $\bar r(\cd)$ being $\dbF^M$-adapted (vector-valued) processes, $G,\bar G$ being $\cF^M_T$-measurable random matrices, $g$ being $\cF_T$-measurable random vector, and $\bar g$ being $\cF^M_T$-measurable random vector. The two terms on the right-hand side of \rf{cost[0,T]} are called the {\it expected running cost} and the {\it expected terminal cost}, respectively. Correspondingly, $f(\cd)$ defined in \rf{f-F} is called the {\it running cost rate function}. In the case that $q(\cd),\bar q(\cd),r(\cd),\bar r(\cd),g,\bar g$ are all zero, we denote the cost functional by $J^0(s,\xi;u(\cd))$. Hence,
	\bel{cost[0,T]-h}\ba{ll}
	\ns\ds J^0(s,\xi;u(\cd))=\dbE\int_s^Tf^0\big(t,X^0(t),\dbE^M[X^0(t)],u(t),\dbE^M[u(t)]\big)dt
	+\dbE\big[ F^0\big(X^0(T),\dbE^M[X^0(T)]\big)\big],\ea\ee
	where
	\bel{f-F0}\left\{\2n\ba{ll}
	\ds f^0(t,x,\bar x,u,\bar u)=\frac12\[\lan Q(t)x,x\ran+2\lan S(t)x,u\ran+\lan R(t)u,u\ran+\lan\bar Q(t)\bar x,\bar x\ran\1n+2\lan\bar S(t)\bar x,\bar u\ran\1n+\1n\lan\bar R(t)\bar u,\bar u\ran\]\\ [1mm]
	\ns\ds F^0(x,\bar x)=\frac12\[\lan Gx,x\ran+\lan\bar G\bar x,\bar x\ran\].\ea\right.\ee
	Our finite time-horizon optimal control problem can be stated as follows.
	
	\ms
	
	\no	\bf Problem (MF-LQ). \rm For any given $(s,\xi)\in\sD$ having $s\in[0,T)$, with the state equation \rf{SDE1} and cost functional \rf{cost[0,T]}, find a control $\bar u(\cd)\in\sU[s,T]$ such that
	\bel{J=inf-open}J(s,\xi;\bar u(\cd))=\inf_{u(\cd)\in\sU[s,T]}J(s,\xi;u(\cd))\equiv V(s,\xi).\ee
	We call any $\bar u(\cd)$ satisfies the above as an {\it open-loop optimal control}, and corresponding $\bar X(\cd)$, $(\bar X(\cd),\bar u(\cd))$ as {\it open-loop optimal state process} and {\it open-loop optimal pair} of Problem (MF-LQ)$^0$, respectively. We also call $V(\cd\,,\cd)$ the {\it value function} of Problem (MF-LQ).	
	
	\ms
	
	The problem for the homogeneous state equation \rf{SDE1-h} and purely quadratic cost functional \rf{cost[0,T]-h} is denoted by Problem (MF-LQ)$^0$.	
	
	\ms

	Now, we highlight the main contributions and clues of the current paper.
	
	\ms
	
	(i) The first natural approach to LQ problem is to use the method of completing the squares. Applying such a method, one could obtain sufficient conditions for the existence of optimal control. It is not clear whether the imposed conditions are necessary for Problem (MF-LQ) to have an optimal control. This motivates our study below.
	
	\ms
	
	(ii) In order to explore the above, mainly inspired by Sun--Yong \cite{Sun-Yong-2020a, Sun-Yong-2020b}, we introduce the open-loop solvability of Problem (MF-LQ), and obtained its characterization in terms of the solvability of the optimality system which is a coupled mean-field forward-backward stochastic differential equation (MF-FBSDE, for short). From this, the open-loop optimal control can be written in terms of the predictable solution of this MF-FBSDE. Note that due to the appearance of the martingale $M(\cd)$, the corresponding MF-FBSDE is different from the classical FBSDE (\cite{Ma-Yong-1999}).
	
	\ms
	
	(iii) It is clear that the above obtained representation of the open-loop optimal control is anticipating: the future information of the open-loop optimal state process is used. To remedy that, we try to find a closed-loop representation of the open-loop optimal control by decoupling the MF-FBSDE using the idea of invariant imbedding (\cite{Bellman-Kalaba-Wing-1960}) (or the Four-Step Scheme, introduced in \cite{Ma-Protter-Yong-1994}). This naturally derives  a differential Riccati equation which is a backward stochastic differential equation (BSDE, for short) in general. As long as such an equation has a predictable solution, the open-loop optimal control admits a closed-loop representation, which is non-anticipating.
	
	\ms
	
	(iv) This logically suggests us introduce the closed-loop solvability of the LQ problem, which in turn can be characterized by the regular solvability of the differential Riccati equation that is derived earlier. Note that, in the current case, the martingale $M(\cd)$ appears not only in the FBSDE, but also in the differential Riccati equation. This brings some new features into the study. It is worthy to point out that the method used in \cite{Li-Sun-Yong-2016} seems to be difficult to extend here. Therefore some new methods will be created.
	
	\ms
	
	The rest of the paper is organized as follows.	Some preliminary results are collected in
	Section 2. In Section 3, we follow the idea of completing the squares for Problem (MF-LQ), naturally deriving the Riccati equation. In Section 4, the open-loop solvability of Problem (MF-LQ)$^0$ is introduced and characterized by a mean-field FBSDE. Close-loop representation of open-loop optimal control is established by decoupling the FBSDE in Section 5. In Section 6, closed-loop solvability is introduced and characterized by the regular solvability of the Riccati equation. In Section 7, it was proved that the strongly regular solvability of the Reccati equation is equivalent to the uniform convexity of the cost functional. This gives the natural sufficient condition under which Problem (MF-LQ) is closed-loop solvable. Finally, some simple conclusion remarks are put in  the last section.
	
	\section{Preliminary}

	First of all, besides the spaces $L^2_{\cF_s}(\O;\dbH)$ and $L^2_\dbF(s,T;\dbH)$ defined in the previous section, we introduce some more spaces: For $0\les s<T$, define
	$$\ba{ll}
	\ns\ds L_{\cF_s^M}^2(\O;\dbH):=\Big\{\xi\in L^2_{\cF_s}(\O;\dbH)\bigm|\xi\hb{ is $\cF_s^M$-measurable}\Big\},\\
	\ns\ds L^2_{\dbF^M}(s,T;\dbH)=\Big\{\f(\cd)\in L^2_\dbF(s,T;\dbH)  \bigm|  \f(\cd) \hb{ is $\dbF^M$-progressively measurable}\Big\},\\
	\ns\ds L^2_{\dbF^M_-}(s,T;\dbH)=\Big\{\f(\cd)\in L^2_\dbF(s,T;\dbH)  \bigm|  \f(\cd) \hb{ is $\dbF^M$-predictable}\Big\},\\
	\ns\ds L^\infty_{\dbF}(s,T;\dbH)=\Big\{\f:[s,T]\times\O\to\dbH\bigm|\f(\cd)\hb{ is $\dbF$-progressively measurable, }\esssup_{t\in[s,T]}\|\f(t,\cd)\|_\i <\infty\Big\},\\
	\ns\ds L^\infty_{\dbF^M}(s,T;\dbH)=\Big\{\f(\cd)\in L^\infty_{\dbF}(s,T;\dbH)\bigm| \f(\cd)\hb{ is $\dbF^M$-progressively measurable}\Big\}. \ea$$
	Let $\dbS^k$ ($\dbS^k_+$) be the set of all the $k\times k$ symmetric (positive-definite) matrices. For a symmetric positive semi-definite matrix $\Psi\in\dbS^{k}$, denoted by $\Psi\ges 0$, we denote the pesudo-inverse of $\Psi$ by $\Psi^\dag$, and the range of $\Psi$ by $\sR(\Psi)$.
	In the following, we may misuse a little the notation $\lan\cd\,,\cd\ran$ that represents inner products in different spaces which can be identified from the contexts.
	
	\ms
	
	Next, for any $s\in [0,T)$, we know that $L^2_{\cF_s}(\O;\dbH)$ is a Hilbert space under the following inner product
	$$\dbE\lan\xi,\eta\ran\equiv\int_\O\lan\xi(\o),\eta(\o)\ran d\dbP(\o),\qq\xi,\eta\in L^2_{\cF_s}(\O;\dbH).$$
	The space $L_{\cF_s^M}^2(\O;\dbH)$ with the same inner product as above, is a closed subspace of $L^2_{\cF_s}(\O;\dbH)$, and its orthogonal complement in $L^2_{\cF_s}(\O;\dbH)$ is given by
	$$L^2_{\cF_s^M}(\O;\dbH)^\perp:=\Big\{\xi\in L^2_{\cF_s}(\O;\dbH)\bigm|\dbE\lan\xi,\eta\ran=0,\q\forall
	\eta\in L^2_{\cF_s^M}(\O;\dbH)\Big\}.$$
	Let the orthogonal projection from $L^2_{\cF_s}(\O;\dbH)$ onto $L^2_{\cF_s^M}(\O;\dbH)$ be denoted by $\Pi_s$, and let $\Pi_s^\perp=I-\Pi_s$. Then, we get an orthogonal decomposition of $L^2_{\cF_s}(\O;\dbH)$ as follows,
	\bel{spacede}L^2_{\cF_s}(\O;\dbH)=L^2_{\cF_s^M}(\O;\dbH)^\perp\oplus L^2_{\cF_s^M}(\O;\dbH).\ee
	Now, for any $\xi\in L^2_{\cF_s}(\O;\dbH)$ and $\eta\in L^2_{\cF^M_s}(\O;\dbH)$,
	$$\dbE\lan\xi-\dbE^M_s[\xi],\eta\ran=\dbE\lan\xi,\eta\ran-\dbE\lan \dbE^M_s[\xi],\eta\ran=0.$$
	Thus, $\xi-\dbE^M_s[\xi]\in L^2_{\cF^M_s}(\O;\dbH)^\perp$. Consequently,
	$$\Pi_s[\xi]-\dbE^M_s[\xi]=\(\xi-\dbE^M_s[\xi]\)-\(\xi-\Pi_s[\xi]\)\in L^2_{\cF^M_s}
	(\O;\dbH)\cap L^2_{\cF^M_s}(\O;\dbH)^\perp=\{0\}.$$
	Hence,
	\bel{Pi_s=E_s}\Pi_s[\xi]=\dbE^M_s[\xi],\qq\forall\xi\in L^2_{\cF_s}(\O;\dbH).\ee
	On the other hand, thanks to Lemma \ref{infty=t} (from the Appendix), we have
	$\dbE^M_s[\,\cd\,]=\dbE^M[\,\cd\,]\equiv\dbE[\,\cd\,|\cF^M_T]$. Therefore,
	\bel{Pi_s=E}\Pi_s[\xi]=\dbE_s^M[\xi]=\dbE^M[\xi],\qq\Pi_s^\perp[\xi]=\xi-\dbE_s^M[\xi]
	=\xi-\dbE^M[\xi],\qq\forall\xi\in L^2_{\cF_s}(\O;\dbH).\ee
	We should note that any non-zero element in $L^2_{\cF_s^M}(\O;\dbH)^\perp$ is not
	$\cF_s^M$-measurable. Also, for any $\xi\in L^2_{\cF_s}(\O;\dbH)$, $\dbE^M_s[\xi]$ is in $L^2_{\cF_s^M}(\O;\dbH)$, and it is not necessarily a deterministic vector.
	
	\ms
	
	Based on the above, for any $0\les s< T$, we further define $\Pi: L^2_{\dbF}(s,T;\dbH)\to L^2_{\dbF^M}(s,T;\dbH) $ as follows:
	\bel{Pi}\Pi[v(\cd)](t)=\Pi_t[v(t)]\equiv\dbE^M_t[v(t)]=\dbE^M[v(t)],\qq\ae t\in[s,T],~\forall v(\cd)\in L^2_\dbF(s,T;\dbH).\ee
	Note that for any $v(\cd)\in L^2(s,T;\dbH)$, $\Pi[v(\cd)]$ is defined for almost all
	$t\in[s,T]$, as a process defined on $[s,T]$.
	We now show that $\Pi$ is the orthogonal projection from $L^2_\dbF(0,T;\dbH)$ onto $L^2_{\dbF^M}(0,T;\dbH)$. In fact, first of all, if $v(\cd)=\bar v(\cd)$ in $L^2_{\dbF}(s,T;\dbH)$, we get
	$$\dbE\int_s^T\big|\Pi[v(\cd)](t)-\Pi[\bar v(\cd)](t)\big|^2dt\les\dbE\int_s^T
	\big|v(t)-\bar v(t)\big|^2dt=0,$$
	which leads to $\Pi[v(\cd)]=\Pi[\bar v(\cd)]$ in $L^2_{\dbF^M}(s,T;\dbH)$. This means that $\Pi$ is well-defined. Clearly, $\Pi^2=\Pi$ and
	$$\ba{ll}
	\ns\ds\lan\Pi[v(\cd)],\bar v(\cd)\ran=\dbE\int_s^T\lan\dbE^M[v(t)],\bar v(t)\ran dt=\dbE\int_s^T\lan\dbE^M[v(t)],\dbE^M[\bar v(t)]\ran dt\\
	\ns\ds\qq\qq\qq=\dbE\int_s^T\lan v(t),\dbE^M[\bar v(t)]\ran dt=\lan v(\cd),\Pi[\bar v(\cd)]\ran.\ea$$
	Thus, $\Pi$ is a self-adjoint idempotent, which means that $\Pi$ is an orthogonal projection from $L^2_\dbF(s,T;\dbH)$ onto $L^2_{\dbF^M}(s,T;\dbH)$. Next, we denote $\Pi^\perp:= I-\Pi$, which is the orthogonal projection from $L^2_{\dbF}(s,T;\dbH)$ onto $L^2_{\dbF^M}(s,T;\dbH)^\perp$,  where
	$$L^2_{\dbF^M}(s,T;\dbH)^\perp:=\Big\{v(\cd)\in L^2_{\dbF}(s,T;\dbH)\bigm|\dbE\int_s^T \lan v(t),\bar v(t)\ran dt=0,\q\forall\bar v(\cd)\in L^2_{\dbF^M}(s,T;\dbH)\Big\}.$$
	Also
	$$L^2_{\dbF_-^M}(s,T;\dbH)^\perp:=\Big\{v(\cd)\in L^2_{\dbF^M}(s,T;\dbH)^\perp\bigm|\f(\cd)\hb{ is $\dbF$-predictable }\Big\}=\Pi^\perp\(L^2_{\dbF_-}(s,T;\dbH)\).$$
	From the above, we know that $L^2_{\dbF^M}(s,T;\dbH)\oplus L^2_{\dbF^M}(s,T;\dbH)^\perp$
	is an orthogonal decomposition of $L^2_{\dbF}(s,T;\dbH)$, and $L^2_{\dbF^M_-}(s,T;\dbH)\oplus L^2_{\dbF^M_-}(s,T;\dbH)^\perp$
	is an orthogonal decomposition of $L^2_{\dbF_-}(s,T;\dbH)$

	\br{} \rm Note that although we have not indicated, the projections $\Pi_s$, $\Pi_s^\perp$ and $\Pi[v(\cd)]$, $\Pi^\perp[v(\cd)]$ in the above definitions  are actually depending on the dimension $k$ of the underlying space $\dbH$. For notational simplicity, we will not indicate such a dependence explicitly, which is clear from the context.
	
	\er
	
	In what follows, for  given $s\in[0,T)$, we will conventionally make the following identifications: For any Euclidean space $\dbH$,
	\bel{convention}\ba{ll}
	\ns\ds\xi=\Pi^\perp_s[\xi]+\Pi_s[\xi]=\xi_1\oplus\xi_2\equiv(\xi_1,\xi_2),\qq\forall\xi\in L^2_{\cF_s}(\O;\dbH);\\
	\ns\ds\f(\cd)=\Pi^\perp[\f(\cd)]+\Pi[\f(\cd)]=\f_1(\cd)\oplus\f_2(\cd)\equiv(\f_1(\cd),\f_2(\cd)),\qq\forall\f(\cd)\in L^2_\dbF(s,T;\dbH).\ea\ee
	Here, $\xi_1\oplus\xi_2$ stands for $\xi_1+\xi_2$ with $\xi_1$ and $\xi_2$ being mutually perpendicular; the meaning of $\f_1(\cd)\oplus\f_2(\cd)$ is similar. In the above, $\f(\cd)$ could be $X(\cd)$, $u(\cd)$, $b(\cd)$ and so on. For simplicity, we will do not distinguish $\Pi_s$ and $\Pi$ below. Further, we will denote
	$$\Pi_1=\Pi^\perp= I-\Pi,\qq\Pi_2=\Pi.$$
	Next, let us make a couple of additional observations.
	
	\ms
	
	$\bullet$ By the independence of $W(\cd)$ and $M(\cd)$, we have
	\bel{Pi W=0}\ba{ll}
	\ns\ds\Pi_1\int_0^t\f(s)dW(s)=(I-\Pi)\int_0^t\f(s)dW(s)=\int_0^t\f(s)dW(s),
	\q\forall\f(\cd)\in L^2_\dbF(0,T;\dbH),\\
	\ns\ds\Pi_2\int_0^t\f(s)dW(s)\equiv\dbE^M\int_0^t\f(s)dW(s)=0,\q\forall\f(\cd)\in
	L^2_\dbF(0,T;\dbH),\\
	\ns\ds\Pi_i\int_0^t\f(s)dM(s)=\int_0^t\Pi_i[\f(s)]dM(s),\q\forall\f(\cd)\in L^2_\dbF(0,T;\dbH),\q i=1,2.\ea\ee
	From the first two in the above, one sees that
	\bel{2.7}\int_0^\cd\f(s)dW(s)\in L^2_{\dbF^M}(0,T;\dbH)^\perp,\qq\forall\f(\cd)\in
	L^2_\dbF(0,T;\dbH).\ee
	The point in the above is that $\f(\cd)$ is arbitrary and it does not have to be
	$\dbF^M$-adapted, or in $L^2_{\dbF^M}(0,T;\dbH)^\perp$.
	
	\ms
	
	$\bullet$ For any $P(\cd)\in L^\i_{\dbF^M}(s,T;\cL(\dbH_1,\dbH_2))$ and $\f(\cd)\in L^2_\dbF(s,T;\dbH_1)$,
	$$\Pi_i\big[P(t)\f(t)\big]=P(t)\Pi_i[\f(t)],\qq i=1,2.$$
	
	\ms
	
	Before going further, let us introduce the following hypothesis which will be assumed throughout of the paper.
	
	\ms
	
	\no	{\bf(H2.1)} Let $0<T<\i$.
	
	\ms
	
	(i) The coefficients and nonhomogeneous terms of the state equation \rf{SDE1} satisfy
	$$\left\{\2n\ba{ll}
	\ds A(\cd),\bar A(\cd),C(\cd),\bar C(\cd)\in L^\infty_{\dbF^M}(0,T;\dbR^{n\times n}),
	\q B(\cd),\bar B(\cd), D(\cd),\bar D(\cd)\in L^\infty_{\dbF^M}(0,T;\dbR^{n\times m}),\\
	\ns\ds b(\cd),\si(\cd)\in L^2_{\dbF}(0,T;\dbR^n).\ea\right.$$
	
	(ii) The weighting coefficients of the cost functional \eqref{cost[0,T]} satisfy
	$$\left\{\2n\ba{ll}
	\ds Q(\cd),\bar Q(\cd)\in L^\infty_{\dbF^M}(0,T;\dbS^{n}),~ R(\cd),\bar R(\cd)\in
	L^\infty_{\dbF^M}(0,T;\dbS^{m}),~ S(\cd),\bar S(\cd)\in L^\infty_{\dbF^M}(0,T;
	\dbR^{m\times n}),\\
	\ns\ds q(\cd)\in L^2_{\dbF}(0,T;\dbR^n),~\bar q(\cd)\in L^2_{\dbF^M}(0,T;\dbR^n),~
	r(\cd)\in L^2_{\dbF}(0,T;\dbR^m),~\bar r(\cd)\in L^2_{\dbF^M}(0,T;\dbR^m),\\
	\ns\ds G,\bar G\in L^\infty_{\cF_T^M}(\O;\dbS^n),~g\in L^2_{\cF_T}(\O;\dbR^n),~
	\bar g\in L^2_{\cF_T^M}(\O;\dbR^n).\ea\right.$$
	
	\ss
	
	It is clear that under (H2.1), for any $(s,\xi)\in\sD$ with $s\in[0,T)$, and
	$u(\cd)\in\sU[s,T]$, the state equation \rf{SDE1} admits a unique solution $X(\cd)\in
	L^2_\dbF(s,T;\dbR^n)$, and the cost functional \rf{cost[0,T]} is well-defined.
	Therefore, Problem (MF-LQ) is well-formulated.
	
	\ms
	
	Now, applying the orthogonal projection $\Pi_2$ to state equation \rf{SDE1}, one has
	(noting the convention \rf{convention})
	\bel{dX_2}\ba{ll}
	\ds dX_2(t)=\[A(t)X_2(t)+\bar A(t)X_2(t)+B(t)u_2(t)+\bar B(t)u_2(t)+b_2(t)\]dt\\
	\ns\ds\qq\q\;=\[\big(A(t)+\bar A(t)\big)X_2(t)+\big(B(t)+\bar B(t)\big)u_2(t)+b_2(t)\]dt.\ea\ee
	Noting $X_1(t)=X(t)-X_2(t)$, and $u_1(t)=u(t)-u_2(t)$, etc., subtracting the above from state equation \rf{SDE1}, we obtain
	\bel{dX_1}\ba{ll}
	\ds dX_1(t)=\[A(t)X_1(t)+B(t)u_1(t)+b_1(t)\]dt\\
	\ns\ds\qq\qq+\[C(t)X_1(t)+\big(C(t)+\bar C(t)\big)X_2(t)+D(t)u_1(t)+\big(D(t)+\bar D(t)\big)u_2(t)+\si(t)\]dW(t).\ea\ee
	Hence, by letting
	\bel{A_1}\ba{ll}
	\ns\ds A_1(\cd):=A(\cd),\q A_2(\cd):=A(\cd)+\bar A(\cd),\q B_1(\cd):=B(\cd),\q B_2(\cd):=B(\cd)+\bar B(\cd),\\
	\ns\ds C_1(\cd):=C(\cd),\q C_2(\cd):=C(\cd)+\bar C(\cd),\q D_1(\cd):=D(\cd),\q D_2(\cd):=D(\cd)+\bar D(\cd),\ea\ee
	we see that state equation \rf{SDE1} is equivalent to the following system (with $t$
	being properly suppressed):
	\bel{SDE2}\left\{\2n\ba{ll}
	\ds dX_1(t)=[A_1 X_1+B_1u_1+b_1]dt+[C_1 X_1+C_2 X_2+D_1u_1+D_2u_2+\si_1+\si_2]dW(t),\q t\in[s,T],\\
	\ns\ds dX_2(t)=[ A_2X_2+B_2u_2+b_2]dt,\q t\in[s,T],\\
	\ns\ds X_1(s)=\xi_1,~X_2(s)=\xi_2, \q \mbox{ with } (\xi_1,\xi_2)=\big(\Pi_1[\xi],\Pi_2[\xi]\big).\ea\right.\eel
	For the cost functional \rf{cost[0,T]}, we observe the following
	$$\ba{ll}
	\ns\ds\dbE\[\lan Q(t)X(t),X(t)\ran\]=\dbE\[\big\lan Q(t)\big(\dbE^M[X(t)]+X(t)-\dbE^M[X(t)]\big),\dbE^M[X(t)]+X(t)-\dbE^M[X(t)]\big\ran\]\\
	\ns\ds=\dbE\[\big\lan Q(t)\Pi_1[X](t),\Pi_1[X](t)\big\ran+\big\langle Q(t)\Pi_2[X](t),\Pi_2[X](t)\big\ran\]\\
	\ns\ds=\dbE\[\big\lan Q X_1(t),X_1(t)\big\ran+\big\lan Q X_2(t), X_2(t)\big\ran\].\ea$$
	Similar calculations apply to the other terms in the cost functional, which leads   \rf{cost[0,T]} to the following:
	\bel{cost1}\ba{ll}
	\ns\ds J(s,\xi_1,\xi_2;u_1(\cd),u_2(\cd))\equiv J(s,\xi;u(\cd))\\
	\ns\ds=\frac12\sum_{i=1}^2\dbE\[\int_s^T\(\lan Q_iX_i,X_i\ran\! +\!2\lan S_iX_i, u_i\ran\!+\!\lan R_iu_i,u_i\ran+2\lan q_i,X_i\ran+2\lan r_i,u_i\ran\)dt\\
	\ns\ds\qq\qq\qq\qq+\lan G_iX_i(T), X_i(T)\ran+ 2\lan g_i,X_i(T)\ran\],\ea\ee
	with 	
	\bel{Q_1}\ba{ll}
	\ns\ds Q_1(\cd):=Q(\cd),\q Q_2(\cd):=Q(\cd)+\bar Q(\cd),\q  S_1(\cd):=S(\cd),\q S_2(\cd):=S(\cd)+\bar S(\cd),\\
	\ns\ds R_1(\cd):=R(\cd),\q R_2(\cd):=R(\cd)+\bar R(\cd),\\
	\ns\ds q_1(\cd)=\Pi_1[q(\cd)],\q q_2(\cd)=\Pi_2[q(\cd)]+\bar q(\cd),\q  r_1(\cd)=\Pi_1[r(\cd)],\q r_2(\cd)=\Pi_2[r(\cd)]+\bar r(\cd),\\
	\ns\ds G_1:=G,\q G_2:=G+\bar G,\q g_1=\Pi_1[g],\q g_2=\Pi_2[g]+\bar g.\ea \ee
	Then Problem (MF-LQ) can be equivalently formulated with the state equation \rf{SDE2} and the cost functional \rf{cost1}.
	
	\ms
	
	In the case that $b(\cd),\si(\cd),q(\cd),\bar q(\cd), r(\cd),\bar r(\cd), g, \bar g$ are all zero, \rf{SDE2} becomes (compare with \rf{SDE1-h})
	\bel{SDE4}\left\{\2n\ba{ll}
	\ds dX_1^0(t)=[A_1X_1^0+B_1u_1]dt+[C_1X_1^0+C_2X_2^0 +D_1u_1+D_2u_2]dW(t),\qq t\in[s,T],\\
	\ns\ds dX_2^0(t)=[A_2X_2^0+B_2u_2]dt,\qq t\in[s,T],\\
	\ns\ds X_1^0(s)=\xi_1,\qq X_2^0(s)=\xi_2.\ea\right.\eel
	The cost functional \rf{cost1} becomes
	\bel{cost2}\ba{ll}
	\ns\ds J^0(s,\xi_1,\xi_2;u_1(\cd),u_2(\cd))\equiv J^0(s,\xi;u(\cd))\\
	\ns\ds=\frac12\dbE\Big\{\int_s^T\[\lan QX^0,X^0\ran\! +\!2\lan SX^0,u\ran\1n+\1n\lan Ru,u\ran\]dt+\lan GX^0(T),X^0(T)\ran\Big\}\\
	\ns\ds=\frac12\sum_{i=1}^2\dbE\Big\{\int_s^T\[\lan Q_iX_i^0,X_i^0\ran\! +\!2\lan S_iX_i^0, u_i\ran\!+\!\lan R_iu_i,u_i\ran\]dt+\lan G_iX^0_i(T),X^0_i(T)\ran\Big\}.\ea\ee
	In this case, we recall, from the introduction, that the corresponding LQ problem is
	named Problem (MF-LQ)$^0$. This will play an important role later.
	
	\section{Completing the Square --- a Classical Approach.}	
	
	\ms
	
	For LQ problems, the most natural approach is the {\it completing the square}. In this section, we present such an approach for our Problem (MF-LQ).
	
	\ms
	
	For any given $(s,\xi)\in\sD$, $J(s,\xi;u(\cd))$ contains the running terms (the terms in the integral over $[s,T]$) and the terminal terms (the terms at $T$). Such a mixed form is not convenient for us to complete the squares. Therefore, our first step is to write the terminal terms into the running terms. To this end, we introduce the following backward stochastic differential equations (BSDEs, for short):
	($i=1,2$)%
	\bel{P_i}\left\{\2n\ba{ll}
	\ds dP_i(t)=\G_i(t)dt+\L_i^M(t)dM(t),\qq t\in[s,T],\\
	\ns\ds P_i(T)=G_i,\ea\right.\ee
	for some undetermined $\G_i(\cd)\in L^\infty_{\dbF^M}(s,T;\dbS^n)$. Thus, for such a case, the drift, denoted $\dot P_i$, of $P_i$ is bounded. 
		By saying $(P_i(\cd),\L_i^M(\cd))$ to be a {\it predictable solution} of \rf{P_i}, we mean that $P_i(\cd),\dot P_i(\cd)\in L^\infty_{\dbF^M}(s,T;\dbS^n)$ and $\L^M_i(\cd)\in L^2_{\dbF_-^M}(s,T;\dbS^n)$ and \rf{P_i} is satisfied in the usual sense (note that $G_i\in L^\i_{\cF_T^M}(\O;\dbS^n)$). Further we introduce the following BSDEs:
	\bel{eta_1}\left\{\2n\ba{ll}
	\ds d\eta_1(t)=\g_1(t)dt+\z_1(t)dW(t)+\z_1^M(t)dM(t),\qq t\in[s,T],\\
	\ns\ds\eta_1(T)=g_1,\ea\right.\ee
	for some undetermined $\g_1(\cd)\in L^2_{\dbF^M}(s,T;\dbR^n)^\perp$, and
	\bel{eta_2}\left\{\2n\ba{ll}
	\ds d\eta_2(t)=\g_2(t)dt+\z_2^M(t)dM(t),\qq t\in[s,T],\\
	\ns\ds\eta_2(T)=g_2,\ea\right.\ee
	for some undetermined $\g_2(\cd)\in L^2_{\dbF^M}(s,T;\dbR^n)$. Similar to the above, by saying that $(\eta_1(\cd),\z_1(\cd),\z_1^M(\cd))$ to be a predictable solution of \rf{eta_1}, we mean that
	\bel{eta_1-z_1}(\eta_1(\cd),\z_1(\cd),\z_1^M(\cd))\in L^2_{\dbF^M}(s,T;\dbR^n)^\perp\times L^2_\dbF(s,T;\dbR^n)\times L^2_{\dbF_-^M}(s,T;\dbR^n)^\perp,\ee
	and \rf{eta_1} is satisfied in the usual sense; and by saying that $(\eta_2(\cd),\z_2^M(\cd))$ to be a predictable solution of \rf{eta_2}, we mean that
	\bel{eta_2-z_2}(\eta_2(\cd),\z_2^M(\cd))\in L^2_{\dbF^M}(s,T;\dbR^n)\times L^2_{\dbF_-^M}(s,T;\dbR^n),\ee
	and \rf{eta_2} is satisfied in the usual sense. Note that due to \rf{2.7}, $\z_1(\cd)\in
	L^2_\dbF(s,T;\dbR^n)$ (not necessarily in $L^2_{\dbF^M}(s,T;\dbR^n)^\perp$). Then, by It\^o's formula (see Appendix B for details),
	\bel{Lagrange}\ba{ll}
	\ns\ds J(s,\xi_1,\xi_2;u_1(\cd),u_2(\cd))\equiv J(s,\xi;u(\cd))\\
	\ns\ds=\frac12\sum_{i=1}^2\dbE\Big\{\lan P_i(s)\xi_i,\xi_i\ran+2\lan\eta_i(s),\xi_i\ran+\int_s^T\[\lan[\G_i+\cQ_i(P_i,P_1)]X_i,X_i\ran
	+2\lan\cS_i(P_i,P_1)X_i+\Br_i,u_i\ran\\
	\ns\ds\qq\qq+\lan\cR_i(P_1)u_i,u_i\ran+2\lan\g_i+\Bq_i,X_i\ran+2\lan\eta_i,b_i\ran
	+2\lan\Pi_i[\z_1],\si_i\ran+\lan P_1\si_i,\si_i\ran\]dt\Big\},\ea\ee
	where
	\bel{cQ-cR-cS}\ba{ll}
	\ns\ds\cQ_i(\cd)\equiv\cQ_i(P_i,P_1)(\cd):=  P_i(\cd)A_i(\cd)+A_i(\cd)^\top P_i(\cd)+C_i(\cd)^\top P_1(\cd)C_i(\cd) +Q_i(\cd) ,\\
	\ns\ds\cR_i(\cd)\equiv\cR_i(P_1)(\cd):= R_i(\cd)+D_i(\cd)^\top P_1(\cd)D_i(\cd)  ,\\
	\ns\ds\cS_i(\cd)\equiv\cS_i(P_i,P_1)(\cd)=B_i(\cd)^\top P_i(\cd)+ D_i(\cd)^\top P_1(\cd) C_i(\cd)+ S_i(\cd),\q i=1,2,\ea\eel
	and
	\bel{Bq}\left\{\2n\ba{ll}
	\ds\Bq_i:=A_i^\top\eta_i+C_i^\top\Pi_i[\z_1]+P_ib_i+C_i^\top P_1\si_i+q_i,\\
	\ns\ds\Br_i:=B_i^\top\eta_i+D_i^\top\Pi_i[\z_1]+D_i^\top P_1\si_i+r_i,\ea\qq i=1,2.\right.\ee
	This finishes the first step: The cost functional \rf{Lagrange} contains only the integral terms. Such a functional is said to be of the {\it Lagrange form}. Now, we perform the second step: completing the square. Assume the following:
	\bel{cR>0}\cR_i(P_1)\ges0,\qq\cS_i(P_i,P_1)X_i+\Br_i\in\sR\big(\cR_i(P_1)\big),\qq i=1,2.\ee
	Then we have
	$$\ba{ll}
	\ns\ds2\lan\cS_iX_i+\Br_i,u_i\ran+\lan\cR_iu_i,u_i\ran\\
	\ns\ds=|\cR_i^\frac12u_i+[\cR_i^\dag]^\frac12[\cS_iX_i+\Br_i]|^2
	-\lan\cR_i^\dag[\cS_iX_i+\Br_i],\cS_i+\Br_i\ran\\
	\ns\ds=|\cR_i^\frac12u_i+[\cR_i^\dag]^\frac12[\cS_iX_i+\Br_i]|^2
	-\lan\cS_i^\top\cR_i^\dag\cS_iX_i,X_i\ran-2\lan\cS_i^\top\cR_i^\dag\Br_i,X_i\ran-\lan\cR_i^\dag
	\Br_i,\Br_i\ran.\ea$$
	Hence, \rf{Lagrange} becomes
	\bel{Lagrange2}\ba{ll}
	\ns\ds J(s,\xi;u(\cd))=\frac12\sum_{i=1}^2\dbE\Big\{\lan P_i(s)\xi_i,\xi_i\ran+2\lan\eta_i(s),\xi_i\ran\\
	\ns\ds\qq\q+\int_s^T\3n\[\lan(\G_i+\cQ_i-\cS_i^\top\cR_i^\dag\cS_i)X_i,
	X_i\ran+|\cR_i^\frac12u_i+(\cR_i^\dag)^\frac12(\cS_iX_i+\Br_i)|^2\\
	\ns\ds\qq\qq+2\lan\g_i+\Bq_i-\cS_i^\top\cR_i^\dag\Br_i,X_i\ran+2\lan\eta_i,b_i\ran
	+2\lan\Pi_i[\z_1],\si_i\ran+\lan P_1\si_i,\si_i\ran-\lan\cR_i^\dag\Br_i,\Br_i\ran\]dt\Big\}.\ea\ee
	This gives the completion of squares involving controls $u_1(\cd)$ and $u_2(\cd)$. The third step is to make a right choice of $\G_i(\cd)$ and $\g_i(\cd)$. We make the following natural choices:
	\bel{G_i}\ba{ll}
	\ns\ds\G_i=-\big[\cQ_i(P_i,P_1)-\cS_i(P_i,P_1)^\top\cR_i(P_1)^\dag\cS_i(P_i,P_1)\big]\\
	\ns\q=-\big[P_iA_i+A_i^\top P_i+C_i^\top P_1C_i+Q_i\\
	\ns\ds\qq\qq-(P_iB_1+C_i^\top P_1D_i+S_i^\top)(R_i+D_i^\top P_1D_i)^\dag(B_i^\top
	P_i+D_i^\top P_1C_i+S_i)\big],\q i=1,2,\ea\ee
	and
	\bel{g_i}\ba{ll}
	\ns\ds\g_i=-\big[\Bq_i-\cS_i(P_i,P_1)^\top\cR_i(P_1)^\dag\Br_i\big]\\
	\ns\ds\q=-\big[A_i^\top\eta_i+C_i^\top\Pi_i[\z_1]+P_ib_i+C_i^\top P_1\si_i+q_i\\
	\ns\ds\qq\q-(P_iB_i+C_i^\top P_1D_i+S_i^\top)(R_i+D_i^\top P_1D_i)^\dag(B_i^\top\eta_i+D_i^\top\Pi_i[\z_1]+D_i^\top P_1\si_i+r_i)\big].\ea\ee
	Then \rf{P_i} reads completely in details as
	\bel{Riccati1}\left\{\2n\ba{ll}
	\ds dP_i=-\big[P_iA_i+A_i^\top P_i+C_i^\top P_iC_i+Q_i\\
	\ns\ds\qq\q-(P_iB_i\1n+\1n C_i^\top P_1D_i\1n+\1n S_i^\top)(R_i\1n+\1n D_i^\top P_1D_i)^{-1}(B_i^\top P_i
	\1n+\1n D_i^\top P_1C_i\1n+\1nS_i)\big]dt\1n+\1n\L_i^MdM,\q t\1n\in\1n[s,T],\\
	\ns\ds P_i(T)=G_i,\ea\right.\ee
	which is called the {\it backward stochastic differential Riccati equation} (BSDRE, for short), \rf{eta_1}--\rf{eta_2} read
	\bel{BSDE-eta_1}\left\{\2n\ba{ll}
	\ds d\eta_1=-\big[A_1^\top\eta_1+C_1^\top\Pi_1[\z_1]+P_1b_1
	+C_1^\top P_1\si_1+q_1\\
	\ns\ds\qq\q-(P_1B_1+C_1^\top P_1D_1+S_1^\top)(R_1+D_1^\top P_1D_1)^\dag(B_1^\top\eta_1+D_1^\top\Pi_1[\z_1]+D_1^\top P_1\si_1+r_1)\big]dt\\
	\ns\ds\qq\qq\qq+\z_1dW+\z_1^MdM,\qq t\in[s,T],\\
	\ds d\eta_2=-\big[A_2^\top\eta_2+C_2^\top\Pi_2[\z_1]+P_2b_2+C_2^\top P_1\si_2+q_2\\
	\ns\ds\qq\q-(P_2B_2+C_2^\top P_1D_2+S_2^\top)(R_2+D_2^\top P_1D_2)^\dag(B_2^\top\eta_2+D_2^\top\Pi_2[\z_1]+D_2^\top P_1\si_2+r_2)\big]dt\\
	\ns\ds\qq\qq\qq+\z_2^MdM,\qq t\in[s,T],\\
	\ns\ds\eta_1(T)=g_1,\qq\eta_2(T)=g_2,\ea\right.\ee
	Note that the BSDE for $\eta_1$ has a $dW$ term, but the BSDE for $\eta_2$ does not have that term. The cost functional \rf{Lagrange2} becomes
	\bel{Lagrange3}\ba{ll}
	\ns\ds J(s,\xi;u(\cd))=\frac12\sum_{i=1}^2\dbE\Big\{\lan P_i(s)\xi_i,\xi_i\ran+2\lan\eta_i(s),\xi_i\ran+\2n\int_s^T\3n\[\cR_i(P_1)^\frac12u_i\1n+\1n[\cR_i(P_1)^\dag]^\frac12
	[\cS_i(P_i,P_1)X_i\1n+\1n\Br_i]|^2\1n\\
	\ns\ds\qq\qq\qq\qq+2\lan\eta_i,b_i\ran+2\lan\Pi_i[\z_1],\si_i\ran+\lan P_1\si_i,\si_i\ran-\lan\cR_i(P_1)^\dag\Br_i,\Br_i\ran\]dt\Big\}.\ea\ee
	Now, we come to Step 4, to determine the optimal controls. From the above, we see that by taking
	\bel{bar u}\ba{ll}
	\ns\ds\bar u_i=-\cR_i(P_1)^\dag[\cS_i(P_i,P_1)\bar X_i+\Br_i]+[I-\cR_i(P_1)^\dag\cR_i(P_1)]\m_i\\
	\ns\ds\q=-(R_i+D_i^\top P_1D_i)^{-1}\((B_i^\top P_i+D_i^\top P_1C_i
	+S_i)\bar X_i+B_i^\top\eta_i+D_i^\top\Pi_i[\z_1]+D_i^\top P_1\si_i+r_i\)\\
	\ns\ds\qq+[I-\cR_i(P_1)^\dag\cR_i(P_1)]\m_i,\ea\ee
	with $(\bar X_1(\cd),\bar X_2(\cd))$ being the corresponding state, one has, from \rf{Lagrange3}, that
	\bel{Lagrange4}\ba{ll}
	\ns\ds J(s,\xi;u(\cd))\ges J(s,\xi;\bar u(\cd))=\inf_{u(\cd)\in\sU[s,T]}J(s,\xi;u(\cd))\\
	\ns\ds=\frac12\sum_{i=1}^2\dbE\Big\{\lan P_i(s)\xi_i,\xi_i\ran\1n+\1n2\lan\eta_i(s),\xi_i\ran\1n+\2n\int_s^T\3n\[\lan\eta_i,b_i\ran
	+2\lan\Pi_i[\z_1],\si_i\ran+\lan P_1\si_i,\si_i\ran-\lan\cR_i(P_1)^\dag\Br_i.\Br_i\ran\]dt\Big\}.\ea\ee
	Or, equivalently,
	\bel{V}\ba{ll}
	\ns\ds V(s,\xi)=\frac12\sum_{i=1}^2\dbE\Big\{\lan P_i(s)\xi_i,\xi_i\ran+2\lan\eta_1(s),\xi_i\ran+\int_s^T\[\lan\eta_i,b_i\ran
	+2\lan\Pi_i[\z_1],\si_i\ran+\lan P_1\si_i,\si_i\ran\\
	\ns\ds\qq\qq\qq-\big|[\cR_i(P_1)^\dag]^\frac12(B_i^\top\eta_i+D_1^\top\Pi_i[\z_1]+D_i^\top P_i\si_i+r_i)\big|^2\]dt\Big\}.\ea\ee
	The above finishes the classical approach to Problem (MF-LQ). We can state the above result as follows.

	\bt{classical} \sl Let {\rm (H2.1)} hold. Let BSDRE \rf{Riccati1} admit a predictable solution $(P_i(\cd),\L_i^M(\cd))$, satisfying
	\bel{cR>0*}\ba{ll}
	\ns\ds\sR\(B_i^\top P_i+D_i^\top P_1C_i+S_i\)\subseteq\sR\(R_i+D_i^\top P_1D_i\),\\
	\ns\ds R_i+D_i^\top P_1D_i\ges0.\ea\ee
	Let BSDEs \rf{BSDE-eta_1} admit a predictable solution $(\eta_1(\cd),\z_1(\cd),\z_1^M(\cd),\eta_2(\cd),\z_2^M(\cd))$ satisfying
	\bel{B_1eta_1}\left\{\2n\ba{ll}
	\ns\ds B_i^\top\eta_i+D_i^\top\Pi_i[\z_1]+D_i^\top P_1\si_i+r_i\in\sR(\cR_i(P_1)),\\
	%
	%
	\ns\ds\cR_1(P_1)^\dag(B_1^\top\eta_1+D_1^\top\Pi_1[\z_1]+D_1^\top P_1\si_1+r_1)\in L^2_{\dbF^M}(s,T;\dbR^m)^\perp,\\
	\ns\ds\cR_2(P_1)^\dag(B_2^\top\eta_2+D_2^\top\Pi_2[\z_1]+D_2^\top P_1 \si_2+r_2) \in L^2_{\dbF^M}(s,T;\dbR^m).
	\ea\right.\ee
	Then Problem (MF-LQ) admits an optimal control $\bar u(\cd)=\bar u_1(\cd)+\bar u_2(\cd)$ given by \rf{bar u} with $\bar X(\cd)=\bar X_1(\cd)+\bar X_2(\cd)$ being the corresponding state process, and the value function $V(\cd\,,\cd)$ is given by \rf{V}.
	
	\et

	The above result seems pretty good, and it gives a sufficient condition for Problem (MF-LQ) to have an optimal control. However, there are at least two major shortcomings:
	
	\ms
	
	$\bullet$ How the usual Pontryagin type maximum principle is related to Problem (MF-LQ)?
	
	\ms
	
	$\bullet$ It is not clear if the imposed conditions (the solvability of BSDREs \rf{Riccati1}, and BSDEs \rf{BSDE-eta_1} with properties \rf{cR>0*}--\rf{B_1eta_1}) are necessary for Problem (MF-LQ) to have an optimal control.
	
	\ms

	In the rest of the paper, we are going to look at Problem (MF-LQ) from a different angle, try to answer the above questions, inspired by the works of Li--Sun--Yong \cite{Li-Sun-Yong-2016}, and Sun--Yong \cite{Sun-Yong-2020a,Sun-Yong-2020b}.

	\section{Open-Loop Solvability of Problem (MF-LQ)}

	This section is essentially answering the first question in the above, namely, we want to characterize the optimal control of Problem (MF-LQ) by the variational method. We now introduce the following definition.

	\bde{open-optimal control} \rm (i) Problem (MF-LQ) is said to be {\it finite} at $(s,\xi)\in\sD$ if
	\bel{>-infty}\inf_{u(\cd)\in\sU[s,T]}J(s,\xi;u(\cd))>-\infty.\ee
	If the above is true at any $(s,\xi)\in\sD$, we simply say that Problem (MF-LQ) is finite.
	
	\ms
	
	(ii) Problem (MF-LQ) is said to be ({\it uniquely})  {\it open-loop solvable} at
	$(s,\xi)\in\sD$, if there exists a (unique) $\bar u(\cd)\in\sU[s,T]$ such that
	\rf{J=inf-open} holds. In this case, $\bar u(\cd)$ is called an {\it open-loop optimal
		control} of Problem (MF-LQ) at $(s,\xi)\in\sD$,  $\bar X(\cd)$ and $(\bar X(\cd),
	\bar u(\cd))$ are called an {\it open-loop optimal process} and an {\it open-loop
		optimal pair}, respectively. Problem (MF-LQ) is said to be ({\it uniquely}) {\it open-loop solvable} if it is (uniquely) open-loop solvable for any $(s,\xi)\in\sD$.
	
	\ede
	
	Now, we look Problem (MF-LQ) from an abstract viewpoint. For $(s,\xi)\in\sD$ and $u(\cd)\in\sU[s,T]$, by the linearity of the state equation \rf{SDE1}, we have
	$$X(\cd)\equiv X(t;s,\xi,u(\cd))=\F_1(t,s)[u(\cd)]+\F_0(t,s)\xi+\f_0[b(\cd),\si(\cd)](t),\qq t\in[s,T],$$
	for some linear bounded operators $\F_1(\cd\, ,s):\sU[s,T]\to L^2_\dbF(s,T;\dbR^n)$,
	$\F_0(\cd\,,s):L^2_{\cF_s}(\O;\dbR^n)\to L^2_\dbF(s,T;\dbR^n)$, and a linear map $(b(\cd),\si(\cd))\mapsto\f_0[b(\cd).\si(\cd)]$. Consequently,
	\bel{cost-H}J(s,\xi;u(\cd))=\frac12\[\lan\Psi_2u(\cd),u(\cd)\ran+2\lan\psi_1,u(\cd)\ran
	+\psi_0\],\ee
	for some
	\bel{Psi_2}\ba{ll}
	\ns\ds\Psi_2:\sU[s,T]\to\sU[s,T]\q\hb{is linear bounded and self-adjoint},\\
	\ns\ds(\xi,b(\cd),\si(\cd),q(\cd),r(\cd),g)\mapsto\psi_1\equiv\psi_1(\xi,b(\cd),\si(\cd),
	q(\cd),r(\cd),g)\q\hb{is linear},\\
	\ns\ds(\xi,b(\cd),\si(\cd),q(\cd),r(\cd),g)\mapsto\psi_0\equiv\psi_0(\xi,b(\cd),\si(\cd),
	q(\cd),r(\cd),g)\q\hb{is quadratic}.\ea\ee
	Clearly, when $(\xi,b(\cd),\si(\cd),q(\cd),r(\cd),g)=0$,
	\bel{cost-H0}J^0(s,0;u(\cd))=\frac12\lan\Psi_2u(\cd),u(\cd)\ran.\ee
	From the above, we see that Problem (MF-LQ) is equivalent to the problem of minimizing quadratic functional \rf{cost-H} in the Hilbert space $\sU[s,T]$. Therefore, the following lemma is very useful, which can be found in Mou--Yong \cite{Mou-Yong-2006} (see also Sun--Yong \cite{Sun-Yong-2020a}).
	
	\bl{in-H} \sl Let $J(s,\xi;u(\cd))$ be defined by \rf{cost-H}. If
	\bel{finite}\inf_{u(\cd)\in\sU[s,T]}J(s,\xi;u(\cd))>-\infty,\ee
	for some $(s,\xi)\in\sD$, then
	\bel{Psi>0}\Psi_2\ges0.\ee
	If \rf{Psi>0} holds, then there exists a $\bar u(\cd)\in\sU[s,T]$ at which the map
	$u(\cd)\mapsto J(s,\xi;u(\cd))$ achieves its minimum if and only if
	\bel{range}\psi_1\in\sR\big(\Psi_2\big).\ee
	In this case, any $\bar u(\cd)\in\sU[s,T]$ is a minimizer of $u(\cd)\mapsto J(s.\xi;u(\cd))$
	if and only if $\bar u(\cd)$ is a solution of the following equation:
	\bel{equation0}\cD_uJ(s,\xi;\bar u(\cd))=\Psi_2\bar u+\psi_1=0,\ee
	where the
	left-hand side of the above is the Fr\'echet derivative of the map $u(\cd)\mapsto
	J(s,\xi;u(\cd))$ at $\bar u(\cd)$. Moreover,
	\bel{bar u*}\bar u=-\Psi_2^\dag\psi_1+(I-\Psi_2^\dag\Psi_2)\m,\ee
	where $\Psi_2^\dag$ is the pseudo-inverse of $\Psi_2$, and for some $\m\in\sU[s,T]$. In this case, the value function is given by
	\bel{Value}V(s,\xi)=\frac12\big(\psi_0-\lan\Psi_2^\dag\psi_1,\psi_1\ran\big),\qq
	V^0(s,\xi)=-\frac12\lan\Psi_2^\dag\psi_1,\psi\ran,\ee
	and consequently, for some constant $K_0>0$,
	\bel{|V^0|}|V^0(s,\xi)|\les K_0|\xi|^2,\qq\forall(s,\xi)\in\sD.\ee
	Further, if $\Psi_2$ is injective, then $\bar u$ is unique, and is given by
	$$\bar u=-\Psi_2^{-1}\psi_1.$$
	
	\el
	
	Note that condition \rf{Psi>0} is equivalent to the convexity of $u(\cd)\mapsto
	J^0(s,\xi;u(\cd))$, which is equivalent to the convexity of $u(\cd)\mapsto
	J(s,\xi;u(\cd))$ for some $\xi\in L^2_{\cF_s}(\O;\dbR^n)$, or, equivalently, for any
	$\xi\in L^2_{\cF_s}(\O;\dbR^n)$. The second part of the above lemma can also be stated as follows:
	Under condition \rf{Psi>0}, a control $u(\cd)\in\sU[s,T]$ is an open-loop optimal control if and only if
	\rf{equation0} holds. Therefore, it suffices to determine $\cD_uJ(s,\xi;u(\cd))$ and to
	solve equation \rf{equation0}. The following result is for this goal.
	
	\bt{LQ-open} \sl Let {\rm(H2.1)} hold.
	
	\ms
	
	{\rm(i)} If Problem {\rm(MF-LQ)} is finite, then $u(\cd)\mapsto J(s,\xi;u(\cd))$ is convex for some $\xi\in L^2_{\cF_s}(\O;\dbR^n)$ (or, for all $\xi\in L^2_{\cF_s}(\O;\dbR^n)$).
	
	\ms
	
	{\rm(ii)} Suppose the map $u(\cd)\mapsto J(s,\xi;u(\cd))$ is convex. Then a
	state-control pair $(\bar X(\cd),\bar u(\cd))$ is an optimal open-loop pair of
	Problem {\rm(MF-LQ)} if and only if there exists a triple $(\bar Y(\cd),\bar Z(\cd),
	\bar Z^M(\cd))$ of $\dbF$-adapted/predictable processes so that the following BSDE is satisfied:
	\bel{BSDE1}\left\{\2n\ba{ll}
	\ns\ds d\bar Y(t)\1n=\1n-\[A(t)^\top\bar Y(t)\1n+\1n\bar A(t)^\top\bar Y_2(t)\1n+\1n C(t)^\top\bar Z(t)\1n+\1n\bar C(t)^\top\bar Z_2(t)\1n+\1n Q(t)\bar X(t)\1n+\1n\bar Q(t)\bar X_2(t)\\
	\ns\ds\qq\qq+S(t)^\top\bar u(t)+\bar S(t)^\top\bar u_2(t)+q(t)+\bar q(t)\]dt+\bar Z(t)dW(t)+\bar Z^M(t)dM(t),\q t\1n\in\1n[s,T],\\
	\ns\ds\bar Y(T)=G\bar X(T)+\bar G\bar X_2(T)+g+\bar g.\ea\right.\ee
	where $\bar Y_2(\cd)=\Pi_2[\bar Y(\cd)]$, etc., and, in addition, the following {\it stationarity condition} holds:
	\bel{stationarity}\ba{ll}
	\ns\ds B(t)^\top\bar Y(t)+\bar B(t)\bar Y_2(t)+D(t)^\top\bar Z(t)+\bar D(t)^\top\bar Z_2(t)\\
	\ns\ds\q+S(t)\bar X(t)+\bar S(t)\bar X_2(t)+R(t)\bar u(t)+\bar R(t)\bar u_2(t)+r(t)+\bar r(t)=0,\q\ae t\in[s,T],~\as\ea\ee
	
	\et
	
	Note that BSDE \rf{BSDE1} is driven by Brownian motion $W(\cd)$ and martingale
	$M(\cd)$, with the drift independent of $\bar Z^M(\cd)$. The above triple
	$(\bar Y(\cd),\bar Z(\cd),\bar Z^M(\cd))$ is called a {\it predictable solution} of BSDE \rf{BSDE1}.
	By the martingale representation theorem, similar to the standard BSDE (driven by the
	Brownian motion only), one can have the well-posedness of such a BSDE, as long as
	$(\bar X(\cd),\bar u(\cd))$ and all the coefficients are given satisfying (H2.1) (see \cite{Carbone-2008}).
	We point out that both $\bar Z(\cd)$ and $\bar Z^M(\cd)$ are merely $\dbF\equiv\dbF^W
	\vee\dbF^M$-adapted/predictable. The above (ii) is essentially the Pontryagin type maximum principle. Since our problem is linear-quadratic, the set of the conditions is not only necessary, but also sufficient. The proof of the above result is very similar to a relevant result in Sun--Yong \cite{Sun-Yong-2020a}. However, for readers' convenience, we present a proof here.
	
	\ms
	
	\it Proof. \rm (i) is obvious. We now prove (ii).
	
	\ms
	
	(ii) \rm Let $(\bar X(\cd),\bar u(\cd))$ be a given open-loop optimal pair, and $u(\cd)\in\sU[s,T]$. let
	$$u^\e(\cd)=\bar u(\cd)+\e u(\cd),\q X^\e(\cd)=X(\cd\,;s,\xi,u^\e(\cd)).$$
	Clearly,
	$$\lim_{\e\to0}\dbE\int_s^T\Big|\frac{X^\e(t)-\bar X(t)}\e-X^0(t)\Big|^2dt=0,$$
	where $X^0(\cd)$ is the solution to the following homogeneous system (the same as
	\rf{SDE1-h}):
	\bel{SDE*}\left\{\2n\ba{ll}
	\ds dX^0(t)=\(A(t)X^0(t)+\bar A(t)X_2^0(t)+B(t)u(t)+\bar B(t)u_2(t)\)dt,\\
	\ns\ds\qq\qq\qq+\(C(t)X^0(t)+\bar C(t)X_2^0(t)+D(t)u(t)+\bar D(t)u_2(t)\)dW(t),\q t\ges s,\\
	\ns\ds X^0(s)=0,\ea\right.\ee
	Thus,
	\bel{0<J'}\ba{ll}
	\ns\ds0\les \frac1\e\[J(s,\xi;u^\e(\cd))-J(s,\xi;\bar u(\cd))\]\\
	\ns\ds\q=\frac 1{2\e}\[\dbE\int_s^T\(\lan QX^\e,X^\e\ran-\lan Q\bar X,\bar X\ran+\lan\bar QX_2^\e,X_2^\e\ran-\lan\bar Q\bar X_2,\bar X_2\ran\\
	\ns\ds\qq\qq+2\lan SX^\e,u^\e\ran-2\lan S\bar X,\bar u\ran+2\lan\bar SX_2^\e,u_2^\e\ran-2\lan\bar S\bar X_2,\bar u_2\ran\\
	\ns\ds\qq\qq+\lan Ru^\e,u^\e\ran-\lan R\bar u,\bar u\ran+\lan\bar Ru_2^\e,u^\e_2\ran-\lan\bar R\bar u_2,\bar u_2\ran\\
	\ns\ds\qq\qq+2\lan q,X^\e-\bar X\ran+2\lan\bar q,X_2^\e-\bar X_2\ran
	+2\lan r,u^\e-\bar u\ran+2\lan\bar r,u_2^\e-\bar u_2\ran\)dt\\
	\ns\ds\qq\qq+\lan GX^\e(T),X^\e(T)\ran-\lan G\bar X(T),\bar X(T)\ran+\lan
	\bar GX_2^\e(T)],X_2^\e(T)]\ran\\
	\ns\ds\qq\qq-\lan\bar G\bar X_2(T),\bar X_2(T)\ran+2\lan g,X^\e(T)
	-\bar X(T)\ran+2\lan\bar g,X_2^\e(T)-\bar X_2(T)\ran\]\\
	\ns\ds\q\to\dbE\[\int_s^T\(\lan Q\bar X,X^0\ran+\lan\bar Q\bar X_2,X_2^0\ran+\lan S\bar X,u\ran+\lan SX^0,\bar u\ran+\lan\bar S\bar X_2,u_2\ran+\lan\bar SX_2^0,\bar u_2\ran\\
	\ns\ds\qq\qq+\lan R\bar u,u\ran+\lan\bar R\bar u_2,u_2\ran+\lan q,X^0\ran+\lan\bar q,X_2^0\ran
	+\lan r,u\ran+\lan\bar r,u_2\ran\)dt\\
	\ns\ds\qq\qq+\lan G\bar X(T),X^0(T)\ran+\lan
	\bar G\bar X_2(T),X_2^0(T)\ran+\lan g,X^0(T)\ran+\lan\bar g,X_2^0\ran\]\\
	\ns\ds\q=\dbE\[\int_s^T\(\lan X^0,Q\bar X+\bar Q\bar X_2+S^\top\bar u+\bar S^\top\bar u_2+q+\bar q\ran+\lan u,S\bar X+\bar S\bar X+R\bar u+\bar R\bar u_2+r+\bar r\ran\)dt\\
	\ns\ds\qq\qq+\lan X^0(T),G\bar X(T)+\bar G\bar X_2(T)+g+\bar g\ran\].\ea\ee
	The above is true for all $u(\cd)\in\sU[t,T]$. However, the above relation is not explicit enough since $X^0(\cd)$ depends on $u(\cd)$. Now, we would like to use duality to transform $X^0(\cd)$ directly in terms of $u(\cd)$. To this end, let $(\bar Y(\cd),\bar Z(\cd),\bar Z^M(\cd))$ be the predictable solution of BSDE \rf{BSDE1} which we now write it in the following compact form (with $\bar\G(\cd)$ being its drift term):
	\bel{BSDE1*}\left\{\2n\ba{ll}
	\ds d\bar Y(t)=-\bar\G(t)dt+\bar Z(t)dW(t)+\bar Z^M(t)dM(t),\qq t\in[s,T],\\
	\ns\ds\bar Y(T)=\bar Y_T\equiv G\bar X(T)+\bar G\bar X_2(T)+g+\bar g.\ea\right.\ee
	Using It\^o's formula, we have
	$$\ba{ll}
	\ns\ds d\lan X^0(t),\bar Y(t)\ran=\[\lan A(t)X^0(t)+\bar A(t)X_2^0(t)+B(t)u(t)+\bar B(t)u_2(t),\bar Y(t)\ran\\
	\ns\ds\qq\qq\qq\qq-\lan X^0(t),\bar\G(t)\ran+\lan C(t)X^0(t)+\bar C(t)X_2^0(t)+D(t)u(t)+\bar D(t)u_2(t),\bar Z(t)\ran\]dt\\
	\ns\ds\qq\qq\qq\qq+\lan C(t)X^0(t)+\bar C(t)X_2^0(t)+D(t)u(t)+\bar D(t)u_2(t),\bar Y(t)\ran dW(t)\\
	\ns\ds\qq\qq\qq\qq+\lan X^0(t),\bar Z(t)\ran dW(t)+\lan X^0(t),\bar Z^M(t)\ran dM(t).\ea$$
	Thus,
	$$\ba{ll}
	\ns\ds\dbE\[\lan X^0(T),\bar Y_T\ran\]=\dbE\int_s^T\[\lan A(t)X^0(t)+\bar A(t)X_2^0(t)+B(t)u(t)+\bar B(t)u_2(t),\bar Y(t)\ran\\
	\ns\ds\qq\qq\qq\qq-\lan X^0(t),\bar\G(t)\ran+\lan C(t)X^0(t)+\bar C(t)X_2^0(t)+D(t)u(t)+\bar D(t)u_2(t),\bar Z(t)\ran\]dt.\ea$$
	Then \rf{0<J'} gives (suppressing $t$)
	$$\ba{ll}
	\ns\ds0\les\lim_{\e\to0}\frac{J(s,\xi;u^\e(\cd))-J(s,\xi;\bar u(\cd))}\e=\dbE\Big\{\int_s^T\(\lan X^0,Q\bar X+\bar Q\bar X_2+S^\top\bar u+\bar S^\top\bar u_2+q+\bar q\ran\\
	\ns\ds\qq\qq+\lan u,S\bar X+\bar S\bar X_2+R\bar u+\bar R\bar u_2+r+\bar r\ran\)dt+\lan X^0(T),G\bar X(T)+\bar G\bar X_2(T)+g+\bar g-\bar Y(T)\ran\Big\}\\
	\ns\ds\qq\qq+\dbE\int_s^T\[\lan AX^0+\bar AX_2^0+Bu+\bar Bu_2,\bar Y\ran-\lan X^0,\bar\G\ran+\lan CX^0+\bar CX^0_2+Du+\bar Du_2,\bar Z\ran\]dt\\
	\ns\ds\q=\dbE\Big\{\int_s^T\2n\(\lan X^0,Q\bar X\1n+\1n\bar Q\bar X_2\1n+\1n S^\top\bar u\1n+\1n\bar S^\top\bar u_2\1n+\1n q\1n+\1n\bar q\1n+\1n A^\top\bar Y\1n+\1n\bar A^\top\bar Y_2\1n+\1n C^\top\bar Z\1n+\1n\bar C^\top\bar Z_2-\bar\G\ran\\
	\ns\ds\qq\qq+\lan u,S\bar X+\bar S\bar X_2+R\bar u+\bar R\bar u_2+r+\bar r+B^\top\bar Y+\bar B^\top\bar Y_2+D^\top\bar Z+\bar D^\top\bar Z_2\ran\)dt\\
	\ns\ds\qq\qq\qq\qq\qq+\lan X^0(T),G\bar X(T)+\bar G\bar X_2(T)+g+\bar g-\bar Y(T)\ran\Big\}\\
	\ns\ds\q=\dbE\int_s^T\2n\lan\cD_uJ(s,\xi;\bar u(\cd)),u\ran dt.\ea$$
	Here, we have used the definition of $\bar\G$, the terminal condition in \rf{BSDE1} (or \rf{BSDE1*}), and
	$$\cD_uJ(s,\xi;\bar u(\cd))=S\bar X+\bar S\dbE^M[\bar X]+R\bar u+\bar R\dbE^M[\bar u]+r+\bar r+B^\top\bar Y+\bar B\dbE^M[\bar Y]+D^\top\bar Z+\bar D^\top\dbE^M[\bar Z].$$
	Hence, $\cD_uJ(s,\xi;\bar u(\cd))=0$, i.e., the stationarity condition
	\rf{stationarity} holds.
	
	\ms
	
	Conversely, under the convexity condition of $u(\cd)\mapsto J(s,\xi;u(\cd))$, if
	$(\bar X(\cd),\bar u(\cd))$ is a state-control pair, and $(\bar Y(\cd),\bar Z(\cd),
	\bar Z^M(\cd))$ satisfy \rf{BSDE1} such that \rf{stationarity}, then it means that
	\rf{equation0} holds. Hence, $(\bar X(\cd),\bar u(\cd))$ is an open-loop optimal pair
	of Problem (MF-LQ). \endpf
	
	\ms
	
	The state equation \rf{SDE1} and the adjoint equation \rf{BSDE1}, together with the stationarity condition \rf{stationarity}, can be put together as follows, called {\it optimality system}:
	\bel{FBSDE1}\left\{\2n\ba{ll}
	\ds d\bar X(t)=\[A(t)\bar X(t)+\bar A(t)\bar X_2(t)]+B(t)\bar u(t)+\bar B(t)\bar u_2(t)+b(t)\]dt,\\
	\ns\ds\qq\qq+\[C(t)\bar X(t)+\bar C(t)\bar X_2(t)]+D(t)\bar u(t)+\bar D(t)\bar u_2(t)+\si(t)\]dW(t),\qq t\in[s,T],\\
	\ns\ds d\bar Y(t)\1n=\1n-\[A(t)^\top\bar Y(t)\1n+\1n\bar A(t)^\top\bar Y_2(t)\1n+\1n C(t)^\top\bar Z(t)\1n+\1n\bar C(t)^\top\bar Z_2(t)\1n+\1n Q(t)\bar X(t)\1n+\1n\bar Q(t)\bar X_2(t)\\
	\ns\ds\qq\qq+S(t)^\top\bar u(t)+\bar S(t)^\top\bar u(t)+q(t)+\bar q(t)\]dt+\bar Z(t)dW(t)+\bar Z^M(t)dM(t),\q t\1n\in\1n[s,T],\\
	\ns\ds\bar X(s)=\xi,\qq\bar Y(T)=G\bar X(T)+\bar G\bar X_2(T)+g+\bar g,\\
	\ns\ds B(t)^\top\bar Y(t)+\bar B(t)^\top\bar Y_2(t)+D(t)^\top\bar Z(t)+\bar D(t)^\top\bar Z_2(t)\\
	\ns\ds\q+S(t)\bar X(t)+\bar S(t)\bar X_2(t)+R(t)\bar u(t)+\bar R(t)\bar u_2(t)+r(t)+\bar r(t)=0,\q\ae t\in[s,T],~\as\ea\right.\ee
	We may rewrite the above into the following, according to the orthogonal decomposition, dropping the bars in $\bar X_1$, etc. and suppressing $t$, which will have a simpler looking  below.
	\bel{FBSDE2}\left\{\2n\ba{ll}
	\ds dX_1=(A_1X_1+B_1u_1+b_1)dt+(C_1X_1+C_2X_2+D_1u_1+D_2u_2+\si)dW,\q t\in[s,T],\\
	\ns\ds dX_2=(A_2X_2+B_2u_2+b_2)dt,\\
	\ns\ds dY_1=-(A_1^\top Y_1+C_1^\top Z_1+Q_1 X_1+S_1^\top u_1+q_1)dt+ZdW+ Z_1^MdM,\q t\1n\in\1n[s,T],\\
	\ns\ds dY_2=-(A_2^\top Y_2+C_2^\top Z_2+Q_2 X_2+S_2^\top u_2+q_2)dt+Z_2^MdM,\q t\1n\in\1n[s,T],\\
	\ns\ds X_1(s)=\xi_1,\q  X_2(s)=\xi_2,\q Y_1(T)=G_1X_1(T)+g_1,\q
	Y_2(T)=G_2 X_2(T)+g_2,\\
	\ns\ds B_i^\top Y_i+D_i^\top Z_i+S_iX_i+R_iu_i+r_i=0,\q\ae t\in[s,T],~\as\q i=1,2.\ea\right.\ee
	Thus, Theorem \ref{LQ-open} amounts to saying that $u=u_1+u_2$ is an open-loop optimal control of Problem (MF-LQ) if and only if the system \rf{FBSDE2} is solvable.
	
	\ms
	
	Intuitively, it seems that we may further rewrite the above as two FBSDEs:
	\bel{FBSDE2-1}\left\{\2n\ba{ll}
	\ds dX_1=(A_1X_1+B_1u_1+b_1)dt+(C_1X_1+C_2X_2+D_1u_1+D_2u_2+\si)dW,\q t\in[s,T],\\
	\ns\ds dY_1=-(A_1^\top Y_1+C_1^\top Z_1+Q_1 X_1+S_1^\top u_1+q_1)dt+ZdW+ Z_1^MdM,\q t\1n\in\1n[s,T],\\
	\ns\ds X_1(s)=\xi_1,\q Y_1(T)=G_1X_1(T)+g_1,\\
	\ns\ds B_1^\top Y_1+D_1^\top Z_1+S_1X_1+R_1u_1+r_1=0,\q\ae t\in[s,T],~\as,\ea\right.\ee
	and
	\bel{FBSDE2-2}\left\{\2n\ba{ll}
	\ns\ds dX_2=(A_2X_2+B_2u_2+b_2)dt,\\
	\ns\ds dY_2=-(A_2^\top Y_2+C_2^\top Z_2+Q_2 X_2+S_2^\top u_2+q_2)dt+Z_2^MdM,\q t\1n\in\1n[s,T],\\
	\ns\ds X_2(s)=\xi_2,\q Y_2(T)=G_2 X_2(T)+g_2,\\
	\ns\ds B_2^\top Y_2+D_2^\top Z_2+S_2X_2+R_2u_2+r_2=0,\q\ae t\in[s,T],~\as\ea\right.\ee
	However, these two FBSDEs are coupled: The first involves $(X_2,u_2)$ in the diffusion, and the second contains $Z_2$ (in the drift) which can only be determined through the BSDE for $(Y_1,Y_2,Z_1,Z_2)$.
	
	\ms
	
	The following corollary is for the homogeneous Problem (MF-LQ)$^0$, which will be useful later.
	
	\bc{homo} \sl Let {\rm(H2.1)} hold and $u(\cd)\mapsto J^0(s,\xi;u(\cd))$ be convex. Then the 0 control is open-loop optimal for the homogeneous Problem (MF-LQ)$^0$ if and only if the following system is solvable:
	\bel{FBSDE2*}\left\{\2n\ba{ll}
	\ds dX_1=A_1X_1dt+(C_1X_1+C_2X_2)dW,\q t\in[s,T],\\
	\ns\ds dX_2=A_2X_2dt,\qq t\in[s,T],\\
	\ns\ds dY_1=-(A_1^\top Y_1+C_1^\top Z_1+Q_1 X_1)dt+ZdW+ Z_1^MdM,\q t\1n\in\1n[s,T],\\
	\ns\ds dY_2=-(A_2^\top Y_2+C_2^\top Z_2+Q_2 X_2)dt+Z_2^MdM,\q t\1n\in\1n[s,T],\\
	\ns\ds X_1(s)=\xi_1,\q  X_2(s)=\xi_2,\q Y_1(T)=G_1X_1(T),\q
	Y_2(T)=G_2 X_2(T),\\
	\ns\ds B_i^\top Y_i+D_i^\top Z_i+S_iX_i=0,\q\ae t\in[s,T],~\as\q i=1,2.\ea\right.\ee

	\ec
	
	Note that the above result only gives the equivalence between open-loop solvability of Problem (MF-LQ) and that of the FBSDE \rf{FBSDE1} (or \rf{FBSDE2}). Therefore, the existence of an open-loop optimal control is not guaranteed. We now introduce the following further condition.
	
	\ms
	
	{\bf(H4.1)} There exists a $\d>0$ such that
	\bel{uconvexity}J^0(s,0;u(\cd))\ges \d\dbE\int_s^T|u(t)|^2dt,\qq\forall u(\cd)\in \sU[s,T].\ee
	
	\ms
	
	The above condition is equivalent to the uniform convexity of the map $u(\cd)\mapsto
	J(s,\xi;u(\cd))$ for all $(s,\xi)\in\sD$, which means $\Psi_2$ (see \rf{Psi_2}) is
	(uniformly) positive definite on $\sU[0,T]$. Hence, by Lemma \ref{in-H}, one has the following result.
	
	\bp{existence} \sl Under {\rm(H2.1)} and {\rm(H4.1)}, Problem {\rm(MF-LQ)} admits a unique open-loop optimal control. Moreover, the open-loop optimal pair $(\bar X(\cd),\bar u(\cd))$ is determined by \rf{FBSDE1}.
	
	\ep

	\section{Closed-Loop Representation}
	
	We note that although \rf{FBSDE1} (or equivalently, \rf{FBSDE2}) characterizes the open-loop optimal control, it is not practically feasible. Here is the reason: to determine the value $\bar u(t)$ of the open-loop optimal control $\bar u(\cd)$ at the current time $t$, through the stationarity condition, say, under the invertibility condition of $R_1(t)$ and $R_2(t)$, the values $(\bar Y(t),\bar Z(t),\bar Z^M(t))$ are needed; These values are
	determined by solving the BSDE (for $(\bar Y(\cd),\bar Z(\cd),\bar Z^M(\cd))$) with the terminal
	condition involving the future value $\bar X(T)$ of the optimal state process; At
	current time $t$, the future value $\bar X(T)$ of the optimal state process
	$\bar X(\cd)$ is not available. Hence, the open-loop optimal control $\bar u(\cd)$
	cannot be practically constructed through the optimality system \rf{FBSDE1}. We now
	try to obtain a non-anticipating representation of the open-loop optimal control,
	without using the future information of $\bar X(\cd)$. The main idea is inspired by
	the so-called {\it invariant embedding} (due to Bellman--Kalaba--Wing
	\cite{Bellman-Kalaba-Wing-1960}; see also Ma--Protter--Yong \cite{Ma-Protter-Yong-1994}).
	More precisely, we let
	\bel{Y=}Y(t)=P_1(t)\big\{X(t)-\dbE[X(t)]\big\}+P_2(t)\dbE[X(t)]+\eta(t),\qq t\in[s,T],\ee
	for some $\dbF^M$-progressively measurable $\dbS^n$-valued processes $P_1(\cd)$ and $P_2(\cd)$, and $\dbF$-progressively measurable $\dbR^n$-process $\eta(\cd)=\eta_1(\cd)\oplus\eta_2(\cd)$, together with
	\bel{G}G_1X_1(T)+G_2X_2(T)+g_1+g_2=Y(T)=P_1(T)X_1(T)+P_2(T)X_2(T)+\eta_1(T)
	+\eta_2(T).\ee
	Hence, we may let $(P_i(\cd),\L^M_i(\cd))$ be the predictable solution to the BSDEs \rf{P_i}, with $\G_i(\cd)\in L^\infty_{\dbF^M}(s,T;\dbS^n)$ undetermined and $\L_i^M(\cd)\in L^2_{\dbF_-^M}(s,T;\dbS^n)$. We also let $(\eta_1(\cd),
	\z_1(\cd),\z_1^M(\cd))$, and $(\eta_2(\cd),\z_2^M(\cd))$ be the predictable solutions of \rf{eta_1} and \rf{eta_2}, respectively, with undetermined $\g_1(\cd)\in L^2_{\dbF^M}(s,T;\dbR^n)^\perp$, and $\g_2(\cd)\in L^2_{\dbF^M}(s,T;\dbR^n)$ (see \rf{eta_1-z_1} and \rf{eta_2-z_2}).  For convenience, we may simply rewrite \rf{Y=} as
	\bel{Y=*}Y=P_1X_1+P_2X_2+\eta_1+\eta_2,\qq t\in[s,T],\ee
	or equivalently,
	\bel{Y_1=*}Y_1=P_1X_1+\eta_1,\qq t\in[s,T],\ee
	\bel{Y_2=*}Y_2=P_2X_2+\eta_2,\qq t\in[s,T],\ee
	Consequently, by \rf{FBSDE2} and \rf{Y=},
	\bel{dY_1}\ba{ll}
	\ns\ds-(A_1^\top Y_1+C_1^\top
	Z_1+Q_1X_1+S_1^\top u_1+q_1)dt+ZdW+Z^M_1dM\\
	\ns\ds=dY_1=\big[\G_1X_1+\g_1+P_1(A_1X_1+B_1u_1+b_1)\big]dt\\
	\ns\ds\qq+\big[P_1(C_1X_1+C_2X_2+D_1u_1+D_2u_2+\si)+\z_1\big]dW
	+(\L_1^MX_1+\z_1^M)dM\\
	\ns\ds=\big[(\G_1+P_1A_1)X_1+\g_1+P_1B_1u_1+P_1b_1\big]dt\\
	\ns\ds\qq+(P_1C_1X_1+P_1C_2X_2+P_1D_1u_1+P_1D_2u_2+P_1\si+\z_1)dW
	+(\L_1^MX_1+\z_1^M)dM.\ea\ee
	Then,
	$$\ba{ll}
	\ds Z=P_1C_1X_1+P_1C_2X_2+P_1D_1u_1+P_1D_2u_2+P_1\si+\z_1,\\
	\ns\ds Z^M_1=\L_1^MX_1+\z_1^M.\ea$$
	This implies
	$$Z_i=P_1C_iX_i+P_1D_iu_i+P_1\si_i+\Pi_i[\z_1],\qq i=1,2.$$
	Next,
	\bel{dY_i2}\ba{ll}
	\ns\ds-(A_2^\top Y_2+C_2^\top Z_2+Q_2X_2+S_2^\top u_2+q_2)dt+Z^M_2dM\\
	\ns\ds=dY_2=\big[\G_2X_2+\g_2+P_2(A_2X_2+B_2u_2+b_2)\big]dt+(\L_2^MX_2+\z_2^M)dM.
	\ea\ee
	Hence,
	$$Z^M_2=\L^M_2X_2+\z^M_2.$$
	The stationarity conditions read
	$$\ba{ll}
	\ds0=B_i^\top Y_i+D_i^\top Z_i+S_iX_i+R_iu_i+r_i\\
	\ns\ds\q=B_i^\top(P_iX_i+\eta_i)+D_i^\top(P_1C_iX_i+P_1D_iu_i+P_1\si_i+\Pi_i[\z_1])
	+S_iX_i+R_iu_i+r_i\\
	\ns\ds\q=(B_i^\top P_i+D_i^\top P_1C_i+S_i)X_i+(R_i+D_i^\top P_1D_i)u_i
	+B_i^\top\eta_i+D_i^\top\Pi_i[\z_1]+D_i^\top P_1\si_i+r_i\\
	\ns\ds\q\equiv\cS_i(P_i,P_1)X_i+\cR_i(P_1)u_i+B_i^\top\eta_i+D_i^\top\Pi_i[\z_1]+D_i^\top P_1\si_i+r_i,\ea$$
	where see \rf{cQ-cR-cS} for the definition of $\cR_i(P_1)$ and $\cS_i(P_i,P_1)$.
	Now, by assuming the inclusion condition (comparable with \rf{cR>0}):
	\bel{inclusion1}\ba{ll}
	\ns\ds\sR\big(\cS_i(P_i,P_1)\big)\subseteq\sR\big(\cR_i(P_1)\big),\\
	\ns\ds B_i^\top\eta_i+D_i^\top\Pi_i[\z_1]+D_i^\top P_1\si_i+r_i\in\sR\big(\cR_i(P_1)\big),\ea\qq i=1,2,\ee
	which is true, in particular, if $\cR_i(P_1)$ is invertible, then
	\bel{u_1}\ba{ll}
	\ns\ds u_i\1n=\1n-\cR_i(P_1)^\dag\big[\cS_i(P_i,P_1)X_i\1n+\1n B_i^\top\eta_i\1n+\1n D_i^\top\Pi_i[\z_1]\1n
	+\1n D_i^\top P_1\si_i\1n+\1n r_i\big]\1n+\1n\big[I\1n-\1n\cR_i(P_1)^\dag\cR_i(P_1)\big]\m_i,\q i=1,2,\ea\ee
	which is the same as \rf{bar u}.
	Next, by comparing the drift terms in \rf{dY_1}, we have
	$$\ba{ll}
	\ns\ds0=A_i^\top(P_iX_i+\eta_i)+C_i^\top(P_1C_iX_i+P_1D_iu_i+P_1\si_i+\Pi_i[\z_1])\\
	\ns\ds\qq+Q_iX_i+S_i^\top u_i+q_i+(\G_i+P_iA_i)X_i+P_iB_iu_i+P_ib_i+\g_i\\
	\ns\ds\q=(\G_i+P_iA_i+A_i^\top P_i+C_i^\top P_1C_i+Q_i)X_i
	+(P_iB_i+C_i^\top P_1D_i+S_i^\top)u_i\\
	\ns\ds\qq+\g_i+A_i^\top\eta_i+C_i^\top\Pi_i[\z_1]+P_ib_i+C_i^\top P_1\si_i+q_i\\
	\ns\ds\q=\big(\G_i+\cQ_i(P_i,P_1)\big)X_i
	+\cS_i(P_i,P_1)^\top u_i+\g_i+A_i^\top\eta_i+C_i^\top\Pi_i[\z_1]+P_ib_i+C_i^\top P_1\si_i+q_i\\
	\ns\ds\q=\big(\G_i+\cQ_i(P_i,P_1)\big)X_i-\cS_i(P_i,P_1)^\top\cR_i(P_1,P_1)^\dag\big[\cS_i(P_i,P_1)
	X_i\1n+\1n B_i^\top\eta_i\1n+\1n D_i^\top\Pi_i[\z_1]\1n+\1n D_i^\top P_1\si_i\1n+\1n r_i\big]\\
	\ns\ds\qq+\g_i+A_i^\top\eta_i+C_i^\top\Pi_i[\z_1]+P_ib_i+C_i^\top P_1\si_i+q_i\\
	\ns\ds\q=\big[\G_i\1n+\1n\cQ_i(P_i,P_1)\1n-\1n\cS_i(P_i,P_1)^\top\cR_i(P_1,P_1)^\dag\cS_i(P_i)
	\big]X_i\1n+\1n\g_i+A_i^\top\eta_i\1n+\1n C_i^\top\Pi_i[\z_1]\1n+\1n P_ib_i\1n+\1n C_i^\top P_1\si_i\1n+\1n q_i\\
	\ns\ds\qq-\cS_i(P_i,P_1)^\top\cR_i(P_1)^\dag(B_i^\top\eta_i+D_i^\top\Pi_i[\z_1]
	+D_i^\top P_1\si_i+r_i),\ea$$
	where $\cQ_i(P_i,P_1)$ is defined by \rf{cQ-cR-cS}. Then, naturally, one should take
	\bel{G_ia}\G_i=-\(\cQ_i(P_i,P_1)-\cS_i(P_i,P_1)^\top\cR_i(P_1)^\dag\cS_i(P_i,P_1)\),\qq i=1,2,\ee
	and
	\bel{g_1}\ba{ll}
	\ns\ds\g_i=-\(A_i^\top\eta_i+C_i^\top\Pi_i[\z_1]+P_ib_i+C_i^\top P_1\si_i+q_i\\
	\ns\ds\qq\qq-\cS_i(P_i.P_1)^\top\cR_i(P_1)^\dag(B_i^\top\eta_i+D_i^\top\Pi_i[\z_1]
	+D_i^\top P_1\si_i+r_i)\),\qq i=1,2.\ea\ee
	Finally, \rf{G} is implied by the following:
	$$Y_i(T)=G_i,\qq\eta_i(T)=g_i,\qq i=1,2.$$
	Combining the above, we obtain the following result.
	
	\bp{closed-loop-ref} \sl Suppose Problem {\rm(MF-LQ)} admits an open-loop optimal
	control. Let the BSDREs \rf{Riccati1} admit predictable solutions $(P_i(\cd),\L_i^M(\cd))$. Let $(\eta_1(\cd),\z_1(\cd),\z_1^M(\cd),\eta_2(\cd),\z_2^M(\cd))$ be the predictable solution of \rf{BSDE-eta_1}. Then, the open-loop optimal control admits the closed-loop representation \rf{bar u}.
	
	\ep
	
	Clearly, \rf{bar u} is a (current) state-feedback representation, which is non-anticipating, and in principle, is practically realizable. Sometimes, we also call \rf{bar u} a state-feedback control.

	\ms
	
	Note that $(\eta_1(\cd),\z_1(\cd),\z_1^M(\cd),\eta_2(\cd),\z_2^M(\cd))$ can be obtained off-line, which are used to handle
	the nonhomogeneous terms in the state equation and the linear weighting terms in the cost functional. For the homogeneous case (i.e., $(b(\cd),\si(\cd),q(\cd),\bar q(\cd),r(\cd),\bar r(\cd),g,\bar g)=0$), $(\eta_1(\cd),\z_1(\cd),\z_1^M(\cd),\eta_2(\cd),\z_2^M(\cd))=0$. 	
	For such a case, \rf{bar u} will have a much simpler form.

	\ms
	
	In this section, by decoupling the optimality system, an FBSDE, we have formally derived the BSDREs. The method is clearly different from that in Section 3.

	\section{Closed-Loop Solvability}

	From the previous section, we see that under proper conditions, open-loop optimal
	control admits a closed-loop representation (or, of state-feedback form), which is non-anticipating. However, it is not clear if such a state-feedback control is optimal within the class of the state feedback controls. This suggests us introduce the so-called {\it closed-loop
		solvability} of the LQ problem. In this section, we are going to make this precise.
	
	\ms
	
	For any $0\les s<T<\i$, recall
	$$ L_{\dbF^M}^\i\1n(s,\1n T;\dbR^{m\times n})\1n=\1n\Big\{\Th\1n:\1n[s,T]\1n\times\1n\O\1n\to\1n\dbR^{m\times n}\bigm|\Th(\cd)\hb{ is $\dbF^M$-progressively measurable, }\esssup_{t\in[s,T]}\|\Th(t,\cd)\|_\i <\infty\Big\},$$
	and set
	$$\hTh[s,T]=L^\i_{\dbF^M}(s,T;\dbR^{m\times n})\times L^\i_{\dbF^M}(s,T;\dbR^{m\times n})\equiv L^\i_{\dbF^M}(s,T;\dbR^{m\times n})^2.$$
	Any element in $\hTh[s,T]$ is denoted by $\BTh(\cd)\equiv(\Th_1(\cd),\Th_2(\cd))$, and any $(\BTh(\cd),v(\cd))\in\hTh[s,T]\times\sU[s,T]$ is called a {\it closed-loop strategy}. Note that both $\Th_1(\cd)$ and $\Th_2(\cd)$ are $\dbF^M$-progressively measurable, whereas $v(\cd)=v_1(\cd)\oplus v_2(\cd)$ is in $\sU[s,T]=L^2_{\dbF^M}(s,T;\dbR^m)^\perp\oplus L^2_{\dbF^M}(s,T;\dbR^m)$. Namely, $v_2(\cd)$ is $\dbF^M$-progressively measurable, and $v_1(\cd)$ is merely $\dbF$-progressively measurable (not $\dbF^M$-measurable, unless it is zero).
	
	\ms
	
	For any closed-loop strategy $(\BTh(\cd),v(\cd))\equiv(\Th_1(\cd),\Th_2(\cd),v_1(\cd),
	v_2(\cd))\in\hTh[s,T]\times\sU[s,T]$, and any initial pair $(s,\xi)\in\sD$, let
	\bel{outcome1}\ba{ll}
	\ns\ds u(t)=[\Th_1(t)X_1(t)+v_1(t)]\oplus[\Th_2(t)X_2(t)+v_2(t)]\\
	\ns\ds\qq\equiv(\Th_1(t),\Th_2(t))\begin{pmatrix}X_1(t)\\ X_2(t)\end{pmatrix}+v(t)
	\equiv\BTh(t)X(t)+v(t),\qq t\in[s,T],\ea\ee
	which is called an {\it outcome} of the closed-loop strategy $(\BTh(\cd),v(\cd))$
	corresponding to the initial pair $(s,\xi)$, where $X(\cd)$ is the solution of the
	following {\it closed-loop system}: (see Appendix B)
	\bel{closed1}dX=\(A_1^{\Th_1}X_1\1n+\1n A_2^{\Th_2}X_2\1n+\1n
	B_1v_1+B_2v_2+b\)dt\1n+\1n\(C_1^{\Th_1}X_1+C_2^{\Th_2}X_2+D_1v_1+D_2v_2+\si
	\)dW(t),\ee
	where the initial state $X(s)=\xi$, and (see \rf{A_1})
	\bel{A^Th}A_i^{\Th_i}(\cd)=A_i(\cd)+B_i(\cd)\Th_i(\cd),\q C_i^{\Th_i}(\cd)=C_i(\cd)+D_i(\cd)\Th_i(\cd),\q i=1,2.\ee
	Or, equivalently
	\bel{closed1*}\left\{\2n\ba{ll}
	\ds dX_1=\(A_1^{\Th_1}X_1\1n+\1n
	B_1v_1+b_2\)dt\1n+\1n\(C_1^{\Th_1}X_1+C_2^{\Th_2}X_2+D_1v_1+D_2v_2+\si
	\)dW(t),\q t\in[s,T],\\
	\ns\ds dX_2=\(A_2^{\Th_2}X_2\1n+B_2v_2+b_2\)dt.\ea\right.\ee
	Note that all the coefficients are $\dbF^M$-adapted. The control $u(\cd)$ defined by \rf{outcome1} is also called a {\it state-feedback control} (which depends on the initial pair $(s,\xi)$ through the state process). We further note that in \rf{outcome1}, $\Th_1(\cd)$ and $\Th_2(\cd)$ could be different. If they were the same, say, equal to some $\Th_0(\cd)$, then
	\bel{outcome2}u(t)=\Th_0(t)X_1(t)+\Th_0(t)X_2(t)+v(t)=\Th_0(t)X(t)+v(t).\ee
	We should distinguish the meaning of $\Th_0(t)X(t)$ on the right-hand of the above and that of $\BTh(t)X(t)$ on the right-hand side of \rf{outcome1}. Clearly, the class of feedback controls of form \rf{outcome1} is much larger than that of the form \rf{outcome2}. Such an idea is borrowed from Yong \cite{Yong-2013}.
	
	\ms
	
	Correspondingly, (see Appendix B)
	\bel{cost2b}\ba{ll}
	\ns\ds J(s,\xi;\BTh(\cd),v(\cd)):=J(s,\xi;\Th_1(\cd)X_1(\cd)+\Th_2(\cd)X_2(\cd)+v(\cd))\\
	\ns\ds=\frac12\dbE\Big\{\int_s^T\[\lan Q_1^{\Th_1}X_1,X_1\ran+2\lan S_1^{\Th_1} X_1,v_1\ran+\lan R_1v_1,v_1\ran+2\lan q_1^{\Th_1},X_1\ran+2\lan r_1,v_1\ran\\
	\ns\ds\qq\qq\qq+\lan Q_2^{\Th_2}X_2,X_2\ran+2\lan S_2^{\Th_2}X_2,v_2\ran+\lan R_2v_2,v_2\ran+2\lan q_2^{\Th_2},X_2\ran+2\lan r_2,v_2\ran\]dt\\
	\ns\ds\qq\qq\qq+\lan G_1X_1(T),X_1(T)\ran+\!2\lan g_1,X_1(T)\ran+\lan G_2X_2(T),X_2(T)\ran+2\lan g_2,X_2(T)\ran\Big\},\ea\ee
	where (see \rf{Q_1})
	\bel{notation-QRS}\ba{ll}
	\ns\ds Q_i^{\Th_i}(t)=Q_i(t)+\Th_i(t)^\top S_i(t)+S_i(t)^\top\Th_i(t)+\Th_i(t)^\top R_i(t)\Th_i(t),\\
	\ns\ds S_i^{\Th_i}(t)=S_i(t)+R_i(t)\Th_i(t),\qq q_i^{\Th_i}(t)=q_i(t)+\Th_i(t)^\top r_i(t),\qq i=1,2.\ea\ee
	In particular, for the homogenous case, one has
	\bel{J^0(Th)}\ba{ll}
	\ns\ds J^0(s,\xi;\BTh(\cd),v(\cd))=\frac12\dbE\Big\{\int_s^T\[\lan Q_1^{\Th_1}X_1,X_1\ran+2\lan S_1^{\Th_1} X_1,v_1\ran+\lan R_1v_1,v_1\ran\\
	\ns\ds\qq\qq\qq\qq\qq\qq+\lan Q_2^{\Th_2}X_2,X_2\ran+2\lan S_2^{\Th_2}X_2,v_2\ran+\lan R_2v_2,v_2\ran\]dt\\
	\ns\ds\qq\qq\qq\qq\qq\qq+\lan G_1X_1(T),X_1(T)\ran+\lan G_2X_2(T),X_2(T)\ran\Big\},\ea\ee
	Note that by taking $\Th_1(\cd)=\Th_2(\cd)=0$, we recover the original state equation and cost functional (with $u_i(\cd)=v_i(\cd)$.
	
	\ms
	
	We now introduce the following definition.
	
	\bde{optimal control} \rm Problem (MF-LQ) is said to be ({\it uniquely})
	{\it closed-loop solvable} at $s\in[0,T)$, if there exists a (unique) pair $(\bar\BTh,
	\bar v(\cd))\in\hTh[s,T]\times\sU[s,T]$ such that for any $\xi\in L^2_{\cF_s}(\O;
	\dbR^n)$,
	\bel{bar J1}J(s,\xi;\bar\BTh(\cd),\bar v(\cd))\les J(s,\xi;\BTh(\cd),v(\cd)),\qq\forall(\BTh(\cd),v(\cd))\in\hTh[s,T]\times\sU[s,T].\ee
	In this case, $(\bar\BTh(\cd),\bar v(\cd))$ is called a (the) {\it closed-loop optimal strategy}. The corresponding state process $\bar X(\cd)$ is called a (the) {\it closed-loop optimal state process}.
	\ede
	
	We emphasize that the closed-loop optimal strategy $(\bar\BTh(\cd),\bar v(\cd))$ (if it exists) is independent of the initial state $\xi\in L^2_{\cF_s}(\O;\dbR^n)$. This is the main feature of closed-loop optimal strategies which distinguishes it from that of the open-loop open controls (see Sun--Yong \cite{Sun-Yong-2020a}). We have the following simple result.

	\bp{6.2} \sl {\rm(i)} If $(\bar\BTh(\cd),\bar v(\cd))$ is a closed-loop optimal strategy of Problem {\rm(MF-LQ)}. Then
	\bel{bar J1b}J(s,\xi;\bar\BTh(\cd),\bar v(\cd))\les J(s,\xi;u(\cd)),\qq\forall u(\cd)\in\sU[s,T],\ee
	which means any outcome of $(\bar\BTh(\cd),\bar v(\cd))$ is an open-loop optimal control.
	
	\ms
	
	{\rm(ii)} If $(\bar\BTh(\cd),\bar v(\cd))\in\hTh[s,T]\times\sU[s,T]$ such that any outcome of it is an open-loop optimal control of $(s,\xi)\in\sD$, that is
	$$J(s,\xi;\bar \BTh(\cd),\bar v(\cd))=\inf_{u(\cd)\in\sU[s,T]}J(s,\xi;u(\cd)),$$
	then it must be a closed-loop optimal control.
	
	\ms
	
	{\rm(iii)} Let Problem {\rm(MF-LQ)} be closed-loop solvable at $s\in[0,T)$. Then for any $s'\in(s,T)$, Problem {\rm(MF-LQ)} is closed-loop solvable at $s'$.
	
	\ep

	\it Proof. \rm (i) For $(s,\xi)\in\sD$, by letting $\bar u(\cd)$ as \rf{outcome1}, we see that it is an open-loop optimal control, which means \rf{bar J1b}  or  \rf{bar J1} holds.
	
	\ms
	
	(ii) For any $(\BTh(\cd),v(\cd))\in\hTh[s,T]\times\sU[s,T]$, and $(s,\xi)\in\sD$, let $u(\cd)$ be the corresponding outcome. Then by our condition,
	$$J(s,\xi;\bar\BTh(\cd),\bar v(\cd))=\inf_{u(\cd)\in\sU[s,T]}J(s,\xi;u(\cd))\les J(s,\xi;u(\cd))\les J(s,\xi;\BTh(\cd),v(\cd)).$$
	Thus proves our conclusion.
	
	\ms
	
	(iii) Let Problem (MF-LQ) be closed-loop solvable at $s\in [0,T)$ and $s'\in (s,T)$. Then there exists a closed-loop optimal strategy $(\bar\BTh(\cd),\bar v(\cd))$. It is clear that the restriction $(\bar\BTh(\cd)\big|_{s',T]},\bar v(\cd)\big|_{[s',T]})$ of $(\bar\BTh(\cd),\bar v(\cd))$ on $[s',T)$ is a closed-loop optimal strategy of Problem (MF-LQ) at $s'$. \endpf
	
	\ms
	
	Let us recall the following BSDRE (see \rf{Riccati1})
	\bel{Riccati2}\left\{\2n\ba{ll}
	\ds dP_i(t)=-\big[\cQ_i(P_i,P_1)-\cS_i(P_2,P_1)^\top\cR_i(P_1)^\dag\cS_i(P_i,P_1) \big](t)dt+\L^M_i(t)dM(t),\qq t\in[s,T],\\
	\ns\ds P_i(T)=G_i,\ea\right.\ee
	and the following BSDEs (see \rf{BSDE-eta_1}):
	\bel{eta*}\left\{\2n\ba{ll}
	\ds d\eta_1(t)=-\big[A_1^\top\eta_1+C_1^\top\Pi_1[\z_1]+P_1b_1+C_1^\top P_1\si_1+q_1\\
	\ns\ns\qq\qq\q-\cS_1(P_1,P_1)^\top \cR_1(P_1)^\dag(B_1^\top\eta_1\1n+\1n D_1^\top\Pi_1[\z_1]\1n+\1n D_1^\top P_1\si_1\1n+\1n r_1)\big]dt+\z_1dW(t)+\z^M_1dM(t),\\
	\ns\ds d\eta_2(t)=-\big[A_2^\top\eta_2+C_2^\top\Pi_2[\z_1]+P_2b_2+C_2^\top P_1\si_2+q_2\\
	\ns\ds\qq\qq\q-\cS_2(P_2,P_1)^\top \cR_2(P_1)^\dag(B_2^\top\eta_2+D_2^\top\Pi_2[\z_1]+D_2^\top P_1\si_2+r_2)\big]dt+\z^M_2dM(t), \\
	\ns\ds\eta_1(T)=g_1,\qq\eta_2(T)=g_2.\ea\right.\ee
	From the previous sections, we see that under (H2.1), if the above equations have predictable solutions so that the following conditions are satisfied:
	\bel{R>0}\sR\big(\cS_i(P_i,P_1)\big)\subseteq\sR\big(\cR_i(P_1)\big),\qq\cR_i(P_1)\ges0,
	\ee
	and (recall \rf{Bq})
	\bel{B_1eta_1*}\left\{\2n\ba{ll}
	\ns\ds\Br_i\equiv B_i^\top\eta_i+D_i^\top\Pi_i[\z_1]+D_i^\top P_1\si_i+r_i\in\sR(\cR_i(P_1)),\\
	\ns\ds\cR_1(P_1)^\dag\Br_1\in L^2_{\dbF^M}(s,T;\dbR^m)^\perp,\qq
	\cR_2(P_1)^\dag\Br_2\in L^2_{\dbF^M}(s,T;\dbR^m),\ea\right.\ee
	then Problem (MF-LQ) has an open-loop optimal control which has a state-feedback representation:
	\bel{bar u*b}\bar u_i=-\cR_i(P_1)^\dag\cS_i(P_i,P_1)\bar X_i-\cR_i(P_1)^\dag\Br_i+[I-\cR_i(P_1)^\dag\cR_i(P_1)]\m_i,\qq i=1,2,\ee
	for some $\m_i$. Now, if we define
	\bel{cl-opt}\left\{\2n\ba{ll}
	\ns\ds\bar\Th_i=-\cR_i(P_1)^\dag\cS_i(P_i,P_1)+[I-\cR_i(P_1)^\dag\cR_i(P_1)]
	\Th_{0i}\in L^\infty_{\dbF^M}(s,T;\dbR^{m\times n}),\\
	\ns\ds\bar v_1=-\cR_1(P_1)^\dag\Br_1\1n+\1n[I\1n-\1n\cR_1(P_1)^\dag\cR_1(P_1)]\th_1\in L^2_{\dbF^M}(s,T;\dbR^{ m})^\perp,\\
	\ns\ds\bar v_2=-\cR_2(P_1)^\dag\Br_2+[I-\cR_2(P_1)^\dag\cR_2(P_1)]\th_2\in L^2_{\dbF^M}(s,T;\dbR^{m}),\ea\right.\ee
	for some $\Th_{0i}$ and $\th_i$, then
	$$\bar u_i=\bar\Th_i\bar X_i+\bar v_i,\qq i=1,2.$$
	This means that $\bar u(\cd)$ is the outcome of $(\bar\BTh(\cd),\bar v(\cd))$. Hence, by Proposition \ref{6.2}, (ii), we see that the closed-loop strategy defined by \rf{cl-opt} is a closed-loop optimal strategy (at $s$). The purpose of this section is to show that the conditions imposed above are also necessary for Problem (MF-LQ) to be closed-loop solvable. For convenience, we introduce the following notion.
	
	\bde{r-solution} \rm Two pairs $(P_i(\cd),\L^M_i(\cd))\in L^\i_{\dbF^M}(s,T;\dbS^n)\times L^2_{\dbF_-^M}(s,T;\dbS^n)$ ($i=1,2$) are called a {\it regular predictable solution} of BSDRE \rf{Riccati2} if (the drift) $\dot P_i\in L^\infty_{\dbF^M}(s,T;\dbS^n)$ and \rf{R>0} holds, and \rf{Riccati2} is satisfied in the usual sense. In this case, the BSDRE \rf{Riccati2} is sad to be {\it regularly solvable}. If, in addition, the following holds:
	\bel{cR>d}\cR_i(P_1)\ges\d I,\qq i=1,2,\ee
	for some $\d>0$, then these pairs are called a {\it strong regular predictable solution} of \rf{Riccati2}. In this case, the BSDRE \rf{Riccati2} is said to be {\it strongly regularly solvable}.
	
	\ede
	
	\ms
	
	We now present the main result of this section.

	\bt{mainthm1}\sl  Let {\rm (H2.1)} hold. Problem {\rm(MF-LQ)} is closed-loop solvable at $s\in[0,T)$ if and only if the following hold:
	
	\ss
	
	{\rm(i)} The BSDRE \eqref{Riccati2} admits a regular predictable solution $(P_i(\cd),\L_i^M(\cd))\in L^\i_{\dbF^M}(s,T;\dbS^n)\times L^2_{\dbF_-^M}(s,T;\dbS^n)$.

	{\rm(ii)} There exist processes $(\eta_1(\cd),\z_1(\cd),\z_1^M(\cd))\in L_{\dbF^M}^2(s,T;\dbR^n)^\perp\times L_\dbF^2(s,T;\dbR^n)\times L_{\dbF^M_-}^2(s,T;\dbR^n)^\perp$ and $(\eta_2(\cd),\z_2^M(\cd))\in L_{\dbF^M}^2(s,T;\dbR^n)\times L_{\dbF^M_-}^2(s,T;\dbR^n)$ having properties \rf{B_1eta_1*}, such that BSDEs \rf{eta*} is satisfied on $[s,T]$.
	
	\ms
	
	In this case, any closed-loop optimal strategy is given by \rf{cl-opt}. Moreover, \rf{V} holds.

	\et
	
	\it Proof. \rm Sufficiency follows from Theorem 3.1, so is the representation of the value function \rf{V}, with a little possible modification. Hence, we need only to prove the necessity. The proof is lengthy, and we split it into a couple of steps.
	
	\ms
	
	\it Step 1. \rm Solvability of BSDREs \rf{Riccati2}.
	
	\ms
	
	Let $(\bar\BTh(\cd),\bar v(\cd))$ be a closed-loop optimal strategy which satisfies, by definition,
	\bel{les}J(s,\xi,\bar u(\cd))=J(s,\xi;\bar\BTh(\cd),\bar v(\cd))\les J(s,\xi;\bar\BTh(\cd),v(\cd)),\qq\forall v(\cd)\in\sU[s,T],\ee
	where $\bar u(\cd)=\bar\Th_1(\cd)\bar X_1(\cd)+\bar\Th_2(\cd)\bar X_2(\cd)+\bar v(\cd)$, the outcome of the optimal strategy $(\bar\BTh(\cd),\bar v(\cd))$. Note that under state feedback control $u(\cd)=\bar\Th_1(\cd)X_1(\cd)+\bar\Th_2(\cd)X_2(\cd)+v(\cd)$, the state
	equation reads (taking $\Th_i(\cd)=\bar\Th_i(\cd)$ in \rf{closed1*})
	\bel{closed4}\left\{\2n\ba{ll}
	\ds dX_1=(A_1^{\bar\Th_1}X_1+B_1v_1+b_1)dt
	+(C_1^{\bar\Th_1}X_1+C_2^{\bar\Th_2}X_2+D_1v_1+D_2v_2+\si)dW(t),\\
	\ns\ds dX_2=(A_2^{\bar\Th_2}X_2+B_2v_2+b_2)dt,\\
	\ns\ds X_1(s)=\xi_1,\qq X_2(s)=\xi_2.\ea\right.\ee
	Correspondingly, similar to \rf{cost2}. we have
	\bel{bar J}\ba{ll}
	\ns\ds\bar J(s,\xi;v(\cd)):=J(s,\xi;\bar\BTh(\cd)X(\cd)+v(\cd))\\
	\ns\ds=\frac12\dbE\Big\{\int_s^T\[\lan Q_1^{\bar\Th_1}X_1,X_1\ran+2\lan S_1^{\bar\Th_1} X_1,v_1\ran+\lan R_1v_1,v_1\ran+2\lan q_1^{\bar\Th_1},X_1\ran+2\lan r_1,v_1\ran\\
	\ns\ds\qq\qq\qq+\lan Q_2^{\bar\Th_2}X_2,X_2\ran+2\lan S_2^{\bar\Th_2}X_2,v_2\ran+\lan R_2v_2,v_2\ran+2\lan q_2^{\bar\Th_2},X_2\ran+2\lan r_2,v_2\ran\]dt\\
	\ns\ds\qq\qq\qq+\lan G_1X_1(T),X_1(T)\ran+\!2\lan g_1,X_1(T)\ran+\lan G_2X_2(T),X_2(T)\ran+2\lan g_2,X_2(T)\ran\Big\},\ea\ee
	with $Q_1^{\bar\Th}$, and so on, defined by \rf{notation-QRS} (replacing $\Th$ by $\bar\Th$). The above means that $\bar v(\cd)$ is an open-loop optimal control of the problem with state equation \rf{closed4} and the cost functional \rf{bar J}, and the optimal state process is $\bar X(\cd)$. Hence, from Theorem \ref{LQ-open} and \rf{FBSDE2}, we have the following (suppressing $t$):
	\bel{FBSDE3}\left\{\2n\ba{ll}
	\ds d\bar X_1=(A_1^{\bar\Th_1}\bar X_1+B_1\bar v_1+b_1)dt+(C_1^{\bar\Th_1}\bar X_1+C_2^{\bar\Th_2}\bar X_2+D_1\bar v_1+D_2\bar v_2+\si)dW,\q t\in[s,T],\\
	\ns\ds d\bar X_2=(A_2^{\Th_2}\bar X_2+B_2\bar v_2+b_2)dt,\\
	\ns\ds d\bar Y_1=-\big[(A_1^{\bar\Th_1})^\top\bar Y_1+(C_1^{\bar\Th_1})^\top\bar Z_1+Q_1^{\bar\Th_1}\bar X_1+(S_1^{\bar\Th_1})^\top\bar v_1+q_1^{\bar\Th_1}\big]dt+\bar ZdW+\bar Z_1^MdM,\q t\1n\in\1n[s,T],\\
	\ns\ds d\bar Y_2=-\big[(A_2^{\bar\Th_2})^\top\bar Y_2+(C_2^{\bar\Th_2})^\top\bar Z_2+Q_2^{\bar\Th_2}\bar X_2+(S_2^{\bar\Th_2})^\top\bar v_2+q^{\bar\Th_2}_2\big]dt+\bar Z_2^MdM,\q t\1n\in\1n[s,T],\\
	\ns\ds\bar X_1(s)=\xi_1,\q \bar X_2(s)=\xi_2.\q\bar Y_1(T)=G_1\bar X_1(T)+g_1,\q\bar Y_2(T)=G_2\bar X_2(T)+g_2,\\
	\ns\ds B_i^\top\bar Y_i+D_i^\top\bar Z_i+S_i^{\bar\Th_i}\bar X_i+R_i\bar v_i+r_i=0,\q\ae t\in[s,T],~\as,\q i=1,2.\ea\right.\ee
	Note that with the same $(\bar\BTh(\cd),\bar v(\cd))$, the above admits a predictable solution $(\bar X,\bar Y,\bar Z,\bar Z^M)$ (which depends on the initial state $\xi$). Hence, if we let $(\bar X^0,\bar Y^0,\bar Z^0,(\bar Z^M)^0)$ be the predictale solution of the above FBSDE corresponding to $\xi=0$, then
	$$\h X=\bar X-\bar X^0,\q\h Y=\bar Y-\bar Y^0,\q\h Z=\bar Z-\bar Z^0,\q\h Z^M=\bar Z^M-(Z^M)^0$$
	satisfies
	\bel{FBSDE4}\left\{\2n\ba{ll}
	\ds d\h X_1=A_1^{\bar\Th_1}\h X_1dt+(C_1^{\bar\Th_1}\h X_1+C_2^{\bar\Th_2}\h X_2)dW,\q t\in[s,T],\\
	\ns\ds d\h X_2=A_2^{\Th_2}\h X_2dt,\qq t\in[s,T],\\
	\ns\ds d\h Y_1=-\big[(A_1^{\bar\Th_1})^\top\h Y_1+(C_1^{\bar\Th_1})^\top\h Z_1+Q_1^{\bar\Th_1}\h X_1\big]dt+\h ZdW+\h Z_1^MdM,\q t\1n\in\1n[s,T],\\
	\ns\ds d\h Y_2=-\big[(A_2^{\bar\Th_2})^\top\h Y_2+(C_2^{\bar\Th_2})^\top\h Z_2+Q_2^{\bar\Th_2}\h X_2\big]dt+\h Z_2^MdM,\q t\1n\in\1n[s,T],\\
	\ns\ds\h X_1(s)=\xi_1,\q \h X_2(s)=\xi_2,\q\h Y_1(T)=G_1\h X_1(T),\q\h Y_2(T)=G_2\bar X_2(T),\\
	\ns\ds B_i^\top\h Y_i+D_i^\top\h Z_i+S_i^{\bar\Th_i}\h X_i=0,\q\ae t\in[s,T],~\as\q i=1,2.\ea\right.\ee
	By Corollary \ref{homo}, we know that 0 control is open-loop optimal for the homogeneous Problem (MF-LQ)$^0$. We note that for given $\xi=\xi_1+\xi_2$, the above decoupled FBSDE for $(\h X(\cd),\h Y(\cd),\h Z(\cd),\h Z^M(\cd))$ admits a unique predictable solution, and the last stationarity conditions are satisfied (for the given $\bar\BTh(\cd)=(\bar\Th_1(\cd),\bar\Th_2(\cd))$). Next, we introduce the following linear BSDEs for $\dbS^n$-valued $\dbF^M$-predictable processes $(P_i(\cd),\L_i^M(\cd))$ ($i=1,2$):
	\bel{BSDE-Pi}\left\{\2n\ba{ll}
	\ds dP_i=-\big[P_iA_i^{\bar\Th_i}+(A_i^{\bar\Th_i})^\top P_i+(C_i^{\bar\Th_i})^\top P_1C_i^{\bar\Th_i}+Q_i^{\bar\Th_i}\big]dt+\L_i^MdM,\qq t\in[s.T],\\
	\ns\ds P_i(T)=G_i.\ea\right.\ee
	These two BSDEs admit unique predictable solutions with $P_i(\cd),\dot P_i(\cd)\in L^\infty_{\dbF^M}(0,T;\dbS^n)$. Now, we define $\wt Y=\wt Y_1+\wt Y_2$ by the following:
	$$\wt Y_i=P_i\h X_i,
	\qq i=1,2.$$
	By It\^o's formula, we have
	$$\ba{ll}
	\ns\ds d\wt Y_1=d(P_1\h X_1)=\[-\(P_1A_1^{\bar\Th_1}+(A^{\bar\Th_1}_1)^\top P_1+(C_1^{\bar\Th_1})^\top P_1C_1^{\bar\Th_1}+Q_1^{\bar\Th_1}\)\h X_1+ P_1A_1^{\bar\Th_1}\h X_1)\]dt\\
	\ns\ds\qq\qq\qq\qq\qq+P_1(C_1^{\bar\Th_1}\h X_1+C_2^{\bar\Th_2}\h X_2)dW+\L_1^M\h X_1dM\\
	\ns\ds\qq=-\((A^{\bar\Th_1}_1)^\top\wt Y_1+(C_1^{\bar\Th_1})^\top P_1C_1^{\bar\Th_1}\h X_1+Q_1^{\bar\Th_1}\h X_1\)dt+(P_1C_1^{\bar\Th_1}\h X_1+ P_1C_2^{\bar\Th_2}\h X_2)dW+\L_1^M\h X_1dM.\ea$$
	Thus, if we let
	$$\wt Z=P_1C^{\bar\Th_1}_1\h X_1+P_1C_2^{\bar\Th_2}\h X_2,\qq\wt Z^M_1=\L^M_1\h X_1.$$
	Then by orthogonal decomposition,
	$$\wt Z_1=P_1C_1^{\bar\Th_1}\h X_1,\qq\wt Z_2=P_1C_2^{\bar\Th_2}\h X_2,$$
	and
	$$d\wt Y_1=-\((A^{\bar\Th_1}_1)^\top\wt Y_1+(C_1^{\bar\Th_1})^\top\wt Z_1+Q_1^{\bar\Th_1}\h X_1\)dt+\wt ZdW+\wt Z^M_1dM.$$
	Likewise,
	$$\ba{ll}
	\ns\ds d\wt Y_2=d(P_2\h X_2)=\[-\(P_2A_2^{\bar\Th_2}+(A_2^{\bar\Th_2})^\top P_2+(C_2^{\bar\Th_2})^\top  P_1C_2^{\bar\Th_2}+Q_2^{\bar\Th_2}\)\h X_2+ P_2A_2^{\bar\Th_2}\h X_2\]dt+\L^M_2\h X_2dM\\
	\ns\ds\qq=-\((A_2^{\bar\Th_2})^\top\wt Y_2+(C_2^{\bar\Th_2})^\top\wt Z_2+Q_2^{\bar\Th_2}\h X_2\)dt+\wt Z_2^MdM,\ea$$
	with
	$$\wt Z_2^M=\L_2^M\h X_2.$$
	Thus, $(\wt Y(\cd),\wt Z(\cd),\wt Z^M(\cd))$ is a predictable solution of the BSDE in \rf{FBSDE4} for $(\h Y(\cd),\h Z(\cd),\h Z^M(\cd))$. By the uniqueness, we obtain
	$$\h Y_i=\wt Y_i=P_i\h X_i,\qq\h Z_i=\wt Z_i=P_1C_i^{\bar\Th_i}\h X_i,\qq
	\h Z_i^M=\wt Z_i^M=\L^M_i\h X_i.$$
	Next, from the stationarity condition in \rf{FBSDE4}, one has
	$$\ba{ll}
	\ns\ds0=B_i^\top\h Y_i+D_i^\top\h Z_i+S_i^{\bar\Th_i}\h X_i=(B_i^\top P_i+D_i^\top P_1C_i^{\bar\Th_i}+S_i^{\bar\Th_i})\h X_i\\
	\ns\ds\q=\big[B_i^\top P_i+D_i^\top P_1C_i+S_i+(R_i+D_i^\top P_1D_i)\bar\Th_i\big]\h X_i=\big[\cS_i(P_i,P_1)+\cR_i(P_1)\bar\Th_i\big]\h X_i,\\
	\ns\ds\qq\qq\qq\qq\qq\qq\qq\qq\qq\qq\qq\ae t\in[s,T],~\as\ea$$
	Since the above holds for all $(s,\xi)\in\sD$, by making use of Proposition \ref{6.2} (iii), we have
	\bel{S+RTh=0}\cS_i\big(P_i(t),P_1(t)\big)+\cR_i\big(P_1(t)\big)\bar\Th_i(t)=0,\qq\ae t\in[s,T],~\as\ee
	Consequently, we obtain the range inclusion conditions:
	$$\sR\(\cS_i\big(P_i(t),P_1(t)\big)\)\subseteq\sR\(\cR_i\big(P_1(t)\big)\),$$
	and
	$$\bar\Th_i(t)=-\cR_i\big(P_1(t)\big)^\dag\cS_i\big( P_i(t),P_1(t)\big)+\big[I-\cR_i\big(P_1(t)\big)^\dag\cR_i\big(P_1(t)\big)\big]\Th_{i0}(t),$$
	for some $\Th_{i0}(\cd)$. This results in (see Appendix B)
	$$P_iA_i^{\bar\Th_i}+(A_i^{\bar\Th_i})^\top P_i+(C_i^{\bar\Th_i})^\top P_1C_i^{\bar\Th_i}+Q_i^{\bar\Th_i}=\cQ_i(P_i,P_1)-\cS_i(P_i,P_1)\cR_i(P_1)^\dag\cS_i(P_i,P_1),$$
	and
	\bel{cost1*}\ba{ll}
	\ns\ds J^0(s,\bar\BTh(\cd),v(\cd))=\dbE\Big\{\lan P_1(s)\xi_1,\xi_1\ran+\lan P_2(s)\xi_2,\xi_2\ran+\int_s^T\[\lan(\cR_1(P_1)v_1,v_1\ran+\lan(\cR_2(P_1)v_2,v_2\ran\]dt\Big\}.\ea\ee
	Now, by the optimality of control $0$ for the homogeneous Problem (MF-LQ)$^0$, we must have the nonnegativity of $\cR_i(P)$ i.e., \rf{cR>0} holds. Then $(P_i(\cd),\L_i^M(\cd))$ is a regular predictable solution of  Riccati differential equations \rf{Riccati2}.
	
	\ms
	
	\it Step 2. \rm Solvability of BSDEs \rf{eta*}.
	
	\ms
	
	Set
	$$\eta_i(t)=\bar Y_i(t)-P_i(t)\bar X_i(t),\qq t\in[s,T],\q i=1,2.$$
	By It\^o's formula, we have
	$$\ba{ll}
	\ns\ds d\eta_1(t)=-\big[(A_1^{\bar\Th_1})^\top\bar Y_1+(C_1^{\bar\Th_1})^\top\bar Z_1+Q_1^{\bar\Th_1}\bar X_1+(S_1^{\bar\Th_1})^\top\bar v_1+q_1^{\bar\Th_1}\big]dt+\bar ZdW+\bar Z_1^MdM\\
	\ns\ds\qq\qq+\(P_1A_1^{\bar\Th_1}+(A^{\bar\Th_1}_1)^\top  P_1+(C_1^{\bar\Th_1})^\top
	P_1C_1^{\bar\Th_1}+Q_1^{\bar\Th_1}\)\bar X_1dt-\L_1^M \bar X_1dM\\
	\ns\ds\qq\qq-P_1(A_1^{\bar\Th_1}\bar X_1+B_1\bar v_1+b_1)dt-P_1(C_1^{\bar\Th_1}\bar X_1+C_2^{\bar\Th_2}\bar X_2+D_1\bar v_1+D_2\bar v_2+\si)dW\\
	\ns\ds=-\((A_1^{\bar \Th_1})^\top(\bar Y_1-P_1\bar X_1)+(C_1^{\bar\Th_1})^\top(\bar Z_1-P_1C_1^{\bar\Th_1}\bar X_1)+\big((S_1^{\bar\Th_1})^\top+P_1B_1\big)\bar v_1+q_1^{\bar\Th_1}+P_1b_1\)dt\\
	\ns\ds\qq\qq+\(\bar Z-P_1(C_1^{\bar\Th_1}\bar X_1+C_2^{\bar\Th_2}\bar X_2+D_1\bar v_1+D_2\bar v_2+\si)\)dW+\(\bar Z_1^M-\L_1^MX_1\)dM\\
	\ns\ds=-\((A_1^{\bar\Th_1})^\top\eta_1+(C_1^{\bar\Th_1})^\top[\Pi_1[\z_1]
	+P_1D_1\bar v_1+P_1\sigma_1]+\big((S_1^{\bar\Th_1})^\top+P_1B_1\big)\bar v_1+q_1^{\bar\Th_1}+P_1b_1\)dt\\
	\ns\ds\qq\qq+\z_1dW+\z_1^MdM\\
	\ns\ds=-\((A_1^{\bar\Th_1})^\top\eta_1+(C_1^{\bar\Th_1})^\top \Pi_1[\z_1]+[P_1B_1+(C_1^{\bar\Th_1})^\top P_1 D_1+(S_1^{\bar\Th_1})^\top]\bar v_1+q_1^{\bar\Th_1}+P_1b_1+(C_1^{\bar\Th_1})^\top P_1\sigma_1\)dt\\
	\ns\ds\qq\qq+\z_1dW+\z_1^MdM\\
	\ns\ds=-\(A_1^\top\eta_1+C_1^\top\Pi_1[\z_1]+(P_1B_1+C_1^\top P_1D_1+S_1^\top)\bar v_1+q_1+P_1b_1+C_1^\top P_1\sigma_1\)dt\\
	\ns\ds\q-\bar\Th_1^\top( B_1^\top\eta_1+D_1^\top\Pi_1[\z_1]+D_1^\top P_1D_1\bar v_1+R_1\bar v_1+r_1+D_1^\top P_1\si_1)dt+\z_1dW+\z_1^MdM,\ea$$
	where
	\bel{z_1}\z_1=\bar Z-P_1(C_1^{\bar\Th_1}\bar X_1+C_2^{\bar\Th_2}\bar X_2+D_1\bar v_1+D_2\bar v_2+\si),\q\z_1^M=\bar Z_1^M-\L_1^MX_1.\ee
	Likewise,
	$$\ba{ll}
	\ds d\eta_2(t)=-\big[(A_2^{\bar\Th_2})^\top\bar Y_2+(C_2^{\bar\Th_2})^\top\bar Z_2+Q_2^{\bar\Th_2}\bar X_2+(S_2^{\bar\Th_2})^\top\bar v_2+q^{\bar\Th_2}_2\big]dt+\bar Z_2^MdM\\
	\ns\ds\q+\big[(A_2^{\bar\Th_2})^\top P_2+ P_2A_2^{\bar\Th_2}+(C_2^{\bar\Th_2})^\top  P_1C_2^{\bar\Th_2}+Q_2^{\bar\Th_2}\big]\bar X_2dt+ \L^M_2 \bar X_2dM-P_2(A_2^{\bar\Th_2}\bar X_2+B_2\bar v_2+b_2)dt\\
	\ns\ds=-\((A_2^{\bar\Th_2})^\top(\bar Y_2-P_2\bar X_2)+(C_2^{\bar\Th_2})^\top(\bar Z_2-P_1C_2^{\bar\Th_2}\bar X_2)+[(S_2^{\bar\Th_2})^\top+P_2B_2]\bar v_2+P_2b_2+q^{\bar\Th_2}_2\)dt\\
	\ns\ds\qq\qq\qq+(\bar Z_2^M-\L_2^M\bar X_2)dM\\
	\ns\ds=-\((A_2^{\bar\Th_2})^\top\eta_2+(C_2^{\bar\Th_2})^\top(\Pi_2[\z_1]+P_1D_2\bar v_2+P_1\si_2)+[(S_2^{\bar\Th_2})^\top+P_2B_2]\bar v_2+P_2b_2+q^{\bar\Th_2}_2\)dt+\z_2^MdM\\
	\ns\ds=-\[A_2^\top\eta_2+C_2^\top(\Pi_2[\z_1]+P_1D_2\bar v_2+P_1\si_2)+(P_2B_2+S_2^\top)\bar v_2+P_2b_2+q_2\\
	\ns\ds\qq\qq\bar\Th_2^\top\(B_2^\top\eta_2+D_2^\top(\Pi_2[\z_1]+D_2^\top P_1D_2\bar v_2+D_2^\top P_1\si_2)+R_2\bar v_2+r_2\)\]dt+\z_2^MdM\\
	\ns\ds=-\[A_2^\top\eta_2+C_2^\top\Pi_2[\z_1]+(P_2B_2+C_2^\top P_1D_2+S_2^\top)\bar v_2+C_2^\top P_1\si_2+P_2b_2+q_2\\
	\ns\ds\qq\qq\bar\Th_2^\top\(B_2^\top\eta_2+D_2^\top\Pi_2[\z_1]+(R_2+D_2^\top P_1D_2)\bar v_2+D_2^\top P_1\si_2+r_2\)\]dt+\z_2^MdM.\ea$$
	Note that
	\bel{Beta=0}\ba{ll}
	\ds B_i^\top\eta_i+D_i^\top\Pi_i[\z_1]+(R_i+D_i^\top P_1D_i)\bar v_i+D_i^\top P_1\si_i+r_i\\
	\ns\ds=B_i^\top (\bar Y_i-P_i\bar X_i)+D_i^\top[\bar Z_i-P_1C_i^{\bar\Th_i}\bar X_i-P_1D_i\bar v_i-P_1\si_i]+(R_i+D_i^\top P_1D_i)\bar v_i+D_i^\top P_1\si_i+r_i\\
	%
	%
	%
	\ns\ds=B_i^\top\bar Y_i+D_i^\top\bar Z_i+S_i^{\bar\Th_i}\bar X_i+R_i\bar v_i+r_i-\big[\cS_i(P_i)^\top+\cR_i(P_1)\bar\Th_i\big]\bar X_i=0.\ea\ee
	The first term is $0$ because of the stationary condition (see \rf{FBSDE3}, line 5) and the second term is $0$ because of \rf{S+RTh=0}. Hence, the above calculation yields that
	$$d\eta_1=-\(A_1^\top\eta_1+C_1^\top\Pi_1[\z_1]+P_1b_1+C_1^\top P_1\si_1+q_1+\cS_1(P_1)^\top\bar v_1\)dt+\z_1dW+\z_1^MdM,$$
	and
	$$d\eta_2(t)=-\(A_2^\top\eta_2+C_2^\top\Pi_2[\z_1]+P_2b_2+C^\top_2P_1\si_2+q_2
	+\cS_2(P_2)^\top\bar v_2\)dt+\z_2^MdM.$$
	On the other hand,
	$$\ba{ll}
	\ns\ds0=B_i^\top\bar Y_i+D_i^\top\bar Z_i+S_i^{\bar\Th_i}\bar X_i+R_i\bar v_i+r_i\\
	\ns\ds\q=B_i^\top(\eta_i+P_i\bar X_i)+D_i^\top\(\Pi_i[\z_1]+P_1C_i^{\bar\Th_i}\bar X_i+D_i\bar v_i+\si_i)\)+S_i^{\bar\Th_i}\bar X_i+R_i\bar v_i+r_i\\
	\ns\ds\q=(B_i^\top P_i+D_i^\top P_1C_i^{\bar\Th_i}+S_i^{\bar\Th_i})\bar X_i+B_i^\top\eta_i+D_i^\top\Pi_i[\z_1]+D_i^\top P_1\si_i+r_1+(R_i+D_i^\top P_1D_i)\bar v_1\\
	\ns\ds\q=B_i^\top\eta_i+D_i^\top\Pi_i[\z_1]+D_i^\top P_1\si_i+r_i+\cR_i(P_1)\bar v_i.\ea$$
	Thus the conclusion in \rf{B_1eta_1*} follows and the equation in \rf{eta*} hold.
	This proves this step. \endpf
	
	\ms
	
	We should point out that the similar problem with constant coefficients was studied in Li-Sun-Yong (\cite{Li-Sun-Yong-2016}). A different method was used in proving the necessity. That method seems to be difficult to extend here.
	
	\ms
	
	In above theorem, we characterized the existence of an optimal closed-loop strategy in terms of the regular solvability of the BSDREs. The natural question is to ask when those BSDREs admit regular solutions? For the open-loop optimal control case,  (H4.1) is adopted for the existence (and uniqueness) of the open-loop optimal control. We expect that such an assumption (together with (H2.1)) also ensures the existence of a closed-loop optimal strategy. This will be established in the next section.
	
	\section{Strongly Regular Solvability of BSDREs}
	
	We begin this section with the following lemma.
	
	\bl{6.5} \sl Let {\rm (H2.1)} hold. Then for any $(\Th_1(\cd),\Th_2(\cd))\in\hTh[s,T]$, there exist unique predictable solutions $(P_i,\L^M_i)\in L^\i_{\dbF^M}(s,T;\dbS^n)\times L^2_{\dbF_-^M}(s,T;\dbS^n)$ to the following BSDEs
	\bel{BSDE-P_i}\left\{\2n\ba{ll}
	\ds dP_i(t)=-\(P_iA_i^{\Th_i}+(A_i^{\Th_i})^\top P_i\1n+\1n(C_i^{\Th_i})^\top P_1C_i^{\Th_i}\1n+\1n Q_i^{\Th_i}\)dt+\L_i^M(t)dM(t),\\
	\ns\ds P_i(T)=G_i,\ea\right.\ee
	where $A_i^{\Th_i}$ and $C_i^{\Th_i}$ are defined by \rf{A^Th}. Moreover, for any $v_1(\cd)\oplus v_2(\cd)\in\sU[s,T]$, we have
	\bel{repJT0}\ba{ll}
	\ns\ds J^0(s,\xi_1,\xi_2;\Th_1,\Th_2,v_1,v_2 )\equiv J^0(s,\xi_1,\xi_2,\Th_1X_1+v_1,\Th_2X_2+v_2)\\
	\ns\ds= \sum_{i=1}^2\dbE\[\lan P_i(s)\xi_i,\xi_i\ran+\int_s^T\(\lan \cR_i(P_1)v_i,v_i\ran+2\lan\big(\cS_i(P_i,P_1)+\cR_i(P_1)\Th_i\big)X_i^0,v_i\ran\)dt\].\ea\ee
	Further, if, in addition, {\rm(H4.1)} holds, then
	\bel{R+DPD}\cR_i(P_1)\ges\d I,\qq P_i\ges-K_0,\qq\as~\ae t\in[s,T],\q i=1,2,\ee
	where $\d>0$ can be taken the same as that in \rf{uconvexity} and $K_0>0$ can be taken as that in \rf{|V^0|}.
	
	\el
	
	\it Proof. \rm Note that \eqref{BSDE-P_i} is a linear BSDE driven by a c\`adl\`ag martingale with bounded coefficients, under (H2.1) and for any $\BTh(\cd)\equiv(\Th_1(\cd),\Th_2(\cd))\in\hTh[s,T]$. Thus, the existence and uniqueness of the predictable solution can be seen from \cite{Carbone-2008}. Moreover, $P_1$ and $P_2$ are bounded almost surely and almost everywhere $t\in[s,T]$. Now, for the given $(\Th_1(\cd),\Th_2(\cd))\in\hTh[s,T]$ and any $(v_1(\cd),v_2(\cd))\in\sU[s,T]$,
	applying It\^o's formula  (see \eqref{itoRC-2} and \rf{cost1*}), we have
	\rf{repJT0}.
	
	\ms
	
	Further, let \eqref{uconvexity} hold additionally. Setting $(\xi_1,\xi_2)=(0,0)$, and denoting $(X_1^{0,0}(\cd),X_2^{0,0}(\cd))$ the state corresponding to the homogeneous state equation under $(\BTh(\cd),v(\cd))$, using \rf{repJT0},
	we get
	$$\ba{ll}
	\ns\ds\sum_{i=1}^2\d\dbE\int_s^T|\Th_i(t)X_i^{0,0}(t)+v_i(t)|^2dt\les J^0(s,0,0;\Th_1X_1^{0,0}+v_1,\Th_2X_2^{0,0}+v_2)\\
	\ns\ds=\sum_{i=1}^2\dbE\int_s^T\(\lan\cR_i(P_1) v_i,v_i\ran+2\lan(\cS_i(P_i,P_1)+  \cR_i(P_1) \Th_i) X_i^{0,0}(t),v_i\ran\)dt.\ea$$
	Hence,
	$$\ba{ll}
	\ns\ds\sum_{i=1}^2\dbE\int_s^T\(\lan(\cR_i(P_1)-\d I)v_i,v_i\ran
	+2\lan[\cS_i(P_i,P_1)+(\cR_i(P_i)-\d I)\Th_i] X_i^{0,0}(t),v_i\ran\)dt\\
	\ns\ds\ges\sum_{i=1}^2\d\dbE\int_s^T\(|\Th_i(t)X^{0,0}_i(t)+v_i(t)|^2+|v_i|^2
	-2\lan\Th_iX^{0,0}_i(t)+v_i(t),v_i(t)\ran\)dt\\
	\ns\ds=\sum_{i=1}^2\d\dbE\int_s^T|\Th_i(t)X_i^{0,0}(t)|^2dt\ges0.\ea$$
	Further, for $t\in[s,T]$, $h>0$ and any $\n_1,\n_2\in L^2_{\cF_t}(\O;\dbR^m)$, setting $v_1(r)=\n_1{\bf I}_{[t,t+h]}(r)$, $v_2(r)=\n_2{\bf I}_{[t,t+h]}(r)$, we know
	\bel{SDE-2}\left\{\2n\ba{ll}
	\ns\ds d\dbE_r^M[X_1^{0,0}(r) ]=\( A_1^{\Th_1}(r)  \dbE_r^M[X_1^{0,0}(r)]+B_1(r)\n_1{\bf I}_{[t,t+h]}(r)\)dr,\q r\in[s,T], \\
	\ns\ds d X_2^{0,0}(r)=\(A_2^{\Th_2}(r)X_2^{0,0}(r)+B_2(r)\n_2{\bf I}_{[t,t+h]}(r)\)dr,\q r\in[s,T],\\
	\ns\ds X_1^{0,0}(s)=0,\qq X_2^{0,0}(s)=0.\ea\right.\ee
	Then,
	$$\ba{ll}\dbE_r^M[X_1^{0,0}(r) ]=\left\{\2n\ba{ll}
	\ds0,\qq\qq\qq\qq\qq\qq\qq\q\ r\in[s,t),\\
	\ns\ds\F_1(r)\int_t^{r\land(t+h)}\Phi_1(\t)^{-1}B_1(\t)\n_1d\t,\q r\in[t,T],   \ea\right.\\
	\ns\ds X_2^{0,0}(r)=\left\{\2n\ba{ll}0,\qq\qq\qq\qq\qq\qq\qq\q\   r\in[s,t),\\ \ns\ds \Phi_2(r)\int_t^{r\land(t+h)}\Phi_2(\t)^{-1}B_2(\t)\n_2d\t,\q r\in[t,T],   \ea\right.\ea$$
	where
	$$d\F_i(r)=A_i^{\Th_i}(r)\F_i(r)dr,\q\F_1(s)=I.$$
	Consequently,
	$$\ba{ll}
	\ns\ds\sum_{i=1}^2\dbE\int_t^{t+h}\Big[\lan(\cR_i(P_1)-\d I)\n_i,\n_i\ran +2\lan[\cS_i(P_i,P_1)+(\cR_i(P_1)-\d I)\Th_i]X_i^{0,0}(r),\n_i\ran \Big]dr\\
	\ns\ds=\sum_{i=1}^2\int_t^{t+h}\dbE\Big[\lan(\cR_i(P_1)-\d I)\n_i,\n_i\ran
	+2\lan[\cS_i(P_i,P_1)+(\cR_i(P_1)-\d I)\Th_i]\dbE_r^M[X_i^{0,0}(r) ],\n_i\ran \Big]dr\\
	\ns\ds=\sum_{i=1}^2\2n\int_t^{t+h}\2n\dbE\Big[\lan(\cR_i(P_1)-\d I)\n_i,\n_i\ran+2\lan[\cS_i(P_i,P_1)+(\cR_i(P_1)-\d I)\Th_i]\Phi_i(r)\int_t^{r}\Phi_i(\t)^{-1}B_i\n_id\t,\n_i\ran \Big]dr
	\1n\ges\1n0.\ea$$
	Noting that
	$$\ba{ll}
	\ns\ds\Big|\dbE\int_t^{t+h}\[2\lan[\cS_i(P_i,P_1)+(\cR_i(P_1)-\d I)\Th_i]\Phi_i(r)\int_t^{r}\Phi_i(\t)^{-1}B_i\n_id\t,\n_i\ran \Big]dr\1n\Big|\les K\2n\int_t^{t+h}\3n
	\int_t^rd\t dr\les Kh^2,\ea$$
	for some constant $K>0$. Therefore, dividing the above by $h$ and letting $h\to0$, we get
	$$\dbE\Big[\lan(\cR_i(P_1)(t)-\d I)\n_i,\n_i\ran\Big]\ges 0 ,\q t\in[s,T],\qq i=1,2,$$
	which implies \eqref{R+DPD}. Moreover, by \rf{|V^0|}, for any $(s,\xi)\in\sD$, we have
	$$-K_0\dbE(|\xi_1|^2+|\xi_2|^2)\les V^0(s,\xi_1,\xi_2)\les J^0(s,\xi_1,\xi_2;0,0) =\sum_{i=1}^2\dbE\lan P_i(s)\xi_i,\xi_i\ran.$$
	Therefore, $P_i\ges-K_0$. \endpf
	
	\ms		
	
	With the above preparations, we now present the following crucial result.

	\bl{7.2}\sl Let {\rm (H2.1)} hold, Then {\rm(H4.1)} hold if and only if BSDRE \eqref{Riccati1} admits a strongly regular solution.
	\el

	\it Proof. \rm {\it Sufficiency.} Let $(P_i(\cd),\L^M_i(\cd))\in L^\i_{\dbF^M}(s,T;\dbS^n)\times L^2_{\dbF^M_-}(s,T;\dbS^n)$ be a strongly regular solution of \rf{Riccati1}. Thus, there exists a $\d>0$ such that
	$$\cR_i(P_1)\ges\d I,\qq\as,\ae t\in[s,T],\q i=1,2.$$
	Let
	$$(\Th_1,\Th_2)=(-\cR_1(P_1)^{-1}\cS_1(P_1,P_1),-\cR_2(P_1)^{-1}\cS_2(P_2,P_1))\in \hTh[s,T].$$
	For any $u(\cd)=u_1(\cd)+u_2(\cd)\in\sU[s,T]$, let
	$$v_i(\cd)=u_i(\cd)-\Th_i(\cd)X_i^{0,0}(t),\qq i=1,2.$$
	Then, by Lemma \ref{6.5} and Lemma 2.3 in \cite{Sun-2016}, we have (noting $\xi_1+\xi_2=0$)
	\bel{repJT0*}\ba{ll}
	\ns\ds  J^0(s,0,0;u_1(\cd),u_2(\cd))=J^0(s,0,0;\Th_1,\Th_2,v_1,v_2 )\\
	\ns\ds= \sum_{i=1}^2\[\dbE\int_s^T\(\lan \cR_i(P_1)v_i,v_i\ran+2\lan\big(\cS_i(P_i,P_1)+\cR_i(P_1) \Th_i \big)X_i^{0,0},v_i \ran\)dt\]\\
	\ns\ds=\sum_{i=1}^2\dbE\int_s^T\lan\cR_i(P_i)(u_i-\Th_i X_i^{0,0}),u_i-\Th_i X_i^{0,0}\ran dt\ges\delta\g\dbE\int_s^T\big[|u_1|^2+|u_2|^2\big]dt\ea\ee	
	for some $ \gamma>0$. This proves the uniform convexity of $J^0(s,0,0;u_1(\cd),u_2(\cd))$, namely, (H4.1).
	
	\ms
	
	{\it Necessity.} Suppose (H4.1) holds. We consider the following sequence of BSDREs parameterized by $k=1,2,\cds$,

	\bel{dP-k}\left\{\2n\ba{ll}
	\ns\ds dP_i^k(t)=-\G_i^{k-1}(P_i^k,P_1^k)dt+\L_i^{M,k}dM, \q t\in[s,T],\\
	\ns\ds P_i^k(T)=G_i,\ea\right.\ee
	where
	\bel{G_i^k}\G_i^{k-1}(P_i,P_1)=P_iA_i^{\Th_i^{k-1}}+(A_i^{\Th_i^{k-1}})^\top P_i+(C_i^{\Th_i^{k-1}})^\top P_1C_i^{\Th_i^{k-1}}+Q_i^{\Th_i^{k-1}},\ee
	with $Q_i^{\Th_i^{k-1}}$ being defined by \rf{notation-QRS} and
	\bel{6.36}\left\{\2n\ba{ll}
	\ns\ds\Th_i^0=0,\\
	\ns\ds\Th^k_ip:=-(R_i+D_i^\top  P_1^kD_i)^{-1}(B_i^\top P_i^k+D_i^\top P_1^kC_i+S_i)\equiv\cR_i(P_1^k)^{-1}\cS_i(P_i^k,P_1^k).\ea\right.\ee
	By Lemma \ref{6.5}, when $\Th_i^0=0$, \rf{dP-k} admits a unique predictable solution $(P_i^1,\L_i^{M,1})\in L^\i_{\dbF^M}(0,T;\dbS^{ n})\times L^2_{\dbF^M_-}(0,T;\dbS^{  n})$, and
	\bel{R+DPD1}\cR_i(P_1^1)\ges\d I,\qq P_i^1\ges-K_0 I,\qq\as,~\ae t\in[s,T] \ee
	with $\d,K_0>0$. Once $(P^{k-1}_i,\L_i^{M,k-1})$ ($i=1,2$) is determined, we clearly have
	$$\ba{ll}
	\ns\ds\Th^k_1=-\cR_1(P^k_1) ^{-1}\cS_1(P_1^k,P_1^k)\in L^\i_{\dbF^M}(s,T;\dbR^{m\times n})^\perp,\\
	\ns\ds\Th^k_2=-\cR_2(P^k_1) ^{-1}\cS_2(P_2^k,P_1^k)\in L^\i_{\dbF^M}(s,T;\dbR^{m\times n}),\ea\qq k\ges1.$$
	Using Lemma \ref{6.5}, by induction, we can get a predictable solution $(P_i^k,\L_i^{M,k})$ to the BSDEs so that
	\bel{cR(P^k)>d}\cR_i( P^k_1)\ges\d I,\qq\as,~\ae t\in[s,T],\q k\ges1.\ee
	Next, we shall show the convergence of the sequence   $\{(P_i^k(\cd),\L^{M,k}_i(\cd),\Th_i^k(\cd))\}_{i\ges 1}$.
	To this end, we observe
	$$\ba{ll}
	\ns\ds\G_i^{k-1}(P_i^k,P_1^k)=P_i^kA_i^{\Th_i^{k-1}}+(A_i^{\Th_i^{k-1}})^\top P_i^k+(C_i^{\Th_i^{k-1}})^\top P_1^kC_i^{\Th_i^{k-1}}+Q_i^{\Th_i^{k-1}} \\
	\ns\ds=P_i^kA_i^{\Th_i^k}+(A_i^{\Th_i^k})^\top P_i^k+(C_i^{\Th_i^k})^\top P_1^kC_i^{\Th_i^k}+Q_i^{\Th_i^k}\\
	\ns\ds\q+\[\(P_i^kA_i^{\Th_i^{k-1}}+(A_i^{\Th_i^{k-1}})^\top P_i^k+(C_i^{\Th_i^{k-1}})^\top P_1^kC_i^{\Th_i^{k-1}}+Q_i^{\Th_i^{k-1}}\)\\
	\ns\ds\q-\(P_i^kA_i^{\Th_i^k}+(A_i^{\Th_i^k})^\top P_i^k+(C_i^{\Th_i^k})^\top P_1^kC_i^{\Th_i^k}+Q_i^{\Th_i^k}\)\]\\
	\ns\ds\equiv\G_i^k(P_i^k,P_1^k)+\Upsilon_i^k,\ea$$
	where
	$$\ba{ll}
	\ns\ds\Upsilon_i^k\equiv\G_1^{k-1}(P_i^k,P_1^k)-\G_i^k(P_i^k,P_1^k)=\(P_i^kA_i^{\Th_i^{k-1}}+(A_i^{\Th_i^{k-1}})^\top P_i^k+(C_i^{\Th_i^{k-1}})^\top P_1^kC_i^{\Th_i^{k-1}}+Q_i^{\Th_i^{k-1}}\)\\
	\ns\ds\q-\(P_i^kA_i^{\Th_i^k}+(A_i^{\Th_i^k})^\top P_i^k+(C_i^{\Th_i^k})^\top P_1^kC_i^{\Th_i^k}+Q_i^{\Th_i^k}\)\\
	\ns\ds=(\Th_i^{k-1})^\top(B_i^\top P_i^k+D_i^\top P_1^kC_i+S_i)
	+(P_i^kB_i+C_i^\top P_i^kD_i+S_i^\top)\Th_i^{k-1}+(\Th_1^{k-1})^\top(R_i
	+D_i^\top P_1^kD_i)\Th_i^{k-1}\\
	\ns\ds\q-\Big[(\Th_i^k)^\top(B_i^\top P_i^k+D_i^\top
	P_1^kC_i+S_i)+(P_i^kB_i+C_i^\top P_1^kD_i+S_i^\top)
	\Th_i^k+(\Th_i^k)^\top(R_i+D_i^\top P_1^kD_i)
	\Th_i^k\Big]\\
	\ns\ds\equiv(\Th^{k-1}_i)^\top\cS_i+\cS_i^\top\Th_i^{k-1}
	+(\Th_i^{k-1})^\top\cR_i\Th_i^{k-1}-\((\Th^k_i)^\top\cS_i+\cS_i^\top\Th_i^k
	+(\Th_i^k)^\top\cR_i\Th_i^k\)\\
	\ns\ds=(\Th_i^{k-1}+\cR_i^{-1}\cS_i)^\top\cR_i(\Th_i^{k-1}+\cR_1^{-1}\cS_1)-
	(\Th_i^k+\cR_i^{-1}\cS_i)^\top\cR_i(\Th_i^k+\cR_1^{-1}\cS_1)\\
	\ns\ds=(\Th_i^{k-1}+\cR_i^{-1}\cS_i)^\top\cR_i(\Th_i^{k-1}+\cR_1^{-1}\cS_1)\ges0.\ea$$
	Here, we have used the definition of $\Th^k_i$ (see \rf{6.36}). We thus have proved
	$$\G_i^{k-1}(P_i^k,P_1^k)\ges\G_i^k(P_i^k,P_1^k).$$
	Now, comparing \rf{dP-k} with the following
	\bel{dP-k+1}\left\{\2n\ba{ll}
	\ns\ds dP_i^{k+1}(t)=-\G_i^k(P_i^{k+1},P_1^{k+1})dt+\L_i^{M,k+1}dM, \q t\in[s,T],\\
	\ns\ds P_i^{k+1}(T)=G_i,\ea\right.\ee
	Since
	$$\ba{ll}
	\ns\ds\G_i^{k-1}(P_i^k,P_1^k)-\G_i^k(P_i^{k+1},P_1^{k+1})=\Upsilon_i^k+\G_i^k(P_i^k,P_1^k)\
	-\G_i^k(P_i^{k+1},P_1^{k+1})\\
	\ns\ds=\Upsilon_i^k+(P_i^k-P_i^{k+1})A_i^{\Th_i^k}+(A_i^{\Th_i^k})^\top (P_i^k-P_i^{k+1})+(C_i^{\Th_i^k})^\top(P_1^k-P_i^{k+1})
	C_i^{\Th_i^k},\ea$$
	we see that $\h P_i^k=P_i^k-P_i^{k+1}$, $\h\L_i^{M,k}=\L_i^{M,k}
	-\L_i^{M,k+1}$ satisfies the following linear BSDE:
	$$\left\{\2n\ba{ll}
	\ns\ds d\h P_i^k=-\(\h P_i^kA_i^{\Th_i^k}+(A_i^{\Th_i^k})^\top\h P_i^k+(C_i^{\Th_i^k})^\top\h P_1^kC_i^{\Th_i^k}+\Upsilon_i^k\)dt+\h\L_i^{M,k},\\
	\ns\ds\h P_i^k(T)=0,\ea\right.$$
	with $\Upsilon_i^k\ges0$. Then we obtain
	$$P_i^k(t)-P_i^{k+1}(t)=\h P_i^k(t)\ges0,\qq k\ges1.$$
	Hence, one has
	$$-K_0I\les P_i^k(t)\les P_i^1(t),\qq k\ges1.$$
	Since $P_i^1(\cd)$ is uniformly bounded, so are the whole sequences $\{P_i^k(\cd)\}_{k\ges1}$ ($i=1,2$). Then by their monotonicity and the dominated convergence theorem, we have the convergence of $P_i^k(\cd)$ to some $P_i(\cd)$ almost surely, almost everywhere on $\O\times[s,T]$ and in the space $L^2_{\dbF^M}(s,T;\dbS^n)$.
	On the other hand, we have $\cR_i(P_1^k)\ges\d I$ uniformly for every $k\ges1$. Thus, by the convergence of $P_i^k(\cd)$ and the definition of $\Th_i^k(\cd)$, we must have
	$$\lim_{k\to\i}\Th_i^k(t)=\Th_i(t)=-\cR_1(P_i(t))^{-1}\cS_i(P_i(t),P_1(t)).$$
	Clearly $(\Th_1(\cd),\Th_2(\cd))\in\hTh[s,T]$. Finally, it is routine that
	$$\ba{ll}
	\ns\ds\dbE\int_s^T|\L_i^{M,k}(t)-\L_i^{M.k'}(t)|^2d\lan M\ran(t)=\dbE\Big|\int_s^T\(\L_i^{M,k}(t)-\L_i^{M,k'}(t)\)dM\Big|^2\\
	\ns\ds=\dbE\Big|\int_s^T\(\G_i^{k-1}(P_i^k,P_1^k)-\G_i^{k'-1}(P_i^{k'},P_1^{k'})\)dt\Big|^2.\ea$$
	The right-hand side goes to zero as $k,k'\to\i$ since $P_i^k(\cd)$ is Cauchy in the right space. Therefore, $\L_i^{M,k}(\cd)$ converges to some $\L_i^M(\cd)$. The proof is complete. \endpf
	
	\ms
	
	We now present the following corollary.

	\bc{convexclos}\sl Let {\rm(H2.1)} and {\rm(H4.1)} hold. Then, Problem {\rm (MF-LQ)} is closed-loop solvable.
	
	\ec
	
	\it Proof. \rm Under our assumptions, according to Lemma \ref{7.2}, we know that \rf{Riccati2} admits a unique strongly regular solution $(P_1(\cd),\L^M_1(\cd)),$ $(P_2(\cd),\L^M_2(\cd))$ such that \eqref{R+DPD} holds, which implies (i), (ii)  in Theorem \ref{mainthm1} is true. Therefore,  Problem {\rm (MF-LQ)}$_\dbT$ is closed-loop solvable. \endpf

	\section{Concluding Remarks}
	
	In the paper, we have studied linear quadratic optimal control problems for a mean-field stochastic differential equations whose coefficients are adapted to another independent martingale. To deal with the mean-field terms involved, an orthogonal projection is introduced which leads to a new linear optimal control problem on the product of two orthogonal spaces. We begin with the classical approach of completing the square to the LQ problem, which leads to the backward stochastic differential Riccati equation (BSDRE) formally. However, in doing this, it is vague about the relationship among the three key notions: The existence of optimal control, the Pontryagin type maximum principle, and the Riccati equation. Inspired by the works of Sun--Yong (\cite{Sun-Yong-2020a, Sun-Yong-2020b}, which discussed the constant coefficient cases), we look at the open-loop and closed-loop solvability of Problem (MF-LQ) and fully characterize them. Finally, we end up with the result that the uniform convexity of the cost functional implies both the open-loop and closed-loop solvability of Problem (MF-LQ).
	
	\appendix
	
	\section{Some Lemmas}\label{AB}
	
	\bl{infty=t}\sl If $\xi$ is $\cF_s$-measurable and integrable, then
	\bel{E=E}\dbE^M_s[\xi]=\dbE^M_t[\xi],\qq\forall t\ges s.\ee
	Consequently,
	\bel{E=E*}\dbE_s^M[\xi]=\dbE^M[\xi].\ee
\end{lemma}

\it Proof. \rm Set $t\ges s$. Since $\dbF^W$ and $\dbF^M$ are independent, for any  $\G\in\cF^M_t$, we have
\bel{WMin}\dbE[{\bf1}_\G|\cF_s]\equiv\dbE[{\bf1}_\G|\cF_s^M\vee\cF_s^W]=\dbE[{\bf1}_\G|\cF_s^M].\eel
Then noting that $\dbE[{\bf1}_\G|\cF_s^M]$ is $\cF_s^M$-measurable and by setting $\eta=\dbE[\xi|\cF_s^M]$, we have
$$\ba{ll}
\ns\ds\dbE\[\eta{\bf1}_\G\]=\dbE\[\dbE[\xi|\cF_s^M]{\bf1}_\G\]
=\dbE\[\dbE\(\dbE[\xi|\cF_s^M]{\bf1}_\G\bigm|\cF_s^M\)\]\\
\ns\ds=\dbE\[\dbE[\xi|\cF_s^M]\dbE[{\bf1}_\G|\cF_s^M]\]=\dbE\[\dbE\(\xi\dbE[{\bf1}_\G|
\cF_s^M]\bigm|\cF_s^M\)\]=\dbE\[\xi\dbE[{\bf1}_\G|\cF_s^M]\]\qq \text{(by \eqref{WMin}})\\
\ns\ds=\dbE\[\xi\dbE[{\bf1}_\G|\cF_s]\]=\dbE\[\dbE[\xi{\bf1}_\G|\cF_s]\]=\dbE[\xi {\bf1}_\G].\ea$$
Since $\eta$ is $\cF_s^M$-measurable as well, which is also $\cF^M_t$-measurable, by the definition of conditional expectation, we have $\eta=\dbE^M_t[\xi]$. Then \rf{E=E} holds and \rf{E=E*} also follows. \endpf

\begin{lemma}\label{allowito}\sl
	Suppose that  $$(Y,\dot Y,\xi)\in L_{\dbF}^2(0,T;\dbR^{k})\times L_{\dbF}^2(0,T;\dbR^{k})\times L_{\dbF}^2(0,T;\dbR^{k}),$$
	$$( P,\dot{P},\zeta)\in L_{\dbF^M}^\infty(0,T;\dbR^{k\times k})\times L_{\dbF^M}^\infty(0,T;\dbR^{k\times k})\times  L_{\dbF^M_-}^2(0,T;\dbR^{k\times k})$$
	and $$(\wt Y,\dot{\wt Y},\wt\xi,\wt\z)\in L_{\dbF}^2(0,T;\dbR^{k})\times L_{\dbF}^2(0,T;\dbR^{k})\times L_{\dbF}^2(0,T;\dbR^{k})\times L_{\dbF_-}^2(0,T;\dbR^{k})$$
	satisfy
	$$dY=\dot Y dt+\xi dW_t, ~dP=\dot {P} dt+\z dM_t \text{  and  }d\wt Y=\dot {\wt Y} dt+\wt\xi dW_t+\wt\z dM_t.$$ Then  for $t\in[s,T]$ a.e., we have
	\bel{itoRC}\dbE \lan Y(t),\wt Y(t)\ran-\dbE \lan Y(s),\wt Y(s)\ran=\dbE\int_s^t\lan\dot Y,\wt Y\ran+ \lan Y,\dot{\wt Y}\ran+\lan\BBxi,\wt\xi\ran dr,\eel
	and
	\bel{itoRC-2}\dbE\lan P(t)Y(t),Y(t)\ran-\dbE\lan P(s)Y(s),Y(s)\ran=\dbE\int_s^t\lan\dot PY,Y\ran+\lan P\dot Y,Y\ran+\lan PY,\dot Y\ran+\lan P\xi,\xi\ran dr.\ee
\end{lemma}

\it Proof. \rm Without loss of generality, we only show the case for $k=1$.
Let $f_\l:\dbR\mapsto \dbR $ be twice continuously differentiable  such that
$f_\l(y)=y$ for $|y|\leq \l$, $f_\l(y)=\l+1$ for $|y|\geq \l+1$ and
$|f'_\l(\cdot)|,~|f''(\cdot)|$ are bounded uniformly in $\l$. Then applying It\^o's formula on $t\mapsto \lan f_\l(Y_t),\wt Y_t\ran$, we have
$$\ba{ll}\ad d\lan f_\l(Y(t)),\wt Y(t)\ran= [f_\l'(Y)\dot Y \wt Y+f(Y)\dot{\wt Y}+\frac12 f_\l''(Y)\xi^2+f_\l'(Y)\xi\wt\xi]dt\\
\ns\ad\qq+f_\l(Y(t))\wt\z(t) dM(t)+[f_\l'(Y(t))\xi(t)Y(t)+f_\l(Y(t))\wt\xi(t)]dW(t). \ea $$
By the selection of $f_\l$, all the local martingales in above equation turn out be martingales. Therefore we have
$$\dbE \lan f_\l(Y(t)),\wt Y(t)\ran-\dbE \lan f_\l(Y(s)),\wt Y(s)\ran=\dbE\int_s^t[f_\l'(Y)\dot Y \wt Y+f(Y)\dot{\wt Y}+\frac12 f_\l''(Y)\xi^2+f_\l'(Y)\xi\wt\xi]dr.$$
Letting $\l\rightarrow\infty$, by dominant convergence theorem, \eqref{itoRC} holds. The proof for \eqref{itoRC-2} is similar.
\endpf

\section{Some Lengthy Calculations}\label{ABC}

In this appendix, we would like to carry out some lengthy and routine calculations, for the convenience of the readers. To be general enough, let us take
\bel{B-u=ThX+v}\left\{\2n\ba{ll}
\ds u_1(t)=\Th_1(t)X_1(t)+v_1(t),\\
\ns\ds u_2(t)=\Th_2(t)X_2(t)+v_2(t),\ea\right.\qq t\in[s,T],\ee
where
$$\Th_i(\cd)\in L^\i_{\dbF^M}(s,T;\dbR^{m\times n}),\q v_1(\cd)\in L^2_{\dbF^M}(s,T;\dbR^m)^\perp,\q v_2(\cd)\in L^2_{\dbF^M}(s,T;\dbR^m).$$
Under \rf{B-u=ThX+v}, our state equation \rf{SDE1} reads
\bel{B-closed1}\ba{ll}
\ns\ds dX=\Big\{A(X_1+X_2)+\bar AX_2+B(\Th_1X_1+\Th_2X_2+v_1+v_2)+\bar B(\Th_2X_2+v_2)+b\Big\}dt\\
\ns\ds\qq\q+\Big\{C(X_1+X_2)+\bar CX_2+D(\Th_1X_1+\Th_2X_2+v_1+v_2)+\bar D(\Th_2X_2+v_2)+\si\Big\}dW(t)\\
\ns\ds\qq=\Big\{(A+B\Th_1)X_1+[A+\bar A+(B+\bar B)\Th_2]X_2+
Bv_1+(B+\bar B)v_2+b\Big\}dt\\
\ns\ds\qq\q+\Big\{(C+D\Th_1)X_1+[C+\bar C+(D+\bar D)\Th_2]X_2+Dv_1+(D+\bar D)v_2+\si\Big\}dW(t)\\
\ns\ds\qq=\(A_1^{\Th_1}X_1\1n+\1n A_2^{\Th_2}X_2\1n+\1n
B_1v_1+B_2v_2+b\)dt\1n+\1n\(C_1^{\Th_1}X_1+C_2^{\Th_2}X_2+D_1v_1+D_2v_2+\si
\)dW(t),\ea\ee
where the initial state $X(s)=\xi$, and
\bel{B-A_1}\left\{\2n\ba{ll}
\ns\ds A_1(t)=A(t),\q A_2(t)=A(t)+\bar A(t),\q C_1(t)=C(t),\q C_2(t)=C(t)+\bar C(t),\\
\ns\ds B_1(t)=B(t),\q B_2(t)=B(t)+\bar B(t),\q D_1(t)=D(t),\q D_2(t)=D(t)+\bar D(t),\ea\right.\ee
and
\bel{B-A^Th}A_i^{\Th_1}(t)=A_i(t)+B_i(t)\Th_i(t),\q C_i^{\Th_1}(t)=C_i(t)+D_i(t)\Th_i(t),\q i=1,2.\ee
Applying $\Pi\equiv\Pi_2$ to \rf{B-closed1}, we can get an equation for $X_2$. Then by noticing $X_1=X-X_2$, we obtain an equation for $X_1$. Thus, the state equation can be equivalently written as:
\bel{B-closed3}\left\{\2n\ba{ll}
\ns\ds dX_1=(A_1^{\Th_1}X_1+B_1v_1+b_1)dt+(C_1^{\Th_1}X_1+C_2^{\Th_2}X_2+D_1v_1
+D_2v_2+\si)dW(t),\\
\ns\ds dX_2=(A_2^{\Th_2}X_2+B_2v_2+b_2)dt.\ea\right.\ee
Correspondingly, the cost functional reads
\bel{B-cost2}\ba{ll}
\ns\ds J(s,\xi;\BTh(\cd),v(\cd)):=J(s,\xi;\Th_1(\cd)X_1(\cd)+\Th_2(\cd)X_2(\cd)+v(\cd))\\
\ns\ds=\frac12\dbE\Big\{\int_s^T\[\lan QX,X\ran+2\lan SX,\Th_1X_1+\Th_2X_2+v\ran+\lan R(\Th_1X_1+\Th_2X_2+v),\Th_1X_1+\Th_2X_2+v\ran\\
\ns\ds\qq\qq+\lan\bar QX_2,X_2\ran+2\lan\bar SX_2,\Th_2X_2+v_2\ran+\lan\bar R(\Th_2X_2+v_2),\Th_2X_2+v_2\ran\\
\ns\ds\qq\qq+2\lan q,X\ran+2\lan\bar q,X_2\ran+2\lan r,\Th_1X_1+\Th_2X_2+v\ran+2\lan\bar r,\Th_2X_2+v_2\ran\]dt\\
\ns\ds\qq\qq+\lan GX(T),X(T)\ran+\lan\bar GX_2(T),X_2(T)\ran+2\lan g,X(T)\ran+2\lan\bar g,X_2(T)\ran\Big\}\\
\ns\ds=\frac12\dbE\Big\{\int_s^T\[\lan QX_1,X_1\ran+\lan QX_2,X_2\ran+2\lan SX_1,\Th_1X_1+v_1\ran+2\lan SX_2,\Th_2X_2+v_2\ran\\
\ns\ds\qq\qq+\lan R(\Th_1X_1+v_1),\Th_1X_1+v_1\ran+\lan R(\Th_2X_2+v_2),\Th_2X_2+v_2\ran+\lan\bar QX_2,X_2\ran\\
\ns\ds\qq\qq+2\lan\bar SX_2,\Th_2X_2+v_2\ran+\lan\bar R(\Th_2X_2+v_2),\Th_2X_2+v_2\ran+2\lan q_1,X_1\ran+2\lan\Pi_2[q],X_2\ran\\
\ns\ds\qq\qq+2\lan\bar q,X_2\ran+2\lan r_1,\Th_1X_1+v_1\ran+2\lan \Pi_2[r],\Th_2X_2+v_2\ran+2\lan\bar r,\Th_2X_2+v_2\ran\]dt\\
\ns\ds\qq\qq+\lan GX_1(T),X_1(T)\ran+\lan(G+\bar G)X_2(T),X_2(T)\ran+2\lan g_1,X_1(T)\ran+2\lan\Pi_2[g]+\bar g,X_2(T)\ran\Big\}\\
\ns\ds=\frac12\dbE\Big\{\int_s^T\[\lan(Q+\Th_1^\top S+S^\top\Th_1+\Th_1^\top R\Th_1)X_1,X_1\ran+2\lan(S+R\Th_1)X_1,v_1\ran+\lan Rv_1,v_1\ran\\
\ns\ds\qq\qq+2\lan q_1+\Th_1^\top r_1,X_1\ran+2\lan r_1,v_1\ran+2\lan\Pi_2[q]+\bar q+\Th_2^\top(\Pi_2[r]+\bar r),X_2\ran+2\lan\Pi_2[r]+\bar r,v_2\ran\\
\ns\ds\qq\qq+\lan[Q+\bar Q+\Th_2^\top(S+\bar S)+(S+\bar S)^\top\Th_2+\Th_2^\top(R+\bar R)\Th_2]X_2,X_2\ran\\
\ns\ds\qq\qq+2\lan[(S+\bar S)+(R+\bar R)\Th_2]X_2,v_2\ran+\lan(R+\bar R)v_2,v_2\ran\]dt\\
\ns\ds\qq\qq+\lan GX_1(T),X_1(T)\ran+\lan(G+\bar G)X_2(T),X_2(T)\ran+2\lan g,X_1(T)\ran+2\lan \Pi_2[g]+\bar g,X_2(T)\ran\Big\}\\
\ns\ds=\frac12\dbE\Big\{\int_s^T\[\lan Q_1^{\Th_1}X_1,X_1\ran+2\lan S_1^{\Th_1} X_1,v_1\ran+\lan R_1v_1,v_1\ran+2\lan q_1^{\Th_1},X_1\ran+2\lan r_1,v_1\ran\\
\ns\ds\qq\qq\qq+\lan Q_2^{\Th_2}X_2,X_2\ran+2\lan S_2^{\Th_2}X_2,v_2\ran+\lan R_2v_2,v_2\ran+2\lan q_2^{\Th_2},X_2\ran+2\lan r_2,v_2\ran\]dt\\
\ns\ds\qq\qq\qq+\lan G_1X_1(T),X_1(T)\ran+\!2\lan g_1,X_1(T)\ran+\lan G_2X_2(T),X_2(T)\ran+2\lan g_2,X_2(T)\ran\Big\},\ea\ee
where
\bel{B-QRS}\ba{ll}
\ns\ds Q_1(t)=Q(t),\q
Q_2(t)=Q(t)+\bar Q(t),\q S_1(t)=S(t),\q S_2(t)=S(t)+\bar S(t),\\
\ns\ds R_1(t)=R(t),\qq R_2(t)=R(t)+\bar R(t),\\
\ns\ds q_1(t)=\Pi_1[q](t),\q q_2(t)=\Pi_2[q](t)+\bar q(t),\q r_1(t)=\Pi_1[r](t),\q r_2(t)=\Pi_2[r](t)+\bar r(t),\\
\ns\ds G_1=G,\q G_2=G+\bar G,\q g_1=\Pi_1[g],\q g_2=\Pi_2[g]+\bar g,\ea\ee
and
\bel{B-Q^Th}\ba{ll}
\ns\ds Q_i^{\Th_i}(t)=Q_i(t)+\Th_i(t)^\top S_i(t)+S_i(t)^\top\Th_i(t)+\Th_i(t)^\top R_i(t)\Th_i(t),\\
\ns\ds S_i^{\Th_i}(t)=S_i(t)+R_i(t)\Th_i(t),\qq q_i^{\Th_i}(t)=q_i(t)+\Th_i(t)^\top r_i(t).\ea\qq i=1,2.\ee
In particular, for the homogenous case, one has
\bel{B-J^0(Th)}\ba{ll}
\ns\ds J^0(s,\xi;\BTh(\cd),v(\cd))=\frac12\dbE\Big\{\int_s^T\[\lan Q_1^{\Th_1}X_1,X_1\ran+2\lan S_1^{\Th_1} X_1,v_1\ran+\lan R_1v_1,v_1\ran\\
\ns\ds\qq\qq\qq\qq\qq\qq+\lan Q_2^{\Th_2}X_2,X_2\ran+2\lan S_2^{\Th_2}X_2,v_2\ran+\lan R_2v_2,v_2\ran\]dt\\
\ns\ds\qq\qq\qq\qq\qq\qq+\lan G_1X_1(T),X_1(T)\ran+\lan G_2X_2(T),X_2(T)\ran\Big\}.\ea\ee
Note that by taking $\Th_1(\cd)=\Th_2(\cd)=0$, we recover the original state equation and cost functional (with $u_i(\cd)=v_i(\cd)$).

Next, let $(P_i(\cd),\L^M_i(\cd))$ be the predictable solution to the following BSDEs:
\bel{B-P_i}\left\{\2n\ba{ll}
\ds dP_i(t)=\G_i(t)dt+\L_i^M(t)dM(t),\\
\ns\ds P_i(T)=G_i,\ea\right.\ee
for some undetermined $\G_1(\cd),\G_2(\cd)\in L^2_{\dbF^M}(s,T;\dbS^n)$. Let $(\eta_1(\cd),\z_1(\cd),\z_1^M(\cd),\eta_2(\cd),\z_2^M(\cd))$ be the predictable solution to the following:
\bel{B-eta_1}\left\{\2n\ba{ll}
\ds d\eta_1(t)=\g_1(t)dt+\z_1(t)dW(t)+\z_1^M(t)dM(t),\\
\ns\ds d\eta_2(t)=\g_2(t)dt+\z_2^M(t)dM(t),\\
\ns\ds\eta_1(T)=g_1,\qq\eta_2(T)=g_2,\ea\right.\ee
for some undermined $\g_1(\cd)\in L^2_{\dbF^M}(s,T;\dbR^n)^\perp$, and $\g_2(\cd)\in L^2_{\dbF^M}(s,T;\dbR^n)$. Then, by It\^o's formula  (suppressing $t$),
$$\ba{ll}
\ns\ds d(P_1X_1)=\[\G_1X_1+P_1(A_1X_1+B_1u_1+b_1)\]dt+P_1(C_1X_1+C_2X_2+D_1u_1+D_2u_2+\si_1+\si_2)dW\\
\ns\ds\qq\qq\qq+\L_1^MX_1dM.\ea$$
Thus,
$$\ba{ll}
\ns\ds d\lan P_1X_1,X_1\ran=\[\lan\G_1X_1+P_1(A_1X_1+B_1u_1+b_1),X_1\ran+\lan P_1X_1,A_1X_1+B_1u_1+b_1\ran\\
\ns\ds\qq\qq+\lan P_1(C_1X_1+C_2X_2+D_1u_1+D_2u_2+\si_1+\si_2),C_1X_1+C_2X_2+D_1u_1+D_2u_2+\si_1+\si_2\ran\]dt\\
\ns\ds\qq\qq+\{\cds\}dW+\{\cds\}dM,\ea$$
where $\{\cds\}$ stands for something which are irrelevant below. Also,
$$\ba{ll}
\ns\ds d\lan\eta_1,X_1\ran=\[\lan\g_1,X_1\ran+\lan\eta_1,A_1X_1+B_1u_1+b_1\ran+\lan\z_1,
C_1X_1+C_2X_2+D_1u_1+D_2u_2+\si_1+\si_2\ran\]dt\\
\ns\ds\qq\qq+\{\cds\}dW+\{\cds\}dM.\ea$$
Consequently,
$$\ba{ll}
\ns\ds\dbE\[\lan G_1X_1(T),X_1(T)\ran+2\lan g_1,X_1(T)\ran\]=\dbE\[\lan P_1(T)X_1(T),X_1(T)\ran+2\lan\eta_1(T),X_1(T)\ran\]\\
\ns\ds=\dbE\Big\{\lan P_1(s)\xi_1,\xi_1\ran+2\lan\eta_1(s),\xi_1\ran+\int_s^T\[\lan(\G_1+P_1A_1+A_1^\top P_1+C_1^\top P_1C_1)X_1,X_1\ran\\
\ns\ds\qq+\lan C_2^\top P_1C_2X_2,X_2\ran+2\lan(B_1^\top P_1+D_1^\top P_1C_1)X_1,u_1\ran
+2\lan D_2^\top P_1C_2X_2,u_2\ran\\
\ns\ds\qq+\lan D_1^\top P_1D_1u_1,u_1\ran+\lan D_2^\top P_1D_2u_2,u_2\ran\\
\ns\ds\qq+2\lan P_1b_1+C_1^\top P_1\si_1+\g_1+A_1^\top\eta_1+C_1^\top\Pi_1[\z_1],X_1\ran+2\lan C_2^\top P_1\si_2+C_2^\top\Pi_2[\z_1],X_2\ran\\
\ns\ds\qq+2\lan D_1^\top P_1\si_1+B_1^\top\eta_1+D_1^\top\Pi_1[\z_1],u_1\ran+2\lan D_2^\top P_1\si_2+D_2^\top\Pi_2[\z_1],u_2\ran\\
\ns\ds\qq+\lan P_1\si_1,\si_1\ran+\lan P_1\si_2,\si_2\ran+2\lan\eta_1,b_1\ran+2\lan\Pi_1[\z_1],\si_1\ran+2\lan\Pi_2[\z_1],\si_2\ran\]dt
\Big\}.\ea$$
Likewise, which is a little different,
$$\ba{ll}
\ns\ds\dbE\[\lan G_2X_2(T),X_2(T)\ran+2\lan g_2,X_2(T)\ran\]=\dbE\[\lan P_2(T)X_2(T),X_2(T)\ran+2\lan\eta_2(T),X_2(T)\ran\]\\
\ns\ds=\dbE\Big\{\lan P_2(s)\xi_2,\xi_2\ran+2\lan\eta_2(s),\xi_2\ran+\int_s^T\[\lan(\G_2+P_2A_2+A_2^\top P_2)X_2,X_2\ran\\
\ns\ds\qq+2\lan B_2^\top P_2X_2,u_2\ran+2\lan P_2b_2+\g_2+A_2^\top\eta_2,X_2\ran+2\lan B_2^\top\eta_2,u_2\ran+2\lan\eta_2,b_2\ran\]dt.\ea$$
Hence,
\bel{B-cost3}\ba{ll}
\ns\ds J(s,\xi_1,\xi_2;u_1(\cd),u_2(\cd))\equiv J(s,\xi;u(\cd))\\
\ns\ds=\frac12\sum_{i=1}^2\dbE\[\int_s^T\(\lan Q_iX_i,X_i\ran\! +\!2\lan S_iX_i, u_i\ran\!+\!\lan R_iu_i,u_i\ran+2\lan q_i,X_i\ran+2\lan r_i,u_i\ran\)dt\\
\ns\ds\qq\qq\qq\qq+\lan G_iX_i(T), X_i(T)\ran+ 2\lan g_i,X_i(T)\ran\]\\
\ns\ds=\frac12\sum_{i=1}^2\dbE\Big\{\lan P_i(s)\xi_i,\xi_i\ran+2\lan\eta_i(s),\xi_i\ran\\
\ns\ds\qq+\int_s^T\3n\[\lan(\G_i\1n+\1n P_iA_i\1n+\1n A_i^\top P_i\1n+\1n C_i^\top P_1C_i\1n+\1n Q_i)X_i,X_i\ran\1n+\1n2\lan(B_i^\top P_i\1n+\1n D_i^\top P_1C_i\1n+\1n S_i)X_i,u_i\ran\\
\ns\ds\qq+\lan(R_i+D_i^\top P_1D_i)u_i,u_i\ran+2\lan\g_i+A_i^\top\eta_i+C_i^\top\Pi_i[\z_1]+P_ib_i+C_i^\top P_1\si_i+q_i,X_i\ran\\
\ns\ds\qq+2\lan B_i^\top\eta_i+D_i^\top\Pi_i[\z_1]+D_i^\top P_1\si_i+r_i,u_i\ran
\\
\ns\ds\qq+2\lan\eta_i,b_i\ran+2\lan\Pi_i[\z_1],\si_i\ran
+\lan P_1\si_i,\si_i\ran\]dt\Big\}\\
\ns\ds=\frac12\sum_{i=1}^2\dbE\Big\{\lan P_i(s)\xi_i,\xi_i\ran+2\lan\eta_i(s),\xi_i\ran\\
\ns\ds\qq+\int_s^T\[\lan[\G_i+\cQ_i(P_i,P_1)]X_i,X_i\ran+2\lan\cS_i(P_i,P_1)X_i+\Br_i,u_i\ran
+\lan\cR_i(P_1)u_i,u_i\ran\\
\ns\ds\qq\qq\qq+2\lan\g_i+\Bq_i,X_i\ran+2\lan\eta_i,b_i\ran
+2\lan\Pi_i[\z_1],\si_i\ran+\lan P_1\si_i,\si_i\ran\]dt\Big\},\ea\ee
where
\bel{B-Bq}\left\{\2n\ba{ll}
\ns\ds\Bq_i=A_i^\top\eta_i+C_i^\top\Pi_i[\z_1]+P_ib_i+C_i^\top P_1\si_i+q_i,\\
\ns\ds\Br_i=B_i^\top\eta_i+D_i^\top\Pi_i[\z_1]+D_i^\top P_1\si_i+r_i.\ea\right.\qq i=1,2.\ee
Now, for any $\BTh(\cd)=(\Th_1(\cd).\Th_2(\cd))$, by the same calculation with $(A_i,C_i,Q_i,S_i)$ replaced by $(A_i^{\Th_i},C_i^{\Th_i},Q_i^{\Th_i},S_i^{\Th_i})$, and $u_i(\cd)$ replaced by $v_i(\cd)$, only for the homogeneous case (thus, $X_i$ becomes $X_i^0$), one has
\bel{cost1*a}\ba{ll}
\ns\ds J^0(s,\xi;\BTh(\cd),v(\cd))\\
\ns\ds=\frac12\sum_{i=1}^2\dbE\Big\{\lan P_i(s)\xi_i,\xi_i\ran+\2n\int_s^T\3n\[\lan(\G_i+P_iA_i^{\Th_i}+(A_i^{\Th_i})^\top P_i\1n+\1n(C_i^{\Th_i})^\top P_1C_i^{\Th_i}\1n+\1n Q_i^{\Th_i})X_i^0,X_i^0\ran\\
\ns\ds\qq+2\lan(B_i^\top P_i+D_i^\top P_1C_i^{\Th_i}+S_i^{\Th_i})X_i^0,v_i\ran+\lan(R_i+D_i^TP_1D_i)v_i,v_i\ran\]dt\Big\}.\ea\ee
Thus, if we choose
\bel{B-G=}\G_i=-\(P_iA_i^{\Th_i}+(A_i^{\Th_i})^\top P_i\1n+\1n(C_i^{\Th_i})^\top P_1C_i^{\Th_i}\1n+\1n Q_i^{\Th_i}\),\qq i=1,2,\ee
i.e., equation \rf{B-P_i} reads
\bel{B-P_i*}\left\{\2n\ba{ll}
\ds dP_i(t)=-\(P_iA_i^{\Th_i}+(A_i^{\Th_i})^\top P_i\1n+\1n(C_i^{\Th_i})^\top P_1C_i^{\Th_i}\1n+\1n Q_i^{\Th_i}\)dt+\L_i^M(t)dM(t),\\
\ns\ds P_i(T)=G_i,\ea\right.\ee
then
\bel{cost1*d}\ba{ll}
\ns\ds J^0(s,\xi;\BTh(\cd),v(\cd))\\
\ns\ds=\frac12\sum_{i=1}^2\dbE\Big\{\lan P_i(s)\xi_i,\xi_i\ran+2\int_s^T\lan(B_i^\top P_i+D_i^\top P_1C_i^{\Th_i}+S_i^{\Th_i})X_i^0,v_i\ran+\lan(R_i+D_i^\top P_1D_i)v_i,v_i\ran\]dt\Big\}\\
\ns\ds=\frac12\sum_{i=1}^2\dbE\Big\{\lan P_i(s)\xi_i,\xi_i\ran+2\int_s^T\lan[\cS_i(P_i)+\cR_i(P_1)\Th_i]X_i^0,v_i\ran+\lan\cR_i(P_1)v_i,v_i\ran\]dt\Big\}.\ea\ee
Further, if we can further achieve
\bel{B-BP}0=B_i^\top P_i+D_i^\top P_1C_i^{\Th_i}+S_i^{\Th_i}=\cS_i(P_1)+\cR_i(P_1)\Th_i,\qq i=1,2,\ee
then
\bel{B-J^0}J^0(s,\xi;\BTh(\cd),v(\cd))=\frac12\sum_{i=1}^2\dbE\Big\{\lan P_i(s)\xi_i,\xi_i\ran+\int_s^T\lan\cR_i(P_1)v_i,v_i\ran dt\Big\}.\ee
We see that \rf{B-BP} means
$$\ba{ll}
\ns\ds\sR\big(\cS_i(P_i)\big)\subseteq\sR\big(R_i(P_1)\big),\\
\ns\ds\Th_i=-\cR_i(P_1)^\dag\cS_i(P_i)+\big[I-\cR_i(P_1)^\dag\cR_i(P_1)\big]\Th_{0i},
\qq i=1,2.\ea$$
for some $\Th_{0i}$. This then implies
$$\ba{ll}
\ns\ds P_iA_i^{\Th_i}+(A_i^{\Th_i})^\top P_i\1n+\1n(C_i^{\Th_i})^\top P_1C_i^{\Th_i}\1n+\1n Q_i^{\Th_i}\\
\ns\ds=P_i(A_i+B_i\Th_i)+(A_i+B_i\Th_i)^\top P_i\1n+\1n(C_i+D_i\Th_i)^\top P_1(C_i+D_i\Th_i)\1n+\1n Q_i+\Th_i^\top S_i+S_i^\top\Th_i+\Th_i^\top R_i\Th_i\\
\ns\ds=P_iA_i\1n+\1n A_i^\top P_i\1n+\1n C_i^\top P_1C_i\1n+\1n Q_i\1n+\1n(P_iB_i\1n+\1n C_iP_1D_i\1n+\1n S_i^\top)\Th_i\1n+\1n\Th_i^\top(B_i^\top P_i\1n+\1n D_i^\top P_1C_i\1n+\1nS_i)\1n+\1n\Th_i^\top(R_i\1n+\1n D_i^\top P_1D_i)\Th_i\\
\ns\ds=\cQ_i(P_i)+\cS_i(P_i)^\top\Th_1+\Th_i^\top\cS_i(P_i)+\Th_i\cR_i(P_1)\Th_i
=\cQ_i(P_i)-\cS_i(P_i)\cR_i(P_1)^\dag\cS_i(P_i).\ea$$
Hence, \rf{B-P_i*} further becomes BSDRE \rf{Riccati1}. This means if $(P_i(\cd),\L_i^M(\cd))$ is a predictable solution of \rf{Riccati1} (with the range condition), then \rf{B-J^0} holds.


\begin{thebibliography}{99}	
	
	
	\bibitem{Anderson-Moore-1971} B.~O.~D~Anderson and J.~B.~Moore, \sl Linear Optimal
	Control, \rm Prentice-Hall Englewood Cliffs, New Jersey, 1971.
	
	\bibitem{Bellman-Glicksberg-Gross-1958} R.~Bellman, I.~Glicksberg, and O.~Gross, \sl Some Aspects of the Mathematical Theory of Control Processes, \rm Rand Corporation, Santa Monica, California, 1958.
	
	\bibitem{Bellman-Kalaba-Wing-1960} R.~Bellman, R.~Kalaba, and G.~M.~Wing, \it Invariant imbedding and the reduction of two-point boundary value problems to initial value problems, \sl Proc. Nat. Acad. Sci. U.S.A., \rm 46 (1960), 1646--1649.
	
	\bibitem{Bensoussan-1981} A.~Bensoussan, \sl Lecture on Stochastic Control, \rm Lecture Notes in Math. Vol. 972, Springer-Verlag, Berlin, 1981.
	
	\bibitem{Buckdahn-Peng-1999} R.~Buckdahn and S.~Peng, \it Stationary backward stochastic differential equations and associated partial differential equations. \sl Probability Theory and Related Fields, \rm 115 (1999), 383--399.
	
	\bibitem{Davis-1977} M.~H.~A.~Davis, \sl Linear Estimation and Stochastic Control, \rm Chapman \& Hall, London, 1977.
	
	\bibitem{Carbone-2008} R.~Carbone, B.~Ferrario, and M.~Santacroce, \it Backward stochastic differential equations driven by c\`{a}dl\`{a}g martingales. \sl Theory of Probability \& its Applications, \rm 52 (2008), 304--314.
	
	\bibitem{Chen-Li-Zhou-1998} S.~Chen, X.~Li and X.~Y.~Zhou, \it Stochastic linear quadratic regulators with indefinite weight costs, \sl SIAM J. Control Optim., \rm 36 (1998), 1695--1702.
	
	%
	
	\bibitem{Franz-2002} U.~Franz, \it What is stochastic independence? \sl Non-commutativity, Infinite-dimensionality and Probability at the Crossroads, \rm 254--274, \sl QP-PQ: Quantum Probab. White Noise Anal., \rm 16, World Sci. Publ., River Edge, NJ, 2002.
	
	
	\bibitem{Huang-Li-Yong-2015} J.~Huang, X.~Li, and J.~Yong, \it A linear-quadratic optimal control problem for mean-field stochastic differential equations in infinite horizon. \sl Math. Control Relat. Fields, \rm 5 (2015), 97--139.
	
	\bibitem{Kalman-1960} R.~E.~Kalman, \it Contributions to the theory of optimal control, \sl Bol. Soc. Mat. Mexicana, \rm 5 (1960), 102--119.
	
	\bibitem{Kushner-1962} H.~J.~Kushner, \it Optimal stochastic control, IRE Trans. Auto. Control, \rm 7 (1962), 120--122.
	
	\bibitem{Lee-Markus-1967} E.~B.~Lee and L.~Markus, \sl Foundations of Optimal Control Theory, \rm John Wiley, New York, 1967.
	
	\bibitem{Letov-1960} A.~M.~Letov, \it Analytic design of regulators, \sl Avtomat. i Telemekh., \rm 21 (1960), 436--441; 561--568; 661--665; 22 (1961), 425--435; 23 (1962), 1405--1413 (in Russian).
	
	\bibitem{Li-Shi-Yong-2021} X.~Li, J.~Shi, and J.~Yong, \it Mean-field linear-quadratic stochastic differential games in an infinite horizon, \sl ESAIM Control Optim. Calc. Var., \rm 27 (2021), 81.
	
	\bibitem{Li-Sun-Yong-2016} X.~Li, J.~Sun, and J.~Yong, \it Mean-field stochastic linear quadratic optimal control problems: closed-loop solbability, \sl Probability, Uncertainty annd Quantitative Risk, \rm 1 (2016), 2.
	
	\bibitem{Ma-Protter-Yong-1994} J.~Ma, P.~Protter, and J.~Yong, \it Solving forward-backward stochastic differential equations explicitly --- a four-step scheme, \sl Probability Theory \& Related Fields, \rm 98 (1994), 339--359.
	
	\bibitem{Ma-Yong-1999} J.~Ma and J.~Yong, \sl Forward-Backward Stochastic Differential Equations and Their Applications, \rm Lecture Notes vol. 1702, Springer-Verlag, 1999.
	
	\bibitem{McLane-1971} P.~J.~McLane, \it Optimal stochastic control of linear systems with state- and control-dependent disturbances, \sl IEEE Trans. Auto. Control, \rm 16 (1971), 793--798.		
	
	\bibitem{Mei-Wei-Yong-2021} H.~Mei, Q.~Wei, and J.~Yong, \it Optimal ergodic control of linear stochastic differential equations with quadratic cost functionals having indefinite weights, \sl SIAM J. Control Optim., \rm 59 (2021), 584--613.
	
	\bibitem{Mou-Yong-2006} L.~Mou and J.~Yong, \it Two-person zero-sum liear quadratic stochastic differentia gmesby a Hilbert space method, \sl J. Indust. Manag. Optim., \rm 2 (2006), 95--117.
	
	
	\bibitem{Sun-2017} J.~Sun, \it Mean-field stochastic linear quadratic optimal control problems: open-loop solvabilities, \sl ESAIM Control Optim. Calc. Var., \rm 23 (2017), 1099--1127.
	
	\bibitem{Sun-2016} J. Sun,  X. Li, and J. Yong. \it  Open-loop and closed-loop solvabilities for stochastic linear quadratic optimal control problems. \sl{SIAM J. Control Optim.}, \rm 54 (2016), 2274--2308.	
	
	
	\bibitem{Sun-2021} J.~Sun and H.~Wang, \it Mean-field stochastic linear-quadratic optimal control problems: weak closed-loop solvability, \sl Math. Control Relat. Fields, \rm 11 (2021), 47--71.
	
	\bibitem{Sun-Li-Yong-2016} J.~Sun, X.~Li, and J.~Yong, \it Open-loop and closed-loop solvabilities for stochastic linear quadratic optimal control problems, \sl SIAM J. Control Optim., \rm 54 (2016), 2274--2308.
	
	
	\bibitem{Sun-Xiong-Yong-2021} J.~Sun, J.~Xiong, and J.~Yong, \it Indefinite stochastic linear-quadratic optimal control problems with random coefficients: Closed-loop representation of open-loop optimal controls. \sl Ann. Appl. Probab., \rm 31 (2021), 460--499.
	
	\bibitem{Sun-Yong-2020a} J.~Sun and J.~Yong, \sl Stochastic Linear-Quadratic Optimal Control Theory: Open-Loop and Closed-Loop Solutions, \rm Springer Brief in Math., Springer, 2020.
	
	\bibitem{Sun-Yong-2020b} J.~Sun and J.~Yong, \sl Stochastic Linear-Quadratic Optimal Control Theory: Differential Games and Mean-Field Problems, \rm Springer Brief in Math., Springer, 2020.
	
	\bibitem{Wei-Yong-Yu-2019} Q.~Wei, J.~Yong, and Z.~Yu, \it Linear quadratic optimal control problems with operator coefficients: open-loop solutions, \sl  ESAIM Control Optim. Calc. Var., \rm 25 (2019), 17.
	
	\bibitem{Wen-Li-2022 } J.~Wen, X.~Li, J,~Xiong, and X.~Zhang, \it Stochastic linear quadratic optimal control problems with random coefficients and Markovian regime switching system, \sl SIAM J. Control Optim., \rm 61  (2023), 949-979.
	
	
	\bibitem{Willems-1971} J.~C.~Willems, \it Least squares stationary optimal control and the algebraic Riccati equation, \sl IEEE Trans. Automat. Control, \rm 16 (1971), 621--634.
	
	\bibitem{Wonham-1968} W.~M.~Wonham, \it On a matrix Riccati equation of stochastic control, \sl SIAM J. Control, \rm 6 (1968), 681--697.
	
	\bibitem{Wonham-1979} W.~M.~Wonham, \sl Linear Multivariable Control: A Geometric
	Approach, 2nd Edition, \rm Springer-Verlag, 1979.
	
	\bibitem{Yin-Mao-2008} J.~Yin and X.~Mao, \it The adapted solution and comparison theorem for backward stochastic differential equations with Poisson jumps and applications. \sl J. Math. Anal. Appl., \rm 346  (2008), 345--358.
	
	\bibitem{Yong-2013} J.~Yong, \it Linear-quadratic optimal control problems for mean-field stochastic differential equations, \sl SIAM J. Control Optim., \rm 51(2013), 2809--2838.
	
	\bibitem{Yong-Zhou-1999} J.~Yong and X.~Y.~Zhou, \sl Stochastic Controls: Hamiltonian Systems and HJB Equations, \rm Springer-Verlag, 1999.
	
	\bibitem{Zhang-Li-2018} X.~Zhang and X.~Li, \it Open-loop and closed-loop solvabilities for stochastic linear quadratic optimal control problems of Markov regime-switching system, \rm ESAIM Control Optim. Calc. Var., \rm 27 (2021), 69.
	
\end{thebibliography}
\end{document}